\documentclass[11pt,reqno]{amsart}

    \usepackage[utf8]{inputenc}
\usepackage[left=3.5cm,right=3.5cm,top=3cm,bottom=3cm]{geometry}
    \usepackage{graphicx}
    \usepackage[normalem]{ulem}
    \usepackage{amsmath, amsfonts, amssymb, amsthm}
    \usepackage{mathtools}
    \usepackage[mathscr]{eucal}
    \usepackage[colorlinks=true,citecolor=blue]{hyperref}
    \usepackage{tikz}
    \usepackage{tikz-cd}
    \usepackage{xparse}
    \usepackage{ifthen}
    \usepackage{xifthen}
    \usepackage{cleveref}

    \theoremstyle{plain}
    \newtheorem{prop}{Proposition}[section]
    \newtheorem{thm}[prop]{Theorem}
    \newtheorem{lemma}[prop]{Lemma}
    \newtheorem{cor}[prop]{Corollary}
    \theoremstyle{remark}
    
    \theoremstyle{definition}
    \newtheorem{rmrk}[prop]{Remark}
    \newtheorem{example}[prop]{Example}
    \newtheorem{mydef}[prop]{Definition}

    \numberwithin{equation}{section}
    \newcommand{\ipic}[3]{\raisebox{#1\height}{\scalebox{#3}{\includegraphics{./#2}}}}
    \allowdisplaybreaks
    \newcommand{\unnumberedsubsection}[1]{\medskip \paragraph{\bf #1}}
    \Crefname{thm}{Theorem}{Theorems}

    \newcommand{\defined}{\vcentcolon=}
    \newcommand{\inv}{^{-1}}
    \newcommand{\tensor}{\otimes}
    \DeclareMathOperator{\id}{id}
    \newcommand{\placeholder}{?}
    \NewDocumentCommand{\field}{ O{} }{
        \ifthenelse{ \equal{#1}{} }{\ensuremath{k}}{\ensuremath{\mathbb{#1}}}
    }
    \newcommand{\oversetEq}[1][ ]{
        \stackrel{\mathmakebox[\widthof{=}]{#1}}{=}
    }

    \newcommand{\cat}[1][C]{{\mathcal{#1}}}
    \newcommand{\pivotalStruct}{\delta}
    \newcommand{\tensUnit}{\mathbf{1}}
    \newcommand{\distInvObj}{D}
    \newcommand{\counitAdj}{\widetilde{\varepsilon}}
    \newcommand{\unitAdj}{\widetilde{\eta}}
    \newcommand{\To}{\Rightarrow}

    \DeclareMathOperator{\End}{End}
    \newcommand{\dualsymbol}{\vee}
    \newcommand{\dualsymbolR}{\vee}
    \newcommand{\dualL}[1]{{#1^{\dualsymbol}}}
    \newcommand{\ddualL}[1]{{#1^{\dualsymbol\dualsymbol}}}
    \NewDocumentCommand{\dualR}{ m O{-0.1} }{
            {\prescript{\dualsymbolR}{}{\hspace*{#2em}#1}}
        }
    
    \DeclareMathOperator{\ev}{ev}
    \DeclareMathOperator{\coev}{coev}
    \DeclareMathOperator{\evL}{\overset{\xleftarrow{\phantom{ev}}}{\ev}}
    \DeclareMathOperator{\coevL}{\overset{\xleftarrow{\phantom{coev}}}{\coev}}
    \DeclareMathOperator{\evR}{\overset{\xrightarrow{\phantom{ev}}}{\ev}}
    \DeclareMathOperator{\coevR}{\overset{\xrightarrow{\phantom{coev}}}{\coev}}
    \newcommand{\Vect}{\mathsf{Vect}}
    \NewDocumentCommand{\hmodM}{O{H}O{}O{}O{}}{\prescript{#4}{#1}{\mathcal{M}}_{#2}^{#3}}

    \newcommand{\hmod}[1][H]{#1\text{-mod}}
    \newcommand{\sweedler}[1]{_{(#1)}}
    \newcommand{\op}{^{\text{op}}}
    \newcommand{\cop}{^{\text{cop}}}
    
    \newcommand{\elQ}[1]{\mathsf{#1}}
    \newcommand{\elQbold}[1]{\boldsymbol{#1}}
    \newcommand{\Dt}{\elQbold{f}}
    \newcommand{\modulus}{\boldsymbol{\gamma}}
    \newcommand{\counit}{\varepsilon}
    \newcommand{\coint}{\boldsymbol{\lambda}}
    \newcommand{\oneQ}{\elQbold{1}}
    \newcommand{\pivotQ}{\elQbold{g}}
    \newcommand{\alphaQ}{\elQbold{\alpha}}
    \newcommand{\betaQ}{\elQbold{\beta}}
    \newcommand{\elDelta}{\elQbold{\delta}}
    \newcommand{\pR}{p^R}
    \newcommand{\qR}{q^R}
    \newcommand{\pL}{p^L}
    \newcommand{\qL}{q^L}
    \newcommand{\elU}{\elQ{U}}
    \newcommand{\elV}{\elQ{V}}
    \newcommand{\elu}{\textsf{u}}
    \newcounter{alphabet}
    \expandafter\newcommand\csname alphList\the\value{alphabet} \endcsname{x}
    \stepcounter{alphabet}
    \expandafter\newcommand\csname alphList\the\value{alphabet} \endcsname{y}
    \stepcounter{alphabet}
    \expandafter\newcommand\csname alphList\the\value{alphabet} \endcsname{z}
    \NewDocumentCommand{\coassQ}{ O{} }
        {\ifthenelse{\equal{#1}{}}
            {\Phi}
            {\MakeUppercase{\csname alphList#1 \endcsname}}
        }
    \NewDocumentCommand{\invCoassQ}{ O{} }
        {\ifthenelse{\equal{#1}{}}
            {\Psi}
            {\csname alphList#1 \endcsname}
        }
    \NewDocumentCommand{\dinatEnd}{ O{i} O{} O{} }{
    \ifthenelse{ \equal{#3}{} }{
            \ifthenelse{ \equal{#2}{} }{\pi_{#1}}{\pi_{#1}(#2)} 
        }{ \pi_{#1}(#2)_{#3} }
        }
    \NewDocumentCommand{\comultEnd}{ O{i} O{} O{} }{
        \ifthenelse{ \equal{#3}{} }{ \Delta^{#1}({#2}) }{ \left( \Delta^{#1}({#2})
        \right)_{#3} }
        }
    \NewDocumentCommand{\counitEnd}{ O{i} O{} O{} }{
        \ifthenelse{ \equal{#3}{} }{ \varepsilon^{#1}({#2}) }{ \left(
        \varepsilon^{#1}({#2}) \right)_{#3} }
        }
    \NewDocumentCommand{\multEnd}{ O{i} O{} O{} }{
        \ifthenelse{ \equal{#3}{} }{ \mu^{#1}({#2}) }{ \left( \mu^{#1}({#2}) \right)_{#3} }
        }
    \NewDocumentCommand{\unitEnd}{ O{i} }{ \ensuremath{ u^{#1} } }
    \NewDocumentCommand{\hopfComonad}{ O{i} O{} }{
        \ifthenelse{ \equal{#2}{} }{\ensuremath{ Z^{#1} }}{\ensuremath{ Z^{#1}\left( #2
        \right) }}
        }
    \NewDocumentCommand{\hopfComonadComp}{ m O{X} O{V} }{
        \ifthenelse{ #1 = 1 }{ \ensuremath{ \dualR{#2} \left( #3 #2 \right) } }{} 
        \ifthenelse{ #1 = 2 }{ \ensuremath{ \dualL{#2} \left( #3 #2 \right) } }{} 
        \ifthenelse{ #1 = 3 }{ \ensuremath{ \left( #2 #3 \right) \dualR{#2} } }{} 
        \ifthenelse{ #1 = 4 }{ \ensuremath{ \left( #2 #3 \right) \dualL{#2} } }{} 
        }
    \NewDocumentCommand{\dinatCoend}{ O{i} O{} O{} }{
        \ifthenelse{ \equal{#3}{} }{ 
            \ifthenelse{ \equal{#2}{} }{\iota_{#1}}{\iota_{#1}(#2)} 
        }{ \iota_{#1}(#2)_{#3} }
        }
    \NewDocumentCommand{\multCoend}{ O{i} O{} }{
        \ifthenelse{ \equal{#2}{} }{ \mu_{#1} }{ \mu_{#1}({#2}) }
        }
    \NewDocumentCommand{\unitCoend}{ O{i} O{} }{
        \ifthenelse{ \equal{#2}{} }{ \eta_{#1} }{ 
            \eta_{#1}({#2}) }
        }
    \NewDocumentCommand{\comultCoend}{ O{i} O{} }{
        \ifthenelse{ \equal{#2}{} }{ \Delta_{#1} }{ \Delta_{#1}({#2}) }
        }
    \NewDocumentCommand{\counitCoend}{ O{i} }{ \ensuremath{ \epsilon_{#1} } }
    \NewDocumentCommand{\hopfMonad}{ O{i} O{} }{
        \ifthenelse{ \equal{#2}{} }{\ensuremath{ A_{#1} }}{\ensuremath{ A_{#1}\left( #2
        \right) }}
        }
    \NewDocumentCommand{\hopfMonadComp}{ m O{X} O{V} }{
        \ifthenelse{ #1 = 1 }{ \ensuremath{ \dualR{#2} \left( #3 #2 \right) } }{} 
        \ifthenelse{ #1 = 2 }{ \ensuremath{ \dualL{#2} \left( #3 #2 \right) } }{} 
        \ifthenelse{ #1 = 3 }{ \ensuremath{ \left( #2 #3 \right) \dualR{#2} } }{} 
        \ifthenelse{ #1 = 4 }{ \ensuremath{ \left( #2 #3 \right) \dualL{#2} } }{} 
        }
    \NewDocumentCommand{\actionEnd}{ O{i} O{} O{} }{
        \ifthenelse{ \equal{#2}{} }
            {\ensuremath{\overset{#1}{\cdot}}}
            {\ensuremath{#2\overset{#1}{\cdot}(#3)}}
        }
    \NewDocumentCommand{\leftAntCoend}{ O{i} O{} }{
        \ifthenelse{ \equal{#2}{} }
            {S^l_{#1}}
            {S^l_{#1}(#2)}
        }
    \NewDocumentCommand{\actionCoend}{ O{i} O{} O{} }{
        \ifthenelse{ \equal{#2}{} }
            {\ensuremath{\overset{#1}{\curvearrowright}}}
            {\ensuremath{#2\overset{#1}{\curvearrowright}(#3)}}
    }

    \newcommand{\braiding}{c}
    \newcommand{\coend}{\mathcal{L}}
    \newcommand{\intLyu}{\Lambda}
    \newcommand{\objectOfInts}[1][A]{\operatorname{Int}#1}
    \newcommand{\rMatrix}{R}
    \newcommand{\rMatrixInv}{\overline{R}}
    \newcommand{\slTwoZ}{SL(2,\mathbb{Z})}
    \newcommand{\elRibbon}{\boldsymbol{v}}
    \newcommand{\elOmega}{\widehat{\omega}}
    \newcommand{\modS}{\mathscr{S}}
    \newcommand{\modT}{\mathscr{T}}
    \newcommand{\hopfPair}{\omega}
    \newcommand{\ribTwist}{\vartheta}
    \newcommand{\dinatLyu}{j}

        \newcommand{\elTh}{\elQbold{\vartheta}}
        \newcommand{\symRightCoint}{\widehat{\coint}^r}
        \newcommand{\symLeftCoint}{\widehat{\coint}^l}
        \newcommand{\Ad}{\mathcal{A}}
        \newcommand{\Adtwist}{\Ad}
        \newcommand{\elX}{\boldsymbol{\xi}}
        \newcommand{\cointCat}{\coint^{\cat}}

        \newcommand{\eQ}{\elQbold{e}}
        \newcommand{\spaceLeftCoint}{\begingroup \textstyle \int^l_H \endgroup}
        \newcommand{\spaceRightCoint}{\begingroup \textstyle \int^r_H \endgroup}
        \newcommand{\spaceLeftSymCoint}
            {\begingroup \textstyle \int^{l,\modulus}_H \endgroup}
        \newcommand{\spaceRightSymCoint}
            {\begingroup \textstyle \int^{r,\modulus}_H \endgroup}
        \newcommand{\spaceRightMonCoint}{\begingroup \textstyle
            \int^{r,\textup{mon}}_{\cat} \endgroup}
        \newcommand{\spaceLeftMonCoint}{\begingroup \textstyle
            \int^{l,\textup{mon}}_{\cat} \endgroup}
        \newcommand{\spaceRightSymMonCoint}{\begingroup \textstyle
            \int^{r,\distInvObj-\textup{sym}}_{\cat} \endgroup}
        \newcommand{\spaceLeftSymMonCoint}{\begingroup \textstyle
            \int^{l,\distInvObj-\textup{sym}}_{\cat} \endgroup}
        \newcommand{\gSSSymbol}{\mathcal{X}}
        \newcommand{\gammaSSymRS}{\gSSSymbol_1}
        \newcommand{\gammaSSymR}{\gSSSymbol_2}
        \newcommand{\gammaSSym}{\gSSSymbol_2}
        \newcommand{\gammaSSymL}{\gSSSymbol_3}
        \newcommand{\gammaSSymLS}{\gSSSymbol_4}
        \newcommand{\isoXC}{\mathbb{A}}
        \newcommand{\monCointRight}[1]{{#1}^{\textup{mon}}}
        \newcommand{\monCointLeft}[1]{{}^{\textup{mon}}{#1}}
        \newcommand{\monCointRightSym}[1]{{#1}^{\text{$\modulus$-sym}}}
        \newcommand{\monCointLeftSym}[1]{{}^{\textup{$\modulus$-sym}}#1}
        \newcommand{\sympFerm}{\mathsf{Q}}

\title{Monadic cointegrals and\\
	applications to quasi-Hopf algebras}

    \usepackage{fancyhdr}
    \newcommand\asteriskfootnote[1]{
        \begingroup
        \renewcommand\thefootnote{}\footnote{#1}
        \addtocounter{footnote}{-1}
        \endgroup
    }

\begin{document}
\setcounter{tocdepth}{1}
    \maketitle
    \thispagestyle{empty}

\begin{center} 
    Johannes Berger\,$^{a,*}$, Azat M.\ Gainutdinov\,$^{b}$~~and~~Ingo
    Runkel\,$^a$
    \\[1.5em]
	{\sl\small $^a$ Fachbereich Mathematik, Universit\"at Hamburg\\
	Bundesstra\ss e 55, 20146 Hamburg, Germany}
\\[0.5em]
    {\sl\small $^b$ Institut Denis Poisson, CNRS, Universit\'e de Tours, Universit\'e
    d'Orl\'eans,\\ Parc de Grandmont, 37200 Tours, France}
\end{center}
\asteriskfootnote{\raggedright 
    * Corresponding Author.

    Emails: 
        {\tt johannes.berger@uni-hamburg.de},
        {\tt azat.gainutdinov@lmpt.univ-tours.fr}, 
        {\tt ingo.runkel@uni-hamburg.de}
}

\begin{abstract}
For $\cat$ a finite tensor category we consider four versions  of the central monad,
$\hopfMonad[1], \dots, \hopfMonad[4]$ on $\cat$.
Two of them are Hopf monads, and for $\cat$ pivotal, so are the remaining two.  
In that case all $\hopfMonad[i]$ are isomorphic as Hopf monads.
We define a {\em monadic cointegral for $\hopfMonad[i]$} to be an $\hopfMonad[i]$-module
morphism $\tensUnit \to \hopfMonad[i][D]$, where $D$ is the distinguished invertible
object of $\cat$.

We relate monadic cointegrals to the categorical cointegral introduced 
by Shimizu (2019),
and, in case $\cat$ is braided, to an integral for the braided Hopf algebra
$\coend = \int^X \dualL{X} \otimes X$ in $\cat$ studied by Lyubashenko (1995).

Our main motivation stems from the application to finite dimensional quasi-Hopf algebras
$H$.
For the category of finite-dimensional $H$-modules, we relate the four monadic cointegrals
(two of which require $H$ to be pivotal) to four existing notions of cointegrals for
quasi-Hopf algebras:
the usual left/right cointegrals of Hausser and Nill (1994), as well as so-called
$\modulus$-symmetrised cointegrals in the pivotal case, for $\modulus$ the modulus of $H$.

For (not necessarily semisimple) modular tensor categories $\cat$, Lyubashenko gave
actions of surface mapping class groups on certain Hom-spaces of $\cat$, in particular of
$SL(2,\mathbb{Z})$ on $\cat(\coend,\tensUnit)$.
In the case of a factorisable ribbon quasi-Hopf algebra, we give a simple expression for
the action of $S$ and $T$ which uses the monadic cointegral.
\end{abstract}



\tableofcontents

\section{Introduction}\label{SEC:Introduction}

\subsection{Monadic cointegrals}

Let $H$ be a finite-di\-men\-sio\-nal Hopf algebra over an algebraically closed field
$\field$.
There are two notions of cointegrals for $H$, namely left cointegrals and right
cointegrals.
These are elements $\coint \in H^*$ which satisfy, for all $h \in H$,
\begin{align*}
&2) \qquad
(\coint \otimes \id) \circ \Delta(h) = \coint(h) 1
&& \text{(right cointegral)} \ ,
\\
&3)\qquad
(\id \otimes \coint) \circ \Delta(h) = \coint(h) 1
&& \text{(left cointegral)} \ .
\end{align*}
The unusual numbering will be explained below.

If $H$ is a pivotal Hopf algebra, that is, if it is equipped with a group-like
element~$\pivotQ$ expressing the square of the antipode via conjugation, one can introduce
two more notions of cointegrals, so-called $\modulus$-symmetrised left and right
cointegrals \cite{BBGa,FOG}.
Here $\modulus \in H^*$ denotes the modulus of $H$, which encodes the difference between
left and right integrals for $H$.
The defining equation for $\coint \in H^*$ to be a $\modulus$-symmetrised left/right
cointegral is
\begin{align*}
    &1) \qquad
    (\coint \otimes \id) \circ \Delta(h) = \coint(h) \pivotQ^{-1}
    && \text{(right $\modulus$-symmetrised cointegral)} \ ,
    \\
    &4) \qquad
    (\id \otimes \coint) \circ \Delta(h) = \coint(h) \pivotQ
    && \text{(left $\modulus$-symmetrised cointegral)} \ .
\end{align*}
Such $\modulus$-symmetrised cointegrals, and in particular the special case where
$\modulus$ is given by the counit, have applications to modified traces and to quantum
invariants of links and three-manifolds~\cite{BBGa, BBGe-LogHenning, DeRenzi-GP}.
Furthermore, $\modulus$-symmetrised cointegrals are an example of $g$-cointegrals for a
group-like $g$ as introduced in~\cite{Radford-integrals}.

The category $\hmodM$ of finite-dimensional left $H$-modules is a finite tensor
category in the sense of \cite{EGNO}, in particular it has left and right duals.
If $H$ is in addition pivotal, then $\hmodM$ becomes a pivotal tensor category
\cite{AAGTV}.
One can now ask if it is possible to describe the above two (or four, in the pivotal case)
notions of cointegrals in terms of the category $\hmodM$ in a such a way that it
generalises to arbitrary (pivotal) finite tensor categories $\cat$.
This is indeed possible, and has been done for left cointegrals in \cite{Sh-integrals}
using Hopf comonads on $\cat$.
In the present paper we define analogous notions for all four variants of cointegrals,
working instead with Hopf monads.

\medskip

Hopf monads on a rigid monoidal category $\cat$ were introduced in \cite{BV-hopfmonads}.
These are monads $M$ on $\cat$, equipped with extra structure:
First, the functor 
$M\colon \cat \to \cat$
 is equipped with a lax comonoidal structure, i.e.\ there are morphisms
 \begin{align}
 &\text{(comultiplication)} &&M(X\otimes Y) \to M(X) \otimes M(Y) \ ,
 &\text{natural in } X,Y \ ,
\nonumber \\
&\text{(counit)} &&M(\tensUnit) \to \tensUnit \ ,
\end{align}
satisfying certain conditions.
The multiplication and the unit of the monad $M$ have to be comonoidal as well.
Finally, $M$ has (unique) left and right antipodes, again given by certain natural
transformations, see \Cref{sec:HopfMonads} for details.

A module over a monad $M$ is an object $V \in \cat$ together with an action $M(V) \to V$
of the monad.
For a Hopf monad, the category $\cat_M$ of $M$-modules is again a rigid monoidal category
\cite{BV-hopfmonads}.

\medskip

Let now $\cat$ be a finite tensor category.
The four monads we are interested in are all given in terms of coends. Namely, for $V \in
\cat$ we set
\begin{align*}
	\hopfMonad[1][V] &= \int^{X\in \cat} \dualR{X} \otimes (V \otimes X) 
	\ , 
	&
	\hopfMonad[2][V] &= \int^{X\in \cat}  \dualL{X} \otimes (V \otimes X)  \ ,
	\\
	\hopfMonad[3][V] &= \int^{X\in \cat}  (X \otimes V) \otimes \dualR{X}  
	\ , 
	&
	\hopfMonad[4][V] &= \int^{X\in \cat}  (X \otimes V) \otimes \dualL{X}  \ .
\end{align*}
Here $\dualR{X}$ denotes the right
	dual
 and $\dualL{X}$ the left dual of an object $X \in
\cat$ (see \Cref{SEC:Conventions} below for our conventions).
Finiteness of $\cat$ ensures existence of these coends.
The index $1,\dots,4$ indicates the ``position'' of the duality symbol
$\dualsymbol$.

$\hopfMonad[2]$ and $\hopfMonad[3]$ are Hopf monads \cite[Sec.~5.4]{BV-Double}, called
{\em central monads}, and are isomorphic as Hopf monads.
If $\cat$ is pivotal, then $\hopfMonad[1]$ and $\hopfMonad[4]$ are also Hopf monads, and
all four $\hopfMonad[i]$ are isomorphic as Hopf monads, see
\Cref{prop:iso_as_HC_pivotal}.
In the following, if we discuss properties of $\hopfMonad[1]$ and $\hopfMonad[4]$, we will
implicitly assume that $\cat$ is in addition pivotal.

We define four types of {\em monadic cointegrals} as follows.
Denote by $\distInvObj \in \cat$ the distinguished invertible object of $\cat$.\footnote{For $\cat=\hmodM$, the object $\distInvObj$ is the one-dimensional representation with the
$H$-action given by $\modulus^{-1}$, and $\modulus$ is the modulus of~$H$.}
The tensor unit $\tensUnit \in \cat$ is an $\hopfMonad[i]$-module via the counit.
$\hopfMonad[i](\distInvObj)$ is an $\hopfMonad[i]$-module by the monad multiplication. 
    A \emph{monadic cointegral for $\hopfMonad[i]$} is an intertwiner of
    $\hopfMonad[i]$-modules $\coint_i : \tensUnit \to \hopfMonad[i](\distInvObj)$.
    This means that $\coint_i$ is a morphism in $\cat$ which satisfies
\begin{equation}\label{intro:mon-coint-def}
    \begin{tikzcd}[row sep=large,column sep=large]
        \hopfMonad[i][\tensUnit] \ar[r, "{\hopfMonad[i](\coint_i)}"] 
        \ar[d,swap,"\counitCoend"]
        & \hopfMonad[i]^2(\distInvObj)
        \ar[d,"{\multCoend[i][\distInvObj]}"]
        \\
        \tensUnit \ar[r,swap,"\coint_i"]
        & \hopfMonad[i][\distInvObj] 
    \end{tikzcd}\qquad ,
\end{equation}
where $\counitCoend$ and $\multCoend$ denote the counit and multiplication of
$\hopfMonad$, respectively.

In the case $\distInvObj=\tensUnit$, this definition of monadic cointegral agrees with the
definition of a cointegral for a Hopf monad that first appeared in
\cite[Sec.~6.3]{BV-hopfmonads}.
For $i=2$ (and general $\distInvObj$), the above definition is related by an isomorphism
to the notion of ``categorical cointegral'' defined in \cite{Sh-integrals}
(\Cref{cor:OurCointIsShimizus}).

\begin{thm}
    For a finite tensor category $\cat$, non-zero monadic cointegrals for
$\hopfMonad[2], \hopfMonad[3]$ (and for $\hopfMonad[1], \hopfMonad[4]$ if $\cat$ is
pivotal) exist and are unique up to scalar multiples.
\end{thm}

This follows from a corresponding result in \cite{Sh-integrals}, see
\Cref{prop:MonCoint_uniq_ex}.

\medskip

We now come to a key observation, which explains the reason for introducing four slightly
different Hopf monads $\hopfMonad[i]$, even though they are all isomorphic.
Namely, for $H$ a finite-dimensional Hopf algebra and	$\cat = \hmodM$, each monad
$\hopfMonad[i]$ has a particularly natural realisation
(\Cref{example:HopfAlgebra-monads}).
With respect to this realisation one finds, firstly, that as a vector space
$\hopfMonad[i](\distInvObj) = H^*$, so that a $\coint$ as in \eqref{intro:mon-coint-def}
is an element of $H^*$, and, secondly, the equivalences
\begin{equation}\label{intro:Hopf-4types}
    \text{$\coint$ is a mon.\ coint.\ for }
    \begin{cases}
        \hopfMonad[1] \\
        \hopfMonad[2] \\
        \hopfMonad[3] \\
        \hopfMonad[4] 
    \end{cases}
    \Leftrightarrow
    \quad
    \text{$\coint$ is a }
    \begin{cases}
    \text{right $\modulus$-sym.} \\
    \text{right} \\
    \text{left} \\
    \text{left $\modulus$-sym.}
    \end{cases}
    \hspace{-1.5em}\text{coint.~for $H$ ,}
\end{equation}
see \Cref{example:HopfAlgebra-catcoint}.
As before, in cases 1 and 4 we require $H$ to be pivotal.

The universal properties of the coends $\hopfMonad[i]$ give isomorphisms between the
various spaces of monadic cointegrals, and hence also between the various spaces of
cointegrals for $H$. 
This will be useful in our application to quasi-Hopf algebras.

\subsection{Application to quasi-Hopf algebras}

For Hopf algebras, the relation between integrals and cointegrals is very simple: passing
to the dual Hopf algebra exchanges the two notions.
For quasi-Hopf algebras, this is no longer the case as the definition of a quasi-Hopf
algebra is	not
symmetric under duality.
While integrals for a quasi-Hopf algebra $H$ are defined in the same way as for Hopf
algebras, cointegrals $\coint \in H^*$ are markedly more complicated.

The definition of left/right cointegrals for a quasi-Hopf algebra $H$ is given
in~\cite{HN-integrals}, and that of $\modulus$-symmetrised left/right cointegrals
in~\cite{Shi-Shi,BGR1} and in \Cref{sec:D-sym-coint-qHopf} below.

\medskip

Fix a finite-dimensional quasi-Hopf algebra $H$ over an algebraically closed field.
As for Hopf algebras, $\hmodM$ is  also a finite tensor category, and each of the monads
$\hopfMonad[i]$ on $\hmodM$ has a natural realisation such that the underlying vector
space of $\hopfMonad[i](\distInvObj)$ is $H^*$ in all four cases, cf.\
\Cref{subseq:HopfMonad-qHopf}.
With these realisations, we describe the monadic cointegrals for $H$ via equations
involving quasi-Hopf data.
For example, an element $\coint\in H^*$ which is an $H$-module intertwiner
$\tensUnit \to A_2(D)$ is a monadic cointegral for $\hopfMonad[2]$ if and only if the
equation~\eqref{eq:qHopf-rmco-lin} holds.

Our main result is a generalisation of the relations in~\eqref{intro:Hopf-4types} for
quasi-Hopf algebras, namely the precise relation between the monadic cointegrals and the
quasi-Hopf cointegrals from~\cite{HN-integrals,Shi-Shi,BGR1}: 

\begin{table}[t]
	{\renewcommand{\arraystretch}{2}
	\begin{center}
		\begin{tabular}{c|c|c}
			If the quasi-Hopf 
			& then the element of 
			& is a monadic \\[-1em]
			cointegral $\coint$ is \dots 
			& $H^*$ given by \dots 
			& cointegral for \dots \\ 
			\hline
			right $\modulus$-sym. & 
			$\coint\Big( S(\betaQ) \,\placeholder\, S\inv(\elTh) \Big)$
			& $\hopfMonad[1]$
			\\ \hline
			right & 
			$\coint\Big( S(\betaQ) \,\placeholder\, S\inv(\elX) \Big)$
			& $\hopfMonad[2]$
			\\ \hline
			left & 
			$\coint\Big( S^{-2}(\betaQ) \,\placeholder\, S(\hat{\elX}) \Big)$
			& $\hopfMonad[3]$
			\\ \hline
			left $\modulus$-sym. & 
			$\coint\Big( \betaQ \,\placeholder\, S(\hat{\elTh}) \Big)$
			& $\hopfMonad[4]$
		\end{tabular}
\end{center}}

\caption{
Relation between various notions of cointegrals.
Here, ``$\,\placeholder\,$'' is a place holder for the function argument, $\betaQ$ is the
coevaluation element and $\elTh$, $\elX$, $\hat{\elX}$, $\hat{\elTh}$ are certain elements
of $H$ defined in \Cref{SEC:MainThm}.
For cases 1 and 4, $H$ is required to be pivotal.
If $H$ is a (non-quasi) Hopf algebra, then $\betaQ=\elTh=\elX=\hat{\elX}=\hat{\elTh}=1$
and one recovers the simple relation in \eqref{intro:Hopf-4types}.}
\label{tab:intro-main-thm}
\end{table}

\begin{thm}\label{intro:MainThm-qHopf}
	We have the bijections from the various types of cointegrals $\coint \in H^*$
	for the finite-dimensional quasi-Hopf algebra $H$ to the corresponding 
	types of monadic cointegrals in
	$\hmodM$ as shown in \Cref{tab:intro-main-thm}.
\end{thm}

This is shown in \Cref{thm:MainThmSectionStatement,thm:MainThmPivotalCase}. 

\begin{rmrk}
    Analogous to \cite{Sh-integrals} one
	can introduce
monadic integrals by noting that
    $\hopfMonad(\tensUnit)$ is a coalgebra and defining left/right monadic integrals for
    $\hopfMonad$ to be left/right comodule morphisms in $\cat(\hopfMonad(D),
    \tensUnit)$.
In the quasi-Hopf case, monadic integrals are elements of $H^{**} \cong H$, and one finds
that the left/right monadic integrals for $\hopfMonad$ are the usual left/right integrals
for the quasi-Hopf algebra $H$ (for $i=1,2$), and the right/left integrals (for $i=3,4$).
        In the Hopf case, a similar result was also noted in \cite{Sh-integrals}.
    
    A slightly different notion of integrals for Hopf monads has already appeared earlier
    in the literature \cite{BV-hopfmonads}.
If $\cat$ is braided, then by \cite[Ex.~5.4]{BV-hopfmonads} this notion agrees with the
definition of cointegrals for Hopf algebras in braided categories as in 
\cite{KerlerLyubashenko}.
By arguments analogous to those given in \Cref{SEC:braided_case}, it in turn also
agrees with the monadic integrals from the beginning of this remark.
    Here we will not go into the details of either of these points
    and focus instead on the study of monadic cointegrals.
\end{rmrk}

\medskip

\subsection{Application to braided tensor categories}

Let $\cat$ now be a braided finite tensor category.
In \cite{LyuMaj,Lyubashenko-mod-transfs} the coend 
    $\coend = \int^{X \in \cat} \dualL{X} \otimes X$
was studied and shown to be a Hopf algebra in $\cat$.
One can use the braiding of $\cat$ to construct a natural family of isomorphisms
\begin{equation*}
	\xi_V: \hopfMonad[2](V) \to \coend \otimes V \ .
\end{equation*}
This provides an isomorphism $\hopfMonad[2] \cong \coend \otimes \placeholder$ of Hopf
monads.

For Hopf algebras $H$ in braided categories, one can define integrals and cointegrals just
as for Hopf algebras over a field, up to one additional subtlety.
Namely, integrals are certain morphisms from a so-called {\em object of integrals}
$\objectOfInts[H] \in \cat$ to $H$.
Conversely, cointegrals are certain morphisms $H \to \objectOfInts[H]$, see
\cite{KerlerLyubashenko} and \Cref{SEC:braided_case}.
For example, a left integral for $H$ is a morphism $\intLyu : \objectOfInts[H] \to H$ such
that 
$$
\begin{tikzcd}[column sep=large]
    H \otimes \objectOfInts[H]
    \ar[rr,"\id\otimes \intLyu"] \ar[d,swap,"\counit \otimes \id"] 
    && H\otimes H \ar[d,"m"] \\
    \tensUnit \otimes \objectOfInts[H] \ar[r,"\sim"] & \objectOfInts[H] \ar[r,"\intLyu"]
    & H 
\end{tikzcd}
$$
commutes. Here $\counit$ and $m$ denote the counit and multiplication of $H$.

We find that monadic cointegrals for $\hopfMonad[2]$ are related to left integrals for
$\coend$, and that the object of integrals for $\coend$ is
$\objectOfInts[\coend]=\dualR{\distInvObj}$, the right dual of the distinguished
invertible object of $\cat$ (\Cref{prop:MonadicCointIsLyuInt}):

\begin{prop}
    Let $\cat$ be a braided finite tensor category.
    Then $\intLyu : \dualR{\distInvObj} \to \coend$ is a left integral for $\coend$ if and
    only if
    \begin{align*}
        \coint =
        \big[
        \tensUnit \xrightarrow{ \coevR_{\distInvObj} }
        \dualR{\distInvObj} \otimes \distInvObj
        \xrightarrow{ \intLyu \otimes \id}
        \coend \otimes \distInvObj
        \xrightarrow{\xi_{\distInvObj}\inv}
        \hopfMonad[2][\distInvObj]
        \big]
    \end{align*}
    is a monadic cointegral for $\hopfMonad[2]$.
\end{prop}

\Cref{prop:MonadicCointIsLyuInt} also contains a similar statement for the
relation of right integrals of $\coend$ with  monadic cointegrals for $\hopfMonad[2]$.

\medskip

An important application of $\coend$ and its integrals arises in the case that $\cat$ is
modular, that is, a (not necessarily semisimple) finite ribbon category whose braiding
satisfies a non-degeneracy condition called {\em factorisability} (see
\Cref{SEC:Sl2Z}).

In this case, one can define a projective representation of the genus-$g$ surface mapping
class group on the Hom-space $\cat(\coend^{\otimes g},1)$ \cite{Lyubashenko-mapClGroups},
as well as a non-semisimple variant of the Reshetikhin-Turaev topological field theory
\cite{DeRenzi-GP,DeRenzi-GR}.
The integral for $\coend$ (which is two-sided for $\cat$ modular) enters in both
constructions.

\medskip

Let us specialise to the case that $\cat = \hmodM$ for $H$ a finite-dimensional
quasi-triangular quasi-Hopf algebra which is in addition ribbon (and so in particular
pivotal). 
In \Cref{sec:qtriang-qHopf} we explicitly relate left integrals for $\coend$, right
monadic cointegrals 	in $\hmodM$,
and right cointegrals for $H$.
In \Cref{sec:SL2Z-qHopf} we assume furthermore that $H$ is factorisable (as defined
in \cite{BT-Factorizable}), which is equivalent to $\hmodM$ being factorisable
\cite{FGR1}. 
In this case, $\distInvObj = \tensUnit$ and both integrals and cointegrals for $\coend$ are two-sided.
We may take $\hopfMonad[2](\tensUnit) = \coend$ and by the above proposition, integrals
for $\coend$ are precisely the same as monadic cointegrals for $\hopfMonad[2]$.

We present the projective representation of $\slTwoZ$ on $\cat(\coend,\tensUnit)$ as an
example for the mapping class group actions mentioned above 
by giving the action of the $\mathbf{S}$ and $\mathbf{T}$ 
generators,
simplifying the corresponding expressions in \cite{FGR1}.
Denote by $Z(H)$ the centre of $H$ and write $\alphaQ Z = \{\alphaQ z \mid z \in Z(H)\}$,
where $\alphaQ$ is the evaluation element of $H$.
Recall that $\coend=H^*$ as a vector space.
One finds that via the natural isomorphism $H \cong H^{**}$ one has $\alphaQ Z \cong
\cat(\coend,\tensUnit)$.
In \Cref{prop:SL2Z-action-via-monadic-qHopf} we compute the action of 
    $\mathbf{S}$ and $\mathbf{T}$
on $\alphaQ Z$ to be, for $z \in Z(H)$,
\begin{align*}
	\mathbf{S} \, . \, (\alphaQ z) 
    = \langle 
    \coint \mid 
    \elOmega_1 z 
    \rangle
    \ \elOmega_2
    \qquad , \qquad
    \mathbf{T} \, . \, (\alphaQ z) 
    = \elRibbon\inv \alphaQ z \ .
\end{align*}
Here, $\coint$ is a monadic cointegral for $\hopfMonad[2]$,
$\elOmega_{1,2}$ are the components of the Hopf-pairing $\hopfPair : \coend \otimes \coend
\to \tensUnit$,
and $\elRibbon$ is the ribbon element of $H$, see \Cref{SEC:Sl2Z} for details.

\medskip

Finally, let us note that the construction in \cite{DeRenzi-GR} of a three-dimensional
topological field theory
from a modular category $\cat$ uses the integral for $\coend$ and the modified trace on
(the projective ideal in) $\cat$.
For $\cat = \hmodM$ with $H$ a factorisable ribbon quasi-Hopf algebra, monadic cointegrals
for $\hopfMonad[2], \hopfMonad[3]$ provide the integral for $\coend$, and monadic
cointegrals for $\hopfMonad[1], \hopfMonad[4]$ provide the modified trace via the
construction using symmetrised cointegrals in \cite{Shi-Shi,BGR1}. 
An important class of factorizable quasi-Hopf algebras as inputs for such topological
field theories comes from the fundamental examples of logarithmic conformal field
theories \cite{Gainutdinov:2015lja, FGR2,CGR, GLO18,N18}.

\medskip

The fact that monadic cointegrals provide integrals for $\coend$ and modified traces in a
uniform setting was one of the motivations to carry out the present investigation.
Another motivation was that because of their direct categorical interpretation, in certain
situations monadic cointegrals for quasi-Hopf algebras may be easier to use than those in
the left column of \Cref{tab:intro-main-thm}.

\subsection{Comparison to the approach of Shibata-Shimizu}

For quasi-Hopf algebras, a relation between right cointegrals for $H$ and categorical
cointegrals in the sense of~\cite{Sh-integrals} was derived in~\cite{Shi-Shi}.
The main theorem in the present paper (\Cref{intro:MainThm-qHopf}) is an analogous
result for the four types of monadic cointegrals.
Comparing to~\cite{Shi-Shi}, we work in a dual setting that uses central Hopf monads
instead of comonads.
This last choice of monads over comonads is merely conventional. 
However, our approach to the application of monadic cointegrals to quasi-Hopf cointegrals
is quite different from the one in~\cite{Shi-Shi}:
the latter uses a detour via the category of Yetter-Drinfeld modules, while we follow a
more direct route.  
In our approach, we found the monadic setting better suited to make the connection to
\cite{HN-integrals} than the comonadic picture.
     
A more conceptual relation between these
two pictures is described in \Cref{rem:H*mod-Hcomod}.

\subsection*{Outline of the paper}

In \Cref{SEC:Monadic_Cointegrals} we 
	start by reviewing
the definition of a Hopf monad.
Then, the central monads and their Hopf structures are described. 
We define four versions of \emph{monadic cointegrals}, and show that
they are related by isomorphisms
to the categorical cointegral considered in \cite{Sh-integrals}.

\smallskip

\Cref{SEC:Coints_qHopf} contains our conventions for (pivotal) quasi-Hopf algebras.
We specialise the definition of the various monadic cointegrals to the category of modules
over a quasi-Hopf algebra, and we review the definition and some properties of left and
right cointegrals for quasi-Hopf algebras from \cite{HN-integrals, BC2-2011}.

\smallskip

Our main theorem showing that quasi-Hopf cointegrals are equivalent to mo\-na\-dic
cointegrals is formulated in \Cref{SEC:MainThm}.
The main ideas of the proof are outlined, while technical details are deferred to
\Cref{sec:proofs-sec-4}.

\smallskip

Examples of monadic cointegrals for quasi-Hopf algebras are given in
\Cref{SEC:Examples}.

\smallskip

In \Cref{SEC:braided_case} we consider integrals for Hopf algebras in braided
finite tensor categories, and we relate left
	and right
integrals for $\coend$ to monadic cointegrals
for $\hopfMonad[2]$.
As an example, we treat finite-dimensional quasi-triangular quasi-Hopf algebras.

\smallskip

With $\cat = \hmodM$ the category of finite-dimensional modules over a
finite-dimen\-sional factorisable ribbon quasi-Hopf algebra $H$, in \Cref{SEC:Sl2Z}
we express the $\slTwoZ$-action on the centre of $H$ using the monadic and the quasi-Hopf
cointegral.

\bigskip

\subsection{Conventions}\label{SEC:Conventions}
    Throughout this paper we fix an algebraically closed field $\field$.
    Following~\cite{EGNO}, by a \emph{finite tensor category} we mean a
    $\field$-linear abelian category that 
    \begin{itemize}
        \item
            has finite-dimensional Hom-spaces, and every object is of finite length,
        \item
            possesses a finite set of isomorphism classes of simple objects,
        \item
            has enough projectives,
        \item
            is rigid monoidal, such that the tensor product functor is $k$-bilinear             
            and the monoidal unit $\tensUnit$ is simple.
    \end{itemize}
    We denote the left and the right dual of an object $X$ by $\dualL{X}$ and
    $\dualR{X}[-0.2]$, respectively.
    The corresponding evaluations and coevaluations are
    \begin{align}
        &\evL_X: \dualL{X} \otimes X \to \tensUnit
        \ ,\quad 
        \coevL_X : \tensUnit \to X \otimes \dualL{X}\ ,
        \notag \\
        &\evR_X: X\otimes \dualR{X}[-0.2] \to \tensUnit
        \ , \quad
        \coevR_X : \tensUnit \to \dualR{X}[-0.2] \otimes X\ ,
    \end{align}
    satisfying the familiar zig-zag equalities.
    We do not assume that $\cat$ is strict monoidal, and (compositions of) coherence
    isomorphisms will be indicated.

    \smallskip

    Our conventions for string diagrams are as follows.
    We read them from bottom to top, and coherence isomorphisms will usually not be
    drawn.

    Left and right coevaluation and evaluation for the object $X\in \cat$ are drawn as
    \begin{equation}
        \ipic{-0.4}{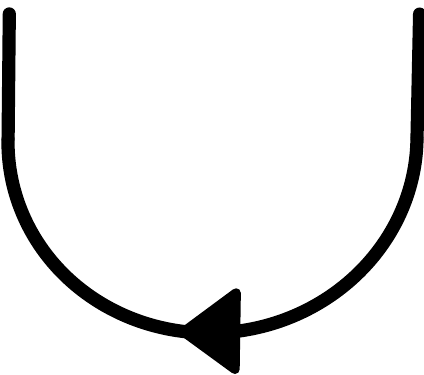}{0.3}
        \put (-40,23) {$X$}
        \put (-05,23) {$\dualL{X}$}
        \quad, \quad \quad
        \ipic{-0.4}{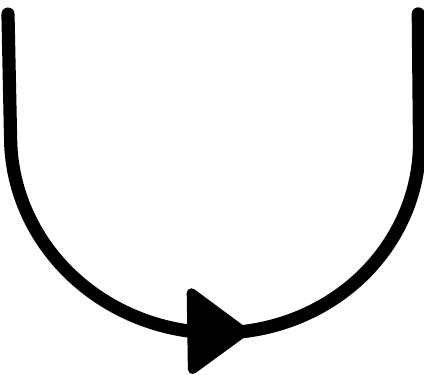}{0.3}
        \put (-46,23) {$\dualR{X}$}
        \put (-05,23) {$X$}
        \quad, \quad \quad
        \ipic{-0.0}{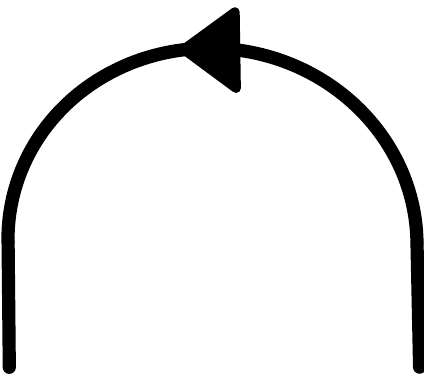}{0.3}
        \put (-42,-13) {$\dualL{X}$}
        \put (-07,-13) {$X$}
        \quad, \quad \quad
        \ipic{-0.0}{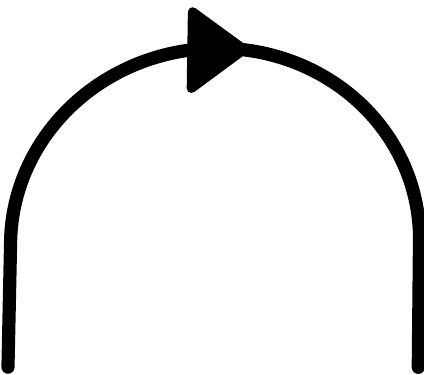}{0.3}
        \put (-11,-13) {$\dualR{X}$}
        \put (-42,-13) {$X$}
        \quad,
    \end{equation}
    respectively, so that in our conventions for duals and string diagrams, arrows on the
    duality maps for left (right) duals point to the left (right).

    \medskip
    
    A functor $F:\cat \to \cat[D]$ between monoidal categories is \emph{lax comonoidal} if
    there is a natural transformation $F_2$ and a morphism $F_0$,\footnote{
    Here and below we sometimes abbreviate $F(X)$ by $FX$ when applying functors to
    objects, and similarly for functors on morphisms.}
    \begin{align}
        F_2(X,Y): F(X\otimes Y) \to FX\otimes FY, \quad 
        F_0: F\tensUnit \to \tensUnit,
    \end{align}
    satisfying certain coherence conditions so that coalgebras in $\cat$ are mapped to
    coalgebras in $\cat[D]$.
    For that reason we will commonly refer to $F_2$ and $F_0$ as the comultiplication and
    the counit of the lax comonoidal functor $F$.
    If $F_2$ and $F_0$ are isomorphisms then $F$ is called a strong comonoidal functor.

    Similarly, a functor $F:\cat \to \cat[D]$ between monoidal categories is
    \emph{lax/strong monoidal} if $F\op : \cat\op \to \cat[D]\op$ is lax/strong
    comonoidal.
    The corresponding natural transformation $F_2$ and the morphism $F_0$ we call the
    multiplication and the unit, respectively.

    A natural transformation $\varphi: F \Rightarrow G$ between two comonoidal functors is
    called \emph{comonoidal} if it commutes with the comonoidal structures.
    That is, if
    \begin{align}
        G_2(X,Y) \circ \varphi_{X\otimes Y} 
        = (\varphi_X \otimes \varphi_Y) \circ F_2(X,Y)
        \quad \text{and} \quad 
        F_0 = G_0 \circ \varphi_{\tensUnit} 
    \end{align}
    is true for all objects $X,Y$.

    Monoidal natural transformations between monoidal functors are defined similarly.

    \medskip

    A rigid category $\cat$ is called \emph{pivotal} if there is a monoidal natural
    isomorphism $\pivotalStruct\colon \id_{\cat} \Rightarrow \ddualL{(\placeholder)}$,
    i.e.\ from the identity functor on $\cat$ to the double dual functor.
    The monoidal structure of the double dual is given in terms of the natural isomorphism
    \begin{align}\label{eq:gammaCanonicalIso}
        \gamma_{V,W}: \dualL{V} \otimes \dualL{W} \to \dualL{(W\otimes V)},
        \quad
        \gamma_{V,W}
        =
        \ipic{-0.5}{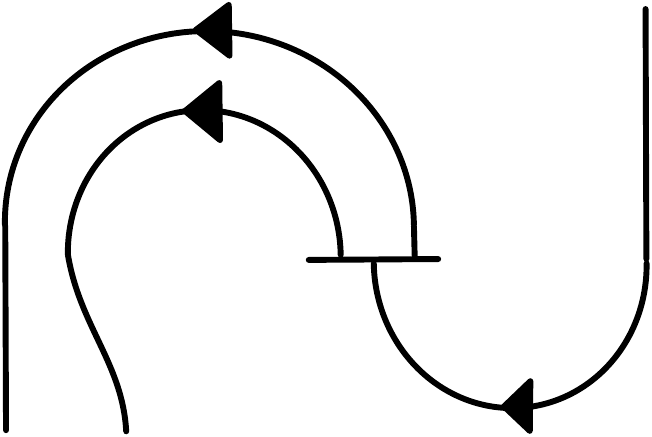}{0.5}
        \put (-98,-43) {$\dualL{V}$}
        \put (-80,-43) {$\dualL{W}$}
        \put (-25,35) {$\dualL{(W \otimes V)}$}
        \put (-30,-10) {$\id$}
    \qquad ,
    \end{align}
    as the composition
    \begin{align}
        \ddualL{V} \otimes \ddualL{W} 
        \xrightarrow{\gamma_{\dualL{V},\dualL{W}}}
        \dualL{\left( \dualL{W} \otimes \dualL{V} \right) }
        \xrightarrow{ \dualL{ \left( \gamma_{W,V} \inv \right) } }
        \ddualL{(V\otimes W)}.
    \end{align}

    Note that the existence of the pivotal structure $\pivotalStruct$ is equivalent to
    requiring that the left and the right dual functor be isomorphic as monoidal functors.
    Indeed, given $\pivotalStruct$ we can form the isomorphism
    \begin{align}\label{eq:RightDualIsLeft-pivotal}
        \ipic{-0.5}{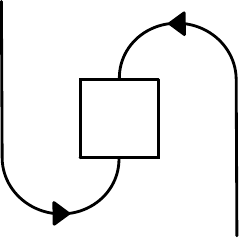}{1.0}
        \put (-78,38) {$\dualR{X}$}
        \put (-40,-03) {$\pivotalStruct_X$}
        \put (-06,-47) {$\dualL{X}$}
        \quad .
    \end{align}
    Conversely, given a natural monoidal isomorphism $\dualR{X}\cong \dualL{X}$, we have
    \begin{align}
        \ddualL{X} \cong \dualL{(\dualR{X})} \cong X
    \end{align}
    where the second isomorphism is
    \begin{align}
        \omega_X =
        \bigg[
            &\dualL{(\dualR{X})} 
            \xrightarrow{\sim} 
            \dualL{(\dualR{X})} \otimes \tensUnit 
            \xrightarrow{\id \otimes \coevR_X}
            \dualL{(\dualR{X})} \otimes (\dualR{X} X)
            \xrightarrow{\sim} 
            (\dualL{(\dualR{X})} \otimes \dualR{X}) X
        \notag \\
            &\xrightarrow{\evL_{\dualR{X}} \otimes \id }
            \tensUnit X \xrightarrow{\sim} X
        \bigg] \ .
        \label{eq:definitionOmegaIso}
    \end{align}
    We will suppress some of the tensor product symbols to
    shorten expressions, e.g.\ in the above expression we only left those tensor
    symbols between objects necessary to make the assignment of duals unambiguous.
	As a string diagram \eqref{eq:definitionOmegaIso} simply reads
    \begin{align}
        \omega_{X} = \quad
        \ipic{-0.5}{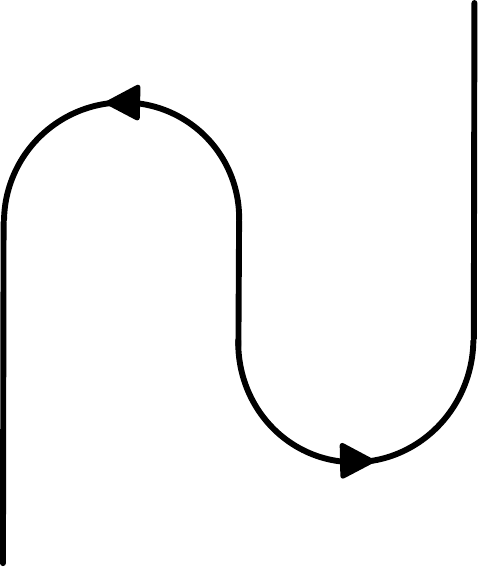}{0.5}
        \put (-82,-53) {$\dualL{(\dualR{X})}$}
        \put (-05,043) {$X$}
        \quad .
    \end{align}

\medskip

\section{Monadic cointegrals}\label{SEC:Monadic_Cointegrals}

This section contains the main definition of this paper, namely that of the four types of
monadic cointegrals (\Cref{def:main-def}).
To state the definition we first briefly review Hopf monads~\cite{BV-hopfmonads}, the
central Hopf monad, and the distinguished invertible object.
Finally we realise monadic cointegrals via Hopf comonads to establish existence and
uniqueness via results in~\cite{Sh-integrals}.

\subsection{Hopf monads}
\label{sec:HopfMonads}
We first recall the basic notions from the theory of Hopf monads on rigid categories.
Throughout, our conventions will closely follow \cite{BV-hopfmonads}.
\unnumberedsubsection{Monads.}
    Recall \cite[Sec.~ VI]{catWorkMath} that a monad $M$ on a category $\cat$ is an
    algebra in $\End(\cat)$, the category of
     endofunctors of $\cat$, which is a monoidal
    category under composition.
    This means that there are natural transformations
    \begin{align}
        \mu:M^2 \Rightarrow M, \quad \eta: \id_{\cat} \Rightarrow M,
    \end{align}
    the \emph{multiplication} and \emph{unit} of $M$, respectively, satisfying
    \begin{equation}
        \begin{tikzcd}
            M^3 V \arrow[r,"M \mu_V"] \arrow[d,"\mu_{MV}",swap] 
            & M^2 V  \arrow[d,"\mu_V"] \\
            M^2 V  \arrow[r,"\mu_{V}"] 
            & M V  
        \end{tikzcd}
        \qquad \text{and} \qquad
        \begin{tikzcd}
            M^2 V  
            \arrow[r, leftarrow, "\eta_{M V}"] 
            \arrow[rd, swap, "\mu_V"]
                &M V  
                \arrow[r, "M \eta_V"]
                \arrow[d, equal]
                    &M^2 V 
            \arrow[ld, "\mu_V"]
            \\
                &M V  &
        \end{tikzcd} ,
    \end{equation}
    for each $V\in\cat$.

    A \emph{module} over a monad $M$ is a tuple $(V,\rho)$, consisting of an object
    $V\in \cat$ together with a morphism $\rho:MV \to V$, called the action, such
    that 
    \begin{equation}
        \begin{tikzcd}
            M^2 V \arrow[r,"M \rho"] \arrow[d,"\mu_{V}"] 
            & M V  \arrow[d,"\rho"] \\
            M V  \arrow[r,"\rho"] 
            & V 
        \end{tikzcd}
        \qquad \text{and} \qquad
        \begin{tikzcd}
            V \arrow[r,"\eta_V"] \arrow[rd,equal]
            & MV  \arrow[d, "\rho"]
            \\
            & V
        \end{tikzcd}
    \end{equation}
    commute.
    A morphism of $M$-modules from $(V,\rho)$ to $(W,\sigma)$ is a morphism $f:V\to W$ in
    $\cat$ which commutes with the action, i.e.\
    \begin{align}
        \sigma \circ Mf = f \circ \rho.
    \end{align}
    The category of $M$-modules is denoted by $\cat_M$.
    The forgetful functor from $\cat_M$ to $\cat$ has a left adjoint which sends an
    object $V\in \cat$ to the free $M$-module $\left( MV, \mu_V \right)$.\footnote{
        For any $M$-module $(B,\nu)$ and any object $A$ in $\cat$,
        the adjunction can be seen from the natural isomorphism 
        $\cat_M \big( (MA, \mu_A), (B, \nu) \big) \cong \cat(A, B)$,
        given by sending an $M$-module morphism $f$ to $f \circ \eta_A$, and  conversely,
        a morphism $g:A\to B$ is sent to $\nu \circ Mg$.
    }

    A morphism of monads is a morphism $\phi: M \Rightarrow M'$ of algebras in
    $\End(\cat)$.
    It therefore induces a functor $\phi^*:\cat_{M'} \to \cat_{M}$ via pullback, cf.\
    \cite[Lem.~1.6]{BV-hopfmonads}.

    If $\cat$ is monoidal, and $A\in \cat$ is an algebra, then $A\otimes \placeholder$ is
    a monad. 
    A comonad on $\cat$ is a monad on $\cat\op$.

\unnumberedsubsection{Bimonads}
    If $\cat$ is a monoidal category, then a \emph{bimonad} on $\cat$ is a 
    monad $M$ such that the functor $M$ is lax comonoidal and the multiplication and unit
    of $M$ are co\-mon\-oi\-dal natural transformations.

    The name ``bimonad'' is in analogy to algebras and bialgebras:
    The category of modules over a bimonad $(M,M_0,M_2)$ is monoidal, and a lax comonoidal
    structure on $M$ is the same as a monoidal structure on $\cat_M$ such that the
    forgetful functor to $\cat$ is strong monoidal, cf.\ \cite[Thm.~7.1]{Moerdijk}.
    Given two $M$-modules $(V,\rho), (W,\sigma)$, their tensor product is defined by
    \begin{align}
        (V,\rho)\otimes (W,\sigma) = (V\otimes W,(\rho \otimes \sigma)\circ M_2(V,W)),
    \end{align}
    and the monoidal unit of $\cat_M$ is the $M$-module $(\tensUnit, M_0)$, which we will
    also denote by $\tensUnit$.

    A morphism of bimonads is a comonoidal natural transformation which is a morphism of
    the underlying monads.
	We will later need the following lemma.

    \begin{lemma}[{\cite[Lem.~2.7]{BV-hopfmonads}}]
        \label{lem:BimonadIsosGiveFunctors}
        Let $M, M'$ be bimonads on $\cat$. 
        Then there is a one-to-one correspondence between morphisms $f\colon M \To M' $ of
        bimonads and strict monoidal functors $F\colon \cat_{M'} \to \cat_{M}$ whose
        underlying functor on $\cat$ is the identity functor.
    \end{lemma}

    \begin{rmrk}\label{rem:BimonadIsosGiveFunctors}
    Given a bimonad morphism $f$, the corresponding functor determined via
    \Cref{lem:BimonadIsosGiveFunctors} is just the pullback of the $M'$-module
    structure along $f$, i.e.\ $F=f^*$.
    Conversely, the monad morphism $f:M\To M'$ corresponding to the functor $F:\cat_{M'}
    \to \cat_M$ is given as follows.
    Let $(M'V,F\mu'_V)\in \cat_M$ be the image of the free $M'$-module on $V$ under $F$.
    Then $f_V = F\mu'_V \circ M\eta'_V$.
    In particular, $f^* = F$.
    \end{rmrk}

    If $\cat$ is braided monoidal, and $B\in \cat$ is a bialgebra, then $B\otimes
    \placeholder$ is a bimonad. 
    A bicomonad on $\cat$ is a bimonad on $\cat\op$.

\unnumberedsubsection{Hopf monads}
    A bimonad $M$ on a rigid category $\cat$ is called a \emph{Hopf monad} if $\cat_M$ is
    rigid, following again the familiar nomenclature of algebras, bialgebras, and Hopf
    algebras.
    For a Hopf algebra, the rigid structure of its category of modules is encoded in the
    antipode.
    For Hopf monads, the corresponding concept is as follows, cf.\ \cite{BV-hopfmonads}.
    A natural transformation $S^l$ with components
    \begin{align}
        S^l_V: M\big(\dualL{(MV)}\big) \to \dualL{V} 
    \end{align}
    is called a \emph{left antipode} for $M$ if it satisfies
    \begin{equation}\label{comdiag:antip_coev}
        \begin{tikzcd}[sep=large]
            M(\dualL{(MV)}\otimes V) \arrow[r, "M_2"] 
                \arrow[dd, swap, "M(\dualL{\eta_V}\otimes \id)"]
            & M(\dualL{(MV)}) \otimes MV  \arrow[r, "M\dualL{\mu_V} \otimes \id"]
            & M(\dualL{(M^2V)}) \otimes MV \arrow[d, "S^l_{MV} \otimes \id"] \\
            & & \dualL{(MV)} \otimes MV \arrow[d,"\evL_{MV}"] \\
            M(\dualL{V} \otimes V) \arrow[r,swap,"M
            \evL_{V}"]            
            &M\tensUnit \arrow[r,swap,"M_0"]
            & \tensUnit
        \end{tikzcd}
    \end{equation}
    and
    \begin{equation}\label{comdiag:antip_ev}
        \begin{tikzcd}[sep=huge]
            M\tensUnit \arrow[r,"M_0"] 
                \arrow[d,swap,"M(\coevL_{MV})"]
            & \tensUnit \arrow[r,"\coevL_V"]
            & V \otimes \dualL{V} \arrow[d,"\eta_V \otimes \id"] \\
            M(MV \otimes \dualL{(MV)}) \arrow[r,swap,"M_2"]
            & M^2V \otimes M(\dualL{(MV)}) \arrow[r,swap,"\mu_V\otimes S^l_V"]
            & MV \otimes \dualL{V}
        \end{tikzcd}.
    \end{equation}
    Given an $M$-module $(V,\rho)$, the antipode allows us to define a morphism
    \begin{align}\label{eq:rho-dual-def}
        \tilde{\rho}
        =
        \left[
            M(\dualL{V})
            \xrightarrow{M(\dualL{\rho})} 
            M\big(\dualL{(MV)}\big) 
            \xrightarrow{S^l_V} \dualL{V}
        \right],
    \end{align}
    which turns $(\dualL{V},\tilde{\rho})$ into an $M$-module
    \cite[Thm.~3.8]{BV-hopfmonads}.
    The evaluation and coevaluation are those in $\cat$,
    \begin{align}
        \evL_{(V,\rho)} = \evL_V
        ,\quad 
        \coevL_{(V,\rho)} = \coevL_V,
    \end{align}
    and that they are indeed $M$-module intertwiners is guaranteed by the two commuting
    diagrams \eqref{comdiag:antip_coev} and \eqref{comdiag:antip_ev}.
    Right duals via the right antipode are defined similarly.
    It was also shown in \cite[Thm.~3.8]{BV-hopfmonads} that $\cat_M$ is rigid if and only
    if the left and right antipodes exist, and that the antipodes are unique.

    A morphism of Hopf monads is a morphism of the underlying bimonads.
    It automatically commutes with the antipodes, \cite[Lem.~3.13]{BV-hopfmonads}.

    A Hopf comonad on $\cat$ is a Hopf monad on $\cat\op$.
    
    If $\cat$ is braided rigid monoidal, and $H\in \cat$ is a Hopf algebra with invertible
    antipode, then $H\otimes \placeholder$ is a Hopf monad, see
    \cite[Ex.~3.10]{BV-hopfmonads}.
    This example will be important in \Cref{SEC:braided_case}.    

\subsection{The central Hopf monad}
    Throughout the rest of this section $\cat$ will denote a finite tensor category.
    Recall the notion of a coend from e.g.\ \cite{catWorkMath,KerlerLyubashenko} or
    \cite[Sec.\,4.2]{Fuchs-Schweigert}.
    It follows from \cite[Cor.~5.1.8]{KerlerLyubashenko} that the coends 
    \begin{align}
        \hopfMonad[1][V] = \int^{X\in \cat} \hopfMonadComp{1} , 
        \quad \quad \quad
        \hopfMonad[2][V] = \int^{X\in \cat} \hopfMonadComp{2}
        \notag \\
        \hopfMonad[3][V] = \int^{X\in \cat} \hopfMonadComp{3} , 
        \quad \quad \quad
        \hopfMonad[4][V] = \int^{X\in \cat} \hopfMonadComp{4}
    \end{align}
    exist for all $V\in \cat$.
    A different and more detailed proof of existence is given in
    \cite[Thm.~3.6]{Sh-unimodFinTens}.
    Note that the subscript $i$ indicates the `position' of the dual symbol
    ${}^\dualsymbol$.
    We write $\dinatCoend[i][V]$ for the universal dinatural transformation of
    $\hopfMonad[i][V]$ so that for example
    \begin{align}
        \dinatCoend[2][V][X]: \hopfMonadComp{2} \to \hopfMonad[2][V].
    \end{align}

    In particular, $\hopfMonad:V \mapsto \hopfMonad[i][V]$ is an endofunctor, and the
    universal dinatural transformations $\dinatCoend[i][V]$ are natural in $V\in
    \cat$.
    In our graphical notation the dinatural transformation $\dinatCoend[2][V]$ is
    drawn as
    \begin{equation}
        \ipic{-0.5}{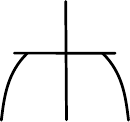}{1.4}
        \put (-40,030) {$\hopfMonad[2][V]$}
        \put (-20,009) {$\dinatCoend[2][V][X]$}
        \put (-59,-36) {$\dualL{X}$}
        \put (-31,-36) {$V$}
        \put (-07,-36) {$X$}
    \end{equation}
    for all $V,X\in \cat$.
    Functors like $\hopfMonad[i]$, and in particular the functor $\hopfMonad[2]$, were
    already studied in e.g.\ \cite{BV-Double}. 
    The latter is known as the \emph{central Hopf monad} \cite{Sh-unimodFinTens}.

    \smallskip

    \newcommand{\varI}{2}
    We will now describe the monad structures in more detail, restricting our exposition
    to the case $i=\varI$. 
    The monad structure is similar for all other cases.

    Recall the natural isomorphism $\gamma_{X,Y}: (\dualL{X}) (\dualL{Y}) \to
    \dualL{(YX)}$ from \eqref{eq:gammaCanonicalIso}.
    The multiplication $\multCoend[2]:(\hopfMonad[2])^2 \Rightarrow \hopfMonad[2]$ with
    components $\multCoend[2][V]:\hopfMonad[2]\hopfMonad[2](V) \to \hopfMonad[2][V]$ is
    determined by the universal property of coends via
    \begin{equation}
        \label{eq:DefMultCoend}
        \ipic{-0.5}{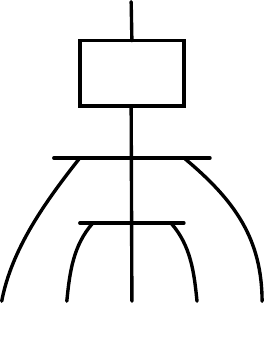}{1.2}
        \put (-59,065) {$\hopfMonad[2][V]$}
        \put (-60,33) {$\multCoend[2][V]$}
        \put (-102,13) {$\dinatCoend[2][\hopfMonad[2]V][Y]$}
        \put (-44,-12) {\footnotesize $\dinatCoend[2][V][X]$}
        \put (-098,-055) {$\dualL{Y}$}
        \put (-075,-055) {$\dualL{X}$}
        \put (-051,-055) {$V$}
        \put (-029,-055) {$X$}
        \put (-006,-055) {$Y$}
        \quad = \quad
        \ipic{-0.4}{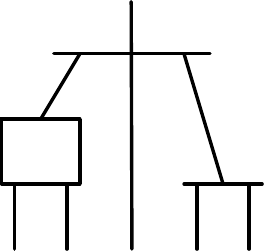}{1.3}
        \put (-62,061) {$\hopfMonad[2][V]$}
        \put (-30,041) {$\dinatCoend[2][V][X\otimes Y]$}
        \put (-95,-02) {$\gamma_{Y,X}$}
        \put (-12,-08) {$\id_{X\otimes Y}$}
        \put (-99,-49) {$\dualL{Y}$}
        \put (-80,-49) {$\dualL{X}$}
        \put (-54,-49) {$V$}
        \put (-33,-49) {$X$}
        \put (-09,-49) {$Y$}
        \put (30,0) {.}
    \end{equation}
    Here we used what is known as the `Fubini theorem' for ends and coends, cf.\
    \cite[Sec.~IX.8]{catWorkMath}, see also \cite[Rem.~1.9]{fosco}, to express the
    dinatural transformation of the iterated coend $\hopfMonad[2]\hopfMonad[2](V)$
    in terms of $\dinatCoend[2][V]$ and $\dinatCoend[2][\hopfMonad[2]V]$.

    The unit of $\hopfMonad[2]$, i.e.\ the natural transformation $\unitCoend[\varI]:
    \id_{\cat} \Rightarrow \hopfMonad[\varI]$, is defined as
    \begin{equation}
        \unitCoend[\varI][V] \defined
        \left[            
            V \xrightarrow{\sim} \hopfMonadComp{\varI}[\tensUnit] 
            \xrightarrow{ \dinatCoend[\varI][V][\tensUnit] } \hopfMonad[\varI][V] 
        \right]
        \ .
    \end{equation}

    \renewcommand{\varI}{2}
    For $i=2,3,$ \hopfMonad is always a Hopf monad \cite[Sec.~5.4]{BV-Double}.
    As an example we again consider $i=\varI$, the other case is similar. 
    The lax comonoidal structure is defined by\footnote{
        Here and in similar places below, we often omit spelling out all components and
        arguments of the dinatural transformations, e.g.\ on the LHS we have
        $\dinatCoend[2](U\otimes V)_X$, etc.
    }
    \begin{equation}
        \ipic{-0.5}{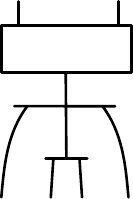}{1.4}
        \put (-60,045) {$\hopfMonad[2][U]$}
        \put (-16,045) {$\hopfMonad[2][V]$}
        \put (-50,018) {$\comultCoend[2][U,V]$}
        \put (-03,-03) {$\dinatCoend[2]$}
        \put (-20,-22) {$\id$}
        \put (-60,-53) {$\dualL{X}$}
        \put (-39,-53) {$U$}
        \put (-25,-53) {$V$}
        \put (-07,-53) {$X$}
        \quad = \quad  
        \ipic{-0.5}{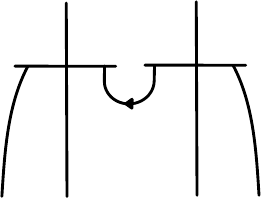}{1.4}
        \put (-090,045) {$\hopfMonad[2][U]$}
        \put (-040,045) {$\hopfMonad[2][V]$}
        \put (-075,018) {$\dinatCoend[2]$}
        \put (-004,011) {$\dinatCoend[2]$}
        \put (-112,-51) {$\dualL{X}$}
        \put (-083,-51) {$U$}
        \put (-030,-51) {$V$}
        \put (-007,-51) {$X$}
        , \quad \quad \quad 
        \ipic{-0.5}{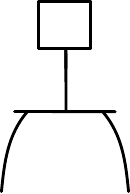}{1.2}
        \put (-28,023) {$\counitCoend[2]$}
        \put (-03,001) {$\dinatCoend[2]$}
        \put (-51,-54) {$\dualL{X}$}
        \put (-07,-54) {$X$}
        \quad = \quad  
        \ipic{-0.4}{./img/evalL.pdf}{0.3}
        \put (-42,-26) {$\dualL{X}$}
        \put (-07,-26) {$X$}
        \ .
        \label{eq:CoendComonoidalStructure}
    \end{equation}
    These are the comultiplication and counit of $\hopfMonad[2]$.
	The left antipode of $\hopfMonad[2]$ is defined by
    \begin{align}
        \ipic{-0.5}{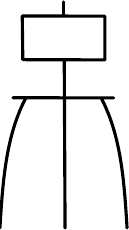}{1.4}
        \put (-30,052) {$\dualL{U}$}
        \put (-41,028) {$\leftAntCoend[2][U]$}
        \put (-04,004) {$\dinatCoend[2]$}
        \put (-65,-58) {$\dualL{X}$}
        \put (-45,-58) {$\dualL{(\hopfMonad[2]U)}$}
        \put (-05,-58) {$X$}
        \quad \quad 
        &= 
        \quad  \quad
        \ipic{-0.5}{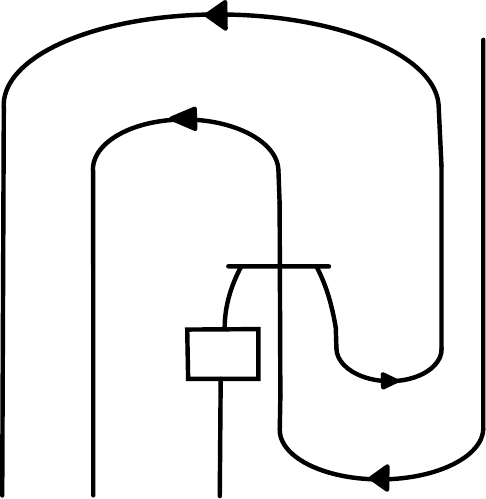}{0.9}
        \put (-004,057) {$\dualL{U}$}
        \put (-050,-00) {{\footnotesize$\dinatCoend[2][U][\dualR{X}]$}}
        \put (-073,-30) {$\sim$}
        \put (-134,-77) {$\dualL{X}$}
        \put (-115,-77) {$\dualL{(\hopfMonad[2]U)}$}
        \put (-074,-77) {$X$}
        \ ,
    \end{align}
    following \cite[Thm.~5.6]{BV-Double}.
    Here, by $\sim$ we mean the canonical isomorphism $X \cong \dualL{(\dualR{X})}$,
    defined similarly to $\omega_X: \dualR{(\dualL{X})}\to X$ from
    \eqref{eq:definitionOmegaIso}.
    The right antipode
    is obtained analogously.    

    For $i=1,4$, the above definition of a bimonad structure on $\hopfMonad[i]$ does not
    work.
    If $\cat$ is pivotal, however, the natural isomorphism $\dualL{X}\cong \dualR{X}$
    from \eqref{eq:RightDualIsLeft-pivotal} can be used when the duals do not match up in
    the comultiplication and counit.
    For example, the counit of $\hopfMonad[1]$ is given by
    \begin{equation}
        \ipic{-0.5}{./img/counitCoend.pdf}{1.2}
        \put (-28,022) {$\counitCoend[1]$}
        \put (-03,-05) {$\dinatCoend[1]$}
        \put (-56,-45) {$\dualR{X}$}
        \put (-07,-45) {$X$}
        \quad \quad = \quad
        \ipic{-0.5}{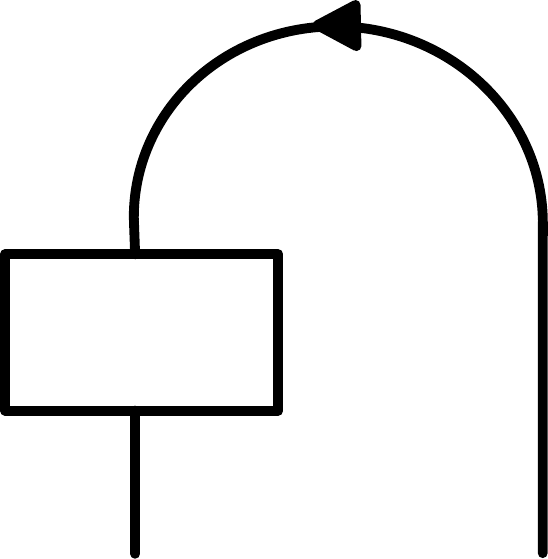}{0.4}
        \put (-62,-09) {$\eqref{eq:RightDualIsLeft-pivotal}$}
        \put (-60,-44) {$\dualR{X}$}
        \put (-07,-44) {$X$}
        \quad .
    \end{equation}
    One checks that in this way one obtains a Hopf monad structure on $\hopfMonad[1]$ and
    $\hopfMonad[4]$.
    \medskip    

    We summarise the preceding discussion in the following proposition.

    \begin{prop}\label{prop:all-are-HM}
        The functors $\hopfMonad[2]$ and $\hopfMonad[3]$ are Hopf monads. 
        If $\cat$ is pivotal, then $\hopfMonad[1]$ and $\hopfMonad[4]$ are also Hopf
        monads.
    \end{prop}

    The following proposition will now show that the canonical natural isomorphisms
    $\kappa_{i,j}: \hopfMonad[i] \To \hopfMonad[j]$ defined by
    \begin{align}
        (\kappa_{i,j})_V \circ \dinatCoend[i][V][X]
        = \dinatCoend[j][V][X'] \circ (\text{isomorphism of components}),
    \end{align}
    are isomorphisms of Hopf monads.\footnote{
        Alternatively, one could have introduced the Hopf monad structure only on
        $\hopfMonad[2]$, and then used the natural isomorphisms $\kappa_{2,j}$ to
        transport the structure to the other functors $\hopfMonad[j]$.
        However, each $\hopfMonad[j]$ comes with a canonical choice of Hopf monad
        structure defined using universal properties of coends, which is the structure we
        want to use throughout the remainder of the text.
        Thus we prefer to present the canonical structure in the four cases first and then
        establish the isomorphisms afterwards.
    }
    Here, $X'$ stands for $X$, $\dualL{X}$ or $\dualR{X}$ as appropriate, and the
    `isomorphisms of components' consist of coherence isomorphisms and the
    isomorphisms $\dualR{(\dualL{X})} \cong X$ and, in the pivotal case, ${\dualL{X}\cong
    \dualR{X}}$.
    For example,
    \begin{align}
        \label{eq:def-kappa_23}
        &\left[ 
            \hopfMonadComp{2}
            \xrightarrow{\dinatCoend[2][V][X]} \hopfMonad[2][V] 
            \xrightarrow{(\kappa_{2,3})_V } \hopfMonad[3][V] 
        \right]
        \nonumber\\
        &~=~
        \left[ 
            \hopfMonadComp{2}
            \xrightarrow{\sim}
            (\dualL{X}\, V)X
            \xrightarrow{\id \otimes \omega_X^{-1}}
            (\dualL{X}\,V)(\dualR{(\dualL{X})})
            \xrightarrow{\dinatCoend[3][V][\dualL{X}]} \hopfMonad[3][V] 
        \right]
        \ .
    \end{align}

    \begin{prop}\label{prop:iso_as_HC_pivotal}
        The natural isomorphism $\kappa_{2,3}$ is an isomorphism of Hopf monads.
        If $\cat$ is pivotal, then $\kappa_{i,j}$ are isomorphisms of Hopf monads for all
        $i,j$.
    \end{prop}

    \begin{proof}
        We first claim that the pullback $F = (\kappa_{2,3})^* : \cat_{\hopfMonad[3]} \to
        \cat_{\hopfMonad[2]}$ is a well-defined functor.
        Namely, on an $\hopfMonad[3]$-module $(V,\rho)$ the functor acts as
        $F(V,\rho) = (V,F\rho)$, where we define $F\rho : \hopfMonad[2]V\to V$ by
        \begin{align}
            \label{eq:A2isoA3}
            F{\rho} \circ \dinatCoend[2][V][X] 
            = 
            \left[
                \hopfMonadComp{2}
                \xrightarrow{\sim}
                (\dualL{X}\,V)(\dualR{(\dualL{X})})
                \xrightarrow{\dinatCoend[3][V][\dualL{X}]}
                \hopfMonad[3][V] 
                \xrightarrow{ \rho }
                V \right],
        \end{align}
        so that indeed $F\rho = \rho \circ \kappa_{2,3}$ by \eqref{eq:def-kappa_23}.
        A calculation shows that $F\rho$ is an $\hopfMonad[2]$-action.

        Next we check the conditions in \Cref{lem:BimonadIsosGiveFunctors}.
        As $F$ is given by pullback, the underlying functor is the identity on $\cat$. To
        verify strict monoidality, one checks that for $(V,\rho)$, $(W,\sigma)\in
        \cat_{\hopfMonad[3]}$ one has
        \begin{align}
            F\left({ \rho \otimes \sigma \circ \comultCoend[3][V,W] }\right)
            = 
            \big(F\rho \otimes F\sigma\big) \circ \comultCoend[2][V,W]
        \end{align}
        and $F(\counitCoend[3]) = \counitCoend[2]$, which is easy to see.

        Thus from \Cref{lem:BimonadIsosGiveFunctors} we obtain a morphism of bimonads
        (and hence of Hopf monads) $\hopfMonad[2] \Rightarrow \hopfMonad[3]$.
        Since $F = (\kappa_{2,3})^*$, by \Cref{rem:BimonadIsosGiveFunctors} this
        morphism is given by $\kappa_{2,3}$.
        As $\kappa_{2,3}$ is an isomorphism, we finally get $\hopfMonad[2] \cong
        \hopfMonad[3]$.

        If $\cat$ is pivotal, then e.g.\ the equivalence $G:\cat_{\hopfMonad[2]}\to
        \cat_{\hopfMonad[1]}$ is given by $G(V,\rho) = (V,G\rho)$ with
        \begin{align}
            G{\rho} \circ \dinatCoend[1][V][X] 
            = 
            \left[
                \dualR{X} (VX) \xrightarrow{\sim} \dualL{X} (VX)
                \xrightarrow{\dinatCoend[2][V][X]}
                \hopfMonad[2][V] 
                \xrightarrow{ \rho }
                V
            \right],
        \end{align}
        where the first isomorphism is given by the inverse to the one
        in~\eqref{eq:RightDualIsLeft-pivotal}.
        It is straightforward to check that $G = (\kappa_{1,2})^*$ and that it is strict
        monoidal and it is the identity on objects from $\cat$.
        Hence $\hopfMonad[2] \cong \hopfMonad[1]$ as Hopf monads.
    \end{proof}

	It is not hard to see that $\cat_{\hopfMonad[2]}\cong \cat[Z(C)] \cong
	\cat_{\hopfMonad[3]}$ as monoidal categories, where $\cat[Z(C)]$ is the Drinfeld
	centre of $\cat$, cf.\ \cite[Sec.~9.3]{BV-hopfmonads}.
    This was the reason to introduce central monads, and also explains the name.

    \smallskip

    \begin{example}\label{example:HopfAlgebra-monads}
        Let $\cat=\hmodM$ be the category of finite-dimensional modules over a
        finite-dimensional Hopf algebra $H$, and let $i= 2,3$.
        As vector spaces, the $\hopfMonad(V)$ are isomorphic to $H^*\otimes V$, and we
        choose the module structures as follows.
        With $h\in H$, $f \in H^*$, and $v\in V$,
        the action $\actionCoend$ of $H$
        on $\hopfMonad(V)$ is~\footnotemark
        \begin{align}
            \notag
            \actionCoend[2][h][f\otimes v]&=
            \langle f \mid S(h\sweedler{1}) \placeholder h\sweedler{3} \rangle 
            \otimes h\sweedler{2}v
            \ ,
            \\
            \actionCoend[3][h][f\otimes v]&=
            \langle f \mid S\inv(h\sweedler{3}) \placeholder h\sweedler{1} \rangle 
            \otimes h\sweedler{2}v
            \ .
            \label{eq:Hopf-example-action-A2-A3}
        \end{align}
    Here we use the sumless Sweedler-notation $\Delta(h) = h\sweedler{1} \otimes
    h\sweedler{2}$ etc., see \Cref{sec:Hopf-conf-def} for details.
        Note that $\hopfMonad[2][\tensUnit]$ is the coadjoint representation of $H$, cf.\
        \cite[Sec.~7]{FGR1}.
        \footnotetext{We use $\langle\placeholder \mid \placeholder \rangle: V^*\otimes V
            \to \field$ to denote the canonical pairing in vector spaces.
        }
        The universal dinatural transformations are defined as
        \begin{align}
            \dinatCoend[2][V][X](f \otimes v \otimes x) &= 
            \sum_i
            \langle
                f \mid e_i.x
            \rangle
            e^i \otimes  v,
            \notag \\ 
            \dinatCoend[3][V][X](x \otimes v \otimes f) &= 
            \sum_i
            \langle
                f \mid e_i.x
            \rangle
            e^i \otimes  v,
            \label{eq:dinat_trans-Hopf}
        \end{align}
        where $f \in X^*$, $v\in V$, $x\in X$, and $\{e_i\}$ is a basis of $H$ with dual
        basis $\{e^i\}$.
    In string diagram notation, these read
    \renewcommand{\varI}{2}
        \begin{equation}
            \dinatCoend[2][V][X] = \quad
            \ipic{-0.5}{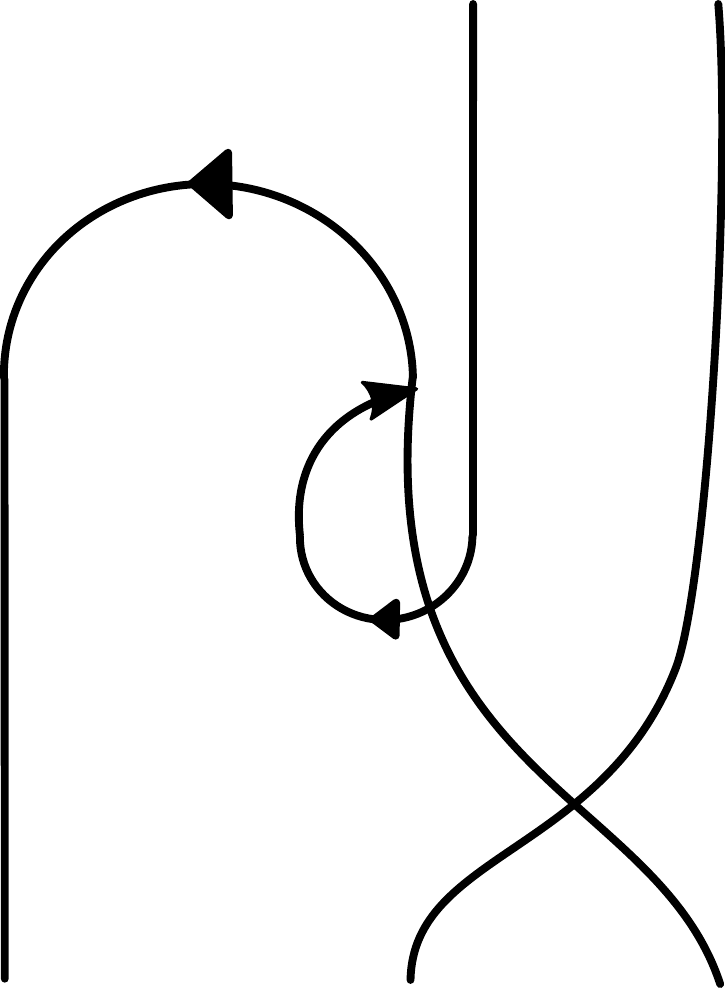}{0.3}
            \put (-28,048) {$H^*$}
            \put (-04,048) {$V$}
            \put (-69,-55) {$\dualL{X}$}
            \put (-32,-55) {$V$}
            \put (-06,-55) {$X$}
            \put (010,050) {$\boxed{\Vect_{\field}}$}
            \quad \quad \quad,\quad
            \dinatCoend[3][V][X] = \quad
            \ipic{-0.5}{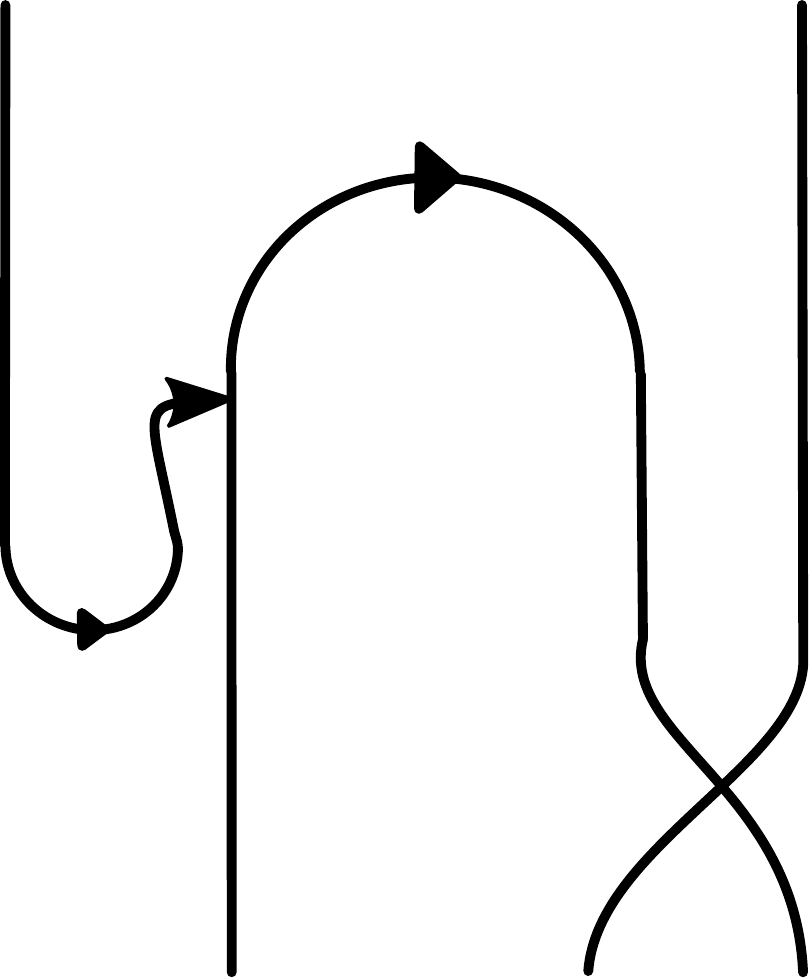}{0.3}
            \put (-73,048) {$H^*$}
            \put (-04,048) {$V$}
            \put (-55,-55) {$X$}
            \put (-24,-55) {$V$}
            \put (-06,-55) {$\dualR{X}$}
            \put (010,050) {$\boxed{\Vect_{\field}}$}
            \ ,
        \end{equation}
    where the boxed $\Vect_{\field}$ signifies that this is to be read in the
    category of $\field$-vector spaces.
        The actions in~\eqref{eq:Hopf-example-action-A2-A3} are uniquely determined by
        requiring $\dinatCoend[2]$ and $\dinatCoend[3]$ to be morphisms in $\cat$.
        
        The units are 
        \begin{align}
            \unitCoend[2][V] = \unitCoend[3][V] = \counit \otimes \id_V, 
        \end{align}
        where $\counit$ is the counit of $H$, and with $\Delta$ the comultiplication of
        $H$, the multiplications are given by\footnote{
    The convention we use for the dual map $\Delta^*$ on $H^* \otimes H^*$ is as follows:
    for $f,g \in H^*$ and $b \in H$ we set $(\Delta^*(f \otimes g))(b) := (f \otimes
    g)(\Delta(b))$.
    Ditto for $(\Delta\op)^*$.
	\label{fn:the-H*H*-convention}
}        
    \begin{align}
        \label{eq:Hopf-example-A23-product}
        \multCoend[2][V] = (\Delta\op)^* \otimes \id_V
        , \quad
        \multCoend[3][V] = \Delta^* \otimes \id_V\ .
    \end{align}

    The comultiplication of $\hopfMonad[i]$ is given by linear maps
    \begin{align}
        \comultCoend[i][V,W]: 
        H^*\otimes V\otimes W \to H^*\otimes V \otimes H^* \otimes W
    \end{align}
    for all $V,W\in \cat$ explicitly as follows. We have 
    \begin{align}
        \comultCoend[2][V, W](f \tensor v \tensor w)
        = \sum_{i,j} \langle f \mid e_i \cdot e_j \rangle \,  
        e^i \tensor v \tensor e^j \tensor w
    \end{align}
    and
    \begin{align}
        \comultCoend[3][V, W](f \tensor v \tensor w)
        = \sum_{i,j} \langle f \mid e_j \cdot e_i \rangle \,  
        e^i \tensor v \tensor e^j \tensor w
    \end{align}
    for $f\in H^*$.
    The counits, being morphisms $\hopfMonad[i][\tensUnit] \to \tensUnit$, can be
    identified with elements in $H$, and we find that they are given by the unit of $H$,
    \begin{align}
        \counitCoend[2]=\counitCoend[3]=\oneQ.
    \end{align}
    The left antipode is given by linear maps 
    \begin{align}
        \leftAntCoend[i][V]: H^* \otimes (H^* \otimes V)^* \to V^*,
    \end{align}
    for $V\in \cat,$ and, denoting by $\widetilde{\leftAntCoend[i]}(V)$ the corresponding
    endomorphism of $H\otimes V$, we have
    \begin{align}
        \widetilde{\leftAntCoend[2]}(V) = S \otimes \id_V
        \quad \text{and} \quad
        \widetilde{\leftAntCoend[3]}(V) = S\inv \otimes \id_V\ .
    \end{align}

    Assume now that $H$ is a pivotal Hopf algebra, i.e.\ that it contains a grouplike
    element $\pivotQ$, called the \emph{pivot}, satisfying $S^2(a)=\pivotQ a \pivotQ\inv$
    for all $a\in H$ \cite{AAGTV,BBGa}.
    The two remaining actions on $\hopfMonad(V)$ for $i=1,4$ can be chosen as
    \begin{align}
        \notag
        \actionCoend[1][h][f\otimes v]&=
        \langle f \mid S\inv(h\sweedler{1}) \placeholder h\sweedler{3} \rangle 
        \otimes h\sweedler{2}v
        \\
        \actionCoend[4][h][f\otimes v]&=
        \langle f \mid S(h\sweedler{3}) \placeholder h\sweedler{1} \rangle 
        \otimes h\sweedler{2}v \ . 
    \end{align}
    With this definition, the corresponding universal dinatural transformations are the
    same linear maps as before: 
        \begin{align}
            \dinatCoend[1][V][X] 
            = \dinatCoend[2][V][X] 
            \quad \text{and} \quad
            \dinatCoend[4][V][X] 
            = \dinatCoend[3][V][X]  \ .
            \label{eq:dinat_trans-Hopf-pivotal}
        \end{align}
        The counits are 
        \begin{align}
            \counitCoend[1]=\pivotQ\inv 
            \quad \text{and} \quad
            \counitCoend[4]=\pivotQ.
        \end{align}
     
    Rather than determining the Hopf monad structure on each $\hopfMonad[i]$ separately as
    stated before \Cref{prop:all-are-HM}, it may be easier to work out only
    one, say $\hopfMonad[2]$, and then to transport the structure via the isomorphisms
    $\kappa_{ij}$. 
    By \Cref{prop:iso_as_HC_pivotal}, this gives the same result.
    The $\kappa_{ij}$ take a simple form in the Hopf case:
        \begin{align}
        (\kappa_{12})_V (f \otimes v) 
        &= \langle f \mid \pivotQ\inv ~ \placeholder \rangle \otimes v
        \ , \notag \\ 
        (\kappa_{23})_V (f \otimes v) 
    &= \big( f \circ S \big) \otimes v 
        \ , \notag \\ 
        (\kappa_{43})_V (f \otimes v) 
    &= \langle f \mid \pivotQ ~ \placeholder \rangle \otimes v
        \ ,
        \end{align}
        for all $f \in H^*, v\in V$.
        Note that then e.g.\ $\counitCoend[1] = \counitCoend[2] \circ (\kappa_{12})_{\tensUnit}$.
\end{example}

\smallskip
    
\subsection{The distinguished invertible object}
    \label{sec:DistInvObj}
    
    Since $\cat$ is a finite tensor category, all projective objects are injective
    \cite[Prop.\,6.1.3]{EGNO}.
    In particular the socle (i.e.\ the maximal semisimple subobject) of the projective
    cover $P_U$ of a simple object $U$ is again simple.
    One can show that the socle $\distInvObj$ of $P_{\tensUnit}$ is an invertible object
    \cite[Lem.~6.4.1]{EGNO}, and we call it the \emph{distinguished invertible object of
    $\cat$}.\footnote{
        This means that our $D$ is in fact dual to the distinguished invertible object of
        \cite[Sec.~6.4]{EGNO}.
        However, our definition agrees with the one given in \cite[Sec.~6]{ENO-S4}.
    }
    This is equivalent to saying that 
    \begin{align}
        \dualL{P_{\tensUnit}} \cong P_{\dualL{D}}.
    \end{align}
    We call $\cat$ \emph{unimodular} if $D\cong \tensUnit$.

    \begin{example}\label{Example:DistInvHopf}
        Let $H$ be a finite-dimensional Hopf algebra.
        A left integral in $H$ is an element $c^l \in H$ such that $h c^l =
        \counit(h) c^l$ for all $h\in H$.
        It can be shown that left integrals are unique up to scalar, and thus there is a
        unique algebra morphism $\modulus: H\to \field$ such that 
        \begin{align}\label{eq:modulus_definition}
            c^l h = \modulus(h) c^l \quad \text{for all } h \in H .
        \end{align}
        This algebra morphism is called the \emph{modulus} of $H$, and by abuse of
        notation we denote the associated one-dimensional $H$-module again by $\modulus$.

        In \cite[Prop.~2.13]{EO-finiteTensor} it is shown that the distinguished
        invertible object $D$ of the category $\hmodM$ of finite-dimensional left
        $H$-modules is precisely the one-dimensional $H$-module with action given by    $\modulus\inv = \modulus \circ S$.
    \end{example}

\smallskip
    
\subsection{Monadic cointegrals for finite tensor categories}
    Consider the free $\hopfMonad$-module $(\hopfMonad[i][\distInvObj],
    \multCoend[i][\distInvObj])$.
    \begin{mydef}\label{def:main-def}
        For $i=2$ (resp.\ $i=3$), a morphism 
        \begin{align}\label{eq:categoricalCointegralsDefined}
            \coint_i : \tensUnit \to \hopfMonad[i][\distInvObj]
        \end{align}
        is called a \emph{right} (resp.\ \emph{left}) \emph{monadic cointegral of $\cat$}
        if it intertwines the trivial $\hopfMonad$-action on $\tensUnit$ and the free
        action on $\hopfMonad[i][\distInvObj]$.
        \\
        If $\cat$ is pivotal, then for $i=1$ (resp.\ $i=4$) such a morphism is called a
        \emph{right} (resp.\ \emph{left}) \emph{$\distInvObj$-symmetrised monadic
        cointegral of $\cat$}.
    \end{mydef}

    We denote the subspace of monadic cointegrals in
    $\cat(\tensUnit,\hopfMonad[i][\distInvObj])$ by:
    \begin{align}
        \label{eq:spacesOfMonadicCointegrals}
        \begin{array}{clcccl}
            i=1 : &
            \spaceRightSymMonCoint 
            &\hspace{3em}&
                    i=2 : &
            \spaceRightMonCoint 
            \\[0.5em]
                    i=4 : &
            \spaceLeftSymMonCoint 
            &&
                    i=3 : &
            \spaceLeftMonCoint 
        \end{array}
    \end{align}

\begin{samepage}
    \begin{rmrk}\label{remark:intertw-diagram-coint-transport}~
        \begin{enumerate}
            \item
                The $A_i$-module intertwining condition
                from~\eqref{eq:categoricalCointegralsDefined} for a morphism $\coint_i:
                \tensUnit \to \hopfMonad[i][\distInvObj]$ in $\cat$ is equivalent to the
                commutativity of the diagram
                \begin{equation}\label{diag:CatCointDiagram}
                    \begin{tikzcd}[row sep=large,column sep=large]
                        \hopfMonad[i][\tensUnit] \ar[r, "{\hopfMonad[i](\coint_i)}"] 
                        \ar[d,swap,"\counitCoend"]
                        & \hopfMonad[i]^2(\distInvObj)
                        \ar[d,"{\multCoend[i][\distInvObj]}"]
                        \\
                        \tensUnit \ar[r,swap,"\coint_i"]
                        & \hopfMonad[i][\distInvObj] 
                    \end{tikzcd}\ ,
                \end{equation}
                that is
                \begin{align}
                    \coint_i \circ \counitCoend 
                    = \multCoend[i][\distInvObj] \circ \hopfMonad(\coint_i)\ .
                    \label{eq:monCointEquation}
                \end{align}
            In \cite[Eq.~(45)]{BV-hopfmonads}, a cointegral of a bimonad $T$
                was defined as an intertwiner of $T$-modules from $(\tensUnit,T_0)$ to
                $(T(\tensUnit),\mu_{\tensUnit})$.
                Thus, if $\cat$ is unimodular, a right (resp.\ left) monadic cointegral of
                $\cat$ is just a cointegral of the bimonad $\hopfMonad[2]$ (resp.\
                $\hopfMonad[3]$).
            \item
                It follows immediately from  \Cref{prop:iso_as_HC_pivotal} and
                from the diagram~\eqref{diag:CatCointDiagram} that $\coint_i$ is a monadic
                cointegral for $\hopfMonad[i]$ if and only if $(\kappa_{i,j})_D \circ
                \coint_i$ is a monadic cointegral for $\hopfMonad[j]$.
        \end{enumerate}
    \end{rmrk}
\end{samepage}

    The names for the monadic cointegrals are chosen because of the relation to
    cointegrals for Hopf algebras, as we will see in the following
    example.\footnote{
        According to our convention of calling the invariants under the regular actions of
        a Hopf algebra \emph{integrals}, one could also call e.g.\ the right
        $\distInvObj$-symmetrised monadic cointegral simply an integral for the Hopf monad
        $\hopfMonad[1]$.
        This would follow more closely the nomenclature of \cite{BV-hopfmonads} 
        (who, however, call the invariants under the regular actions of a Hopf algebra
        ``cointegrals'', which is opposite to our convention). It would also fit to 
        \Cref{cor:OurCointIsShimizus}, which roughly states that monadic
        cointegrals are dual to the categorical cointegrals of \cite{Sh-integrals}.
    
        However, as we explain in \Cref{example:HopfAlgebra-catcoint} and
        \Cref{SEC:MainThm}, the reason for keeping these names is that  the four
        versions of monadic cointegrals automatically correspond to the four versions of
        cointegrals for $H$ if $\cat = \hmod$ for a pivotal (quasi) Hopf algebra $H$.
    }

    \begin{example}\label{example:HopfAlgebra-catcoint}
        Let $\cat=\hmodM$ be as in \Cref{example:HopfAlgebra-monads}.
        By \Cref{Example:DistInvHopf}, the distinguished invertible object
        $\distInvObj$ is just the ground field $\field$ with action given by the algebra
        morphism $\modulus\inv$, where $\modulus$ is the modulus of $H$.
        Thus, a morphism $\tensUnit \to \hopfMonad[i][\distInvObj]$ is the same as an
        element in $H^*$ intertwining some specific $H$-actions.

        Let us first look at the linear condition coming from diagram
        \eqref{diag:CatCointDiagram}.
        Using the Hopf monad structure as given in
        \Cref{example:HopfAlgebra-monads}, we see that a right (resp.\ left)
        monadic cointegral is, as a linear form, a solution to
        \begin{align}\label{eq:Hopf-example-mon-coint}
            (\coint_2 \otimes \id) (\Delta(h))
            = 
            \coint_2(h) \oneQ
            , \quad \text{resp.} \quad
            (\id \otimes \coint_3) (\Delta(h))
            = 
            \coint_3(h) \oneQ.
        \end{align}
        That is, it is a right (resp.\ left) cointegral for the Hopf algebra $H$ in the
        usual sense, cf.\ \cite[Def.~10.1.2]{Radford-Hopf-algebras}.
        Conversely, e.g.\ a solution to the first equation in
        \eqref{eq:Hopf-example-mon-coint} is a right monadic cointegral, provided it is in
        addition an intertwiner $\tensUnit \to \hopfMonad[2][\distInvObj]$ of $H$-modules.
        However, by \cite[Thm.~10.5.4(e)]{Radford-Hopf-algebras} a right cointegral
        $\coint$ satisfies $ \coint(a S\inv(b)) = \modulus\inv( b\sweedler{2} ) \coint(
        S(b\sweedler{1}) a)$.
        Applying~\eqref{eq:Hopf-example-action-A2-A3} for $A_2(D)$, we have
        \begin{align}\label{eq:from-sym-to-inter}
            \modulus\inv( h\sweedler{2} ) \coint( S(h\sweedler{1}) a h\sweedler{3} )
            =
            \coint( a h\sweedler{2} S\inv(h\sweedler{1}) )
            =
            \counit(h) \coint(a)\ ,
        \end{align}
        that is, the intertwining condition is automatically satisfied.

        If $(H,\pivotQ)$ is a pivotal Hopf algebra, then 
        diagram \eqref{diag:CatCointDiagram}
         can, as a linear equation, be evaluated for $i=1,4$,
        and it gives the equations
        \begin{align}\label{eq:Hopf-example-symmon-coint}
            (\coint_1 \otimes \id) (\Delta(h))
            = 
            \coint_1(h) \pivotQ\inv
            , \quad
            (\id \otimes \coint_4) (\Delta(h))
            = 
            \coint_4(h) \pivotQ.
        \end{align}
        According to \cite[Sec.~4.4]{FOG}, solutions to these equations are precisely
        $\modulus$-sym\-metrised cointegrals for $H$ (where we regard $H$ as a Hopf
        $G$-coalgebra for $G$ the trivial group), see also~\cite{BBGa} for the unimodular
        case.
        As above, in the converse direction, solutions to e.g.\ the first equation in
        \eqref{eq:Hopf-example-symmon-coint} are automatically intertwiners of $H$-modules
        from $\tensUnit$ to $\hopfMonad[1][\distInvObj]$.\footnote{
            To see this, note that by \cite[Prop.~4.18]{FOG} the linear form $\coint_1$
            lies in the space $\gammaSSymRS$ from \eqref{eq:Def_gammaSSym_pivotal}.
            This space is isomorphic to $\cat(\tensUnit, \hopfMonad[1][\distInvObj])$, and
            the isomorphism \eqref{eq:isoXC-pivotal} is the identity in the Hopf case.
    }
        Finally, let us note that $\modulus$-symmetrised cointegrals are an example of
        $g$-cointegrals for a group-like $g$ as introduced in \cite{Radford-integrals}
        (and called $g$-integrals there), see \cite[Rem.\ 3.10]{BGR1}.

        Let us stress a point already made in the introduction.
        As we just saw, via the very natural realisation of each monad $\hopfMonad[i]$
        given in \Cref{example:HopfAlgebra-monads}, the monadic cointegrals for
        $\hopfMonad[1],\dots,\hopfMonad[4]$ reduce to four known versions of cointegrals
        for finite dimensional (pivotal) Hopf algebras.
        This is an important motivation to keep all four of the $\hopfMonad[i]$, even
        though they are all isomorphic.
        Indeed, also in the Hopf case one can easily give explicit isomorphisms between
        the four spaces of cointegrals, but in practice it is important to have all four
        notions available, rather than  singling one out arbitrarily.
    \end{example}

    \medskip
    The preceding example shows that for $\cat=\hmodM$ with $H$ a finite-dimensional
    (pivotal) Hopf algebra, left/right ($\distInvObj$-symmetrised) monadic cointegrals
    exist and are unique up to scalar.    
    The next proposition states
    that this remains true for any (pivotal) finite tensor
    category.

\hyphenation{mo-na-dic}

    \begin{prop}
        \label{prop:MonCoint_uniq_ex}
        Let $\cat$ be a finite tensor category. 
        \begin{enumerate}
            \item         
                Non-zero left/right monadic cointegrals exist and are unique up to scalar
                multiples.
            \item
                Suppose $\cat$ is in addition pivotal.
                Then non-zero left/right $\distInvObj$-symmetrised monadic cointegrals
                exist and are unique up to scalar multiples.
        \end{enumerate}
    \end{prop}

    The proof will follow from results in \cite{Sh-integrals}, after we relate monadic
    cointegrals to the categorical cointegral of \cite{Sh-integrals}, and is given at the
    end of the next subsection.

\smallskip
    
\subsection{Relation to the categorical cointegral}\label{subsec:comonad-Z}
    Define functors $\hopfComonad$ via the ends
    \begin{align}
        \hopfComonad[1][V] = \int_{X\in \cat} \hopfComonadComp{1} , 
        \quad \quad \quad
        \hopfComonad[2][V] = \int_{X\in \cat} \hopfComonadComp{2}
        \notag \\
        \hopfComonad[3][V] = \int_{X\in \cat} \hopfComonadComp{3} , 
        \quad \quad \quad
        \hopfComonad[4][V] = \int_{X\in \cat} \hopfComonadComp{4}
    \end{align}
    with corresponding universal dinatural transformations $\dinatEnd[i][V]$, so that for
    example
    \begin{align}
        \dinatEnd[4][V][X]:  \hopfComonad[4][V] \to \hopfComonadComp{4} \ .
    \end{align}
    Below we will give an adjunction between $\hopfComonad[4]$ and $\hopfMonad[2]$. One
    can formulate such adjunctions in the three other cases, too, but we will not need
    this and will only consider $\hopfComonad[4]$ in the following.

    Similarly to how the $\hopfMonad$, $i=2,3$ became Hopf monads, 
    $\hopfComonad[4]$ becomes a Hopf comonad and we denote the comultiplication, counit, 
    multiplication, and unit by 
    \begin{alignat}{3}
        \label{eq:structureMapsEnd}
        \comultEnd[4][V]&: \hopfComonad[4][V] \to 
  	  	\hopfComonad[4]\hopfComonad[4](V)
        , \quad
        &&\counitEnd[4][V]&&: \hopfComonad[4][V] \to V,
        \notag \\
        \multEnd[4][V,W]&: \hopfComonad[4][V] \otimes \hopfComonad[4][W]
        \to \hopfComonad[4][V\otimes W]
        , \quad
        &&\quad \unitEnd[4]&&: \tensUnit \to \hopfComonad[4][\tensUnit]
        \ ,
    \end{alignat}
    respectively.
    $\hopfComonad[4]$ is precisely the central comonad of \cite{Sh-integrals}, where also
    a detailed description of the structure maps \eqref{eq:structureMapsEnd} can be found.

    We can now recall the definition of the \emph{categorical cointegral} from
    \cite[Def.~4.3]{Sh-integrals}:
    It is a $\hopfComonad[4]$-comodule morphism
    \begin{align}
        \label{eq:ShCointegral}
    \coint^{\textup{Sh}}: 
        (\hopfComonad[4][\dualL{\distInvObj}], 
        \comultEnd[4][\dualL{\distInvObj}]) \to \tensUnit
    \end{align}
    from the cofree comodule on $\dualL{\distInvObj}$ to the tensor unit
    considered as the trivial comodule.\footnotemark
    \footnotetext{
        Although Shimizu's definition is not explicitly stated this way, it is easy to see
        that \cite[Def.~4.3]{Sh-integrals} and \eqref{eq:ShCointegral} are equivalent.
        This is also mentioned in the proof of \cite[Thm.~4.8]{Sh-integrals}.
    }

    \medskip

    To relate the two notions \emph{categorical cointegral} and \emph{monadic cointegral},
    we observe that there is an adjunction $\hopfMonad[2] \dashv \hopfComonad[4]$, i.e.\
    the central Hopf monad $\hopfMonad[2]$ is left adjoint to $\hopfComonad[4]$.
    Indeed,
    \begin{align}\label{eq:HomSetAdj-AZ}
        \notag
        \cat(\hopfMonad[2][V], W) 
        &\cong \operatorname{Dinat}( \hopfMonadComp{2}[-][V], W ) \\
        &\cong \operatorname{Dinat}( V, \hopfComonadComp{4}[-][W] )
        \cong \cat( V, \hopfComonad[4][W] ).
    \end{align}
    We denote the counit and unit of this adjunction by 
    \begin{align}
        \counitAdj: \hopfMonad[2]\hopfComonad[4]\Rightarrow \id_{\cat},
        \quad 
        \unitAdj : \id_{\cat} \Rightarrow \hopfComonad[4] \hopfMonad[2]
    \end{align}
    respectively.
    They can easily be deduced from \eqref{eq:HomSetAdj-AZ}; for example
    \begin{equation}
        \ipic{-0.5}{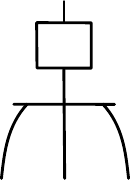}{1.4}
        \put (-32,040) {$V$}
        \put (-34,015) {$\counitAdj_V$}
        \put (-04,-07) {$\dinatCoend[2]$}
        \put (-58,-50) {$\dualL{X}$}
        \put (-38,-50) {$\hopfComonad[4]V$}
        \put (-06,-50) {$X$}
        \quad 
        = 
        \quad
        \ipic{-0.5}{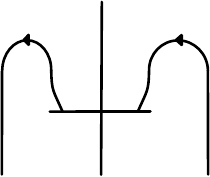}{1.4}
        \put (-50,038) {$V$}
        \put (-22,-13) {$\dinatEnd[4]$}
        \put (-91,-48) {$\dualL{X}$}
        \put (-54,-48) {$\hopfComonad[4]V$}
        \put (-07,-48) {$X$}
    \end{equation}
    determines the counit.

    \medskip
    For a comonad $M$ on $\cat$, the category of comodules is denoted by $\cat^M$.

    \begin{lemma}
        \label{Lem:EquivalenceOfModulesComodules}
        The functor $F:\cat^{\hopfComonad[4]} \to \cat_{\hopfMonad[2]}$, given on objects
        and morphisms by
        \begin{align}\label{eq:Equivalence_Zcomod_Amod}
            F(V,\rho) = (V, \counitAdj_V \circ \hopfMonad[2][\rho]), 
            \quad\quad
            Ff = f,
        \end{align}
        is an equivalence.
    \end{lemma}
    \begin{proof}
        This statement follows immediately from the fact that $\hopfMonad[2]$ is
        left adjoint to $\hopfComonad[4]$.
        The inverse equivalence $G: \cat_{\hopfMonad[2]} \to \cat^{\hopfComonad[4]}$ is
        given on objects and morphisms by 
        \begin{align}\label{eq:InvEquivalence_Zcomod_Amod}
            G(V,\nu) = (V, \hopfComonad[4][\nu] \circ \unitAdj_V), 
            \quad\quad
            Gf = f,
        \end{align}
        and a simple check using the adjunction triangles proves the claim.
    \end{proof}

    We make the following observation.
    \begin{prop}\label{prop:OurCointegralsAreShimizus}
        \label{prop:iso_as_A_modules}
        Let $F$ be as in \Cref{Lem:EquivalenceOfModulesComodules}. There is an
        isomorphism
        \begin{align}
            (\hopfMonad[2][V], \multCoend[2][V]) \cong
            \dualL{\big( 
                F( \hopfComonad[4][\dualR{V}], \comultEnd[4][\dualR{V}])
                \big)}
        \end{align}
        of $\hopfMonad[2]$-modules, natural in $V\in\cat$.
    \end{prop}

    \begin{proof}
        Abbreviate
        \begin{align}
            \widetilde{V} = \hopfComonad[4][\dualR{V}]
            \quad \text{and} \quad 
            \rho_{\widetilde{V}} = \comultEnd[4][\dualR{V}]: \widetilde{V} \to 
            \hopfComonad[4]\big( \widetilde{V} \big).
        \end{align}
        Under the equivalence from \Cref{Lem:EquivalenceOfModulesComodules} we have
        \begin{align}
            F(\widetilde{V}, \rho_{\widetilde{V}}) = (\widetilde{V},
            \sigma_{\widetilde{V}})
        \end{align}
        where $\sigma_{\widetilde{V}} = \counitAdj_{\widetilde{V}} \circ
        \hopfMonad[2](\rho_{\widetilde{V}}) : \hopfMonad[2](\widetilde{V}) \to
        \widetilde{V}$ is the $\hopfMonad[2]$-action corresponding to the free coaction.

        Define the natural isomorphism $E_V: \hopfMonad[2][V] \to \dualL{\widetilde{V}}$
        by
        \begin{equation}
            E_V \circ \dinatCoend[2][X][V]
            = ~
            \ipic{-0.5}{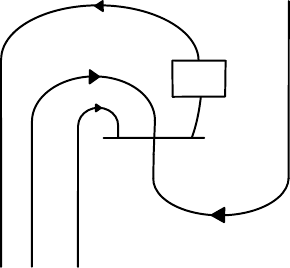}{1.5}
            \put (-007,064) {$\dualL{\widetilde{V}}$}
            \put (-046,022) {$\omega_X$}
            \put (-055,-15) {$\dinatEnd[4][\dualR{V}][\dualR{X}]$}
            \put (-134,-70) {$\dualL{X}$}
            \put (-115,-70) {$V$}
            \put (-097,-70) {$X$}
        \end{equation}
        for $V\in\cat$. Here $\omega_X : \dualL{(\dualR{X})} \to X$ denotes the natural
        isomorphism from \eqref{eq:definitionOmegaIso}.

        We want to show that $E_V$ is an $\hopfMonad[2]$-module map, that is
        \begin{align}
            \label{eq:E-isAnAIntertwiner}
            E_V \circ \multCoend[2][V] 
            =  
            \leftAntCoend[2][\widetilde{V}]
            \circ \hopfMonad[2][\dualL{\sigma_{\widetilde{V}}} \circ E_V],
        \end{align}
        where we also used the action~\eqref{eq:rho-dual-def} on the dual
        $\hopfMonad[2]$-module.
        To check that this equality holds we establish that both sides of
        \eqref{eq:E-isAnAIntertwiner} satisfy the same universal property for the iterated coend $\hopfMonad[2]\hopfMonad[2]$.
        For the left hand side of \eqref{eq:E-isAnAIntertwiner} we get 
        \begin{center}
            \begin{minipage}{0.40\textwidth}
                \begin{align*}
                    E_V 
                    &\circ \multCoend[2][V]
                    \circ \dinatCoend[2][\hopfMonad[2]V][Y] \\
                    &
                    \circ \big(\id_\dualL{Y} \otimes ( \dinatCoend[2][V][X] \otimes \id_Y
                    )\big)
                    = 
                    \hspace*{-1cm}
                \end{align*}
            \end{minipage}
            \begin{minipage}{0.50\textwidth}
                \begin{equation}
                    \label{eq:BigUglyLHS}
                    \ipic{-0.5}{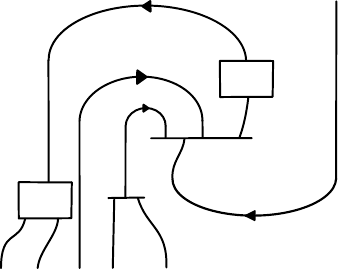}{1.5}\ .
                    \put (-013,064) {$\dualL{\widetilde{V}}$}
                    \put (-054,022) {$\omega_X$}
                    \put (-072,-15) {$\dinatEnd[4][\dualR{V}][\dualR{(XY)}]$}
                    \put (-138,-30) {$\gamma$}
                    \put (-160,-70) {$\dualL{Y}$}
                    \put (-144,-70) {$\dualL{X}$}
                    \put (-125,-70) {$V$}
                    \put (-110,-70) {$X$}
                    \put (-085,-70) {$Y$}
                \end{equation}
            \hfill
            \end{minipage}
        \end{center} 
        The right hand side of \eqref{eq:E-isAnAIntertwiner} composed with the same
        dinatural transformation immediately yields
        \begin{align}
            \label{eq:BigUglyRHS-1}
            \ipic{-0.5}{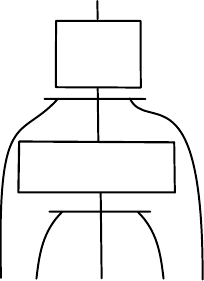}{1.5}
            \put (-055,065) {$\dualL{(\widetilde{V})}$}
            \put (-022,015) {$\dinatCoend[2]$}
            \put (-062,034) {$\leftAntCoend[2][\widetilde{V}]$}
            \put (-078,-15) {$\dualL{(\sigma_{\widetilde{V}})} \circ E_V$}
            \put (-021,-33) {$\dinatCoend[2]$}
            \put (-095,-72) {$\dualL{Y}$}
            \put (-078,-72) {$\dualL{X}$}
            \put (-049,-72) {$V$}
            \put (-023,-72) {$X$}
            \put (-005,-72) {$Y$}
            \quad = \quad
        \ipic{-0.5}{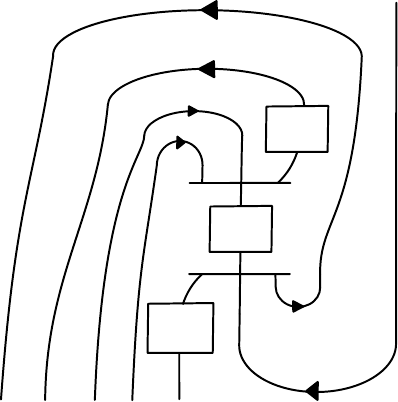}{1.3}
        \put (-005,078) {$\dualL{\widetilde{V}}$}
        \put (-045,024) {$\omega_X$}
        \put (-039,002) {\footnotesize $\dinatEnd[4]$}
        \put (-067,-13) {$\sigma_{\widetilde{V}}$}
        \put (-039,-30) {\footnotesize $\dinatCoend[2]$}
        \put (-090,-52) {$\omega\inv_Y$}
        \put (-156,-87) {$\dualL{Y}$}
        \put (-139,-87) {$\dualL{X}$}
        \put (-119,-87) {$V$}
        \put (-105,-87) {$X$}
        \put (-087,-87) {$Y$}
        \end{align}
        A simple calculation shows 
        \begin{align}
            \dinatEnd[4][\dualR{V}][\dualR{X}]
            \circ \sigma_{\widetilde{V}}
            \circ
            \dinatCoend[2][\widetilde{V}][\dualR{Y}]
            =~
            \ipic{-0.5}{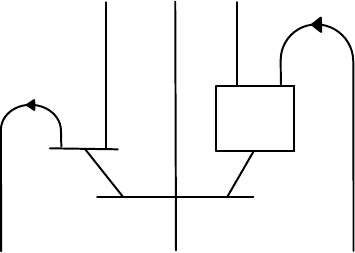}{1.4}\ ,
            \put (-120,055) {$\dualR{X}$}
            \put (-093,055) {$\dualR{V}$}
            \put (-070,055) {$\dualL{(\dualR{X})}$}
        \put (-058,-00) {$\gamma\inv$}
            \put (-047,-35) {$\dinatEnd[4]$}
            \put (-105,-05) {$\id$}
            \put (-165,-62) {$\dualL{(\dualR{Y})}$}
            \put (-090,-65) {$\widetilde{V}$}
            \put (-019,-62) {$\dualR{Y}$}
        \end{align}
        and we thus get 
        \begin{align}
            \label{eq:BigUglyRHS-2}
            \eqref{eq:BigUglyRHS-1}~
            =
            \ipic{-0.5}{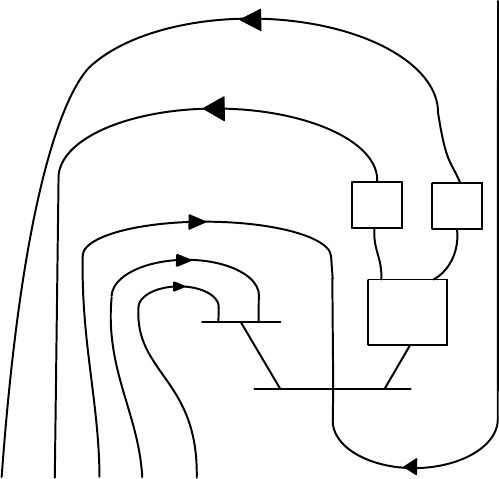}{1.2}\ .
            \put (-018,088) {$\dualL{(\widetilde{V})}$}
            \put (-054,010) {$\omega$}
            \put (-025,010) {$\omega$}
        \put (-048,-28) {$\gamma\inv$}
            \put (-055,-60) {$\dinatEnd[4]$}
            \put (-190,-94) {$\dualL{Y}$}
            \put (-174,-94) {$\dualL{X}$}
            \put (-152,-94) {$V$}
            \put (-137,-94) {$X$}
            \put (-117,-94) {$Y$}
        \end{align}

        Let $\gamma^r$ be the analogue of $\gamma$ (defined in
        \eqref{eq:gammaCanonicalIso}) for right duals.
        One checks that the diagram 
        \begin{equation}
            \begin{tikzcd}[column sep=large, row sep=large]
                \dualL{(\dualR{X} \otimes \dualR{Y})}
                \ar[r,"{\gamma\inv_{\dualR{Y},\dualR{X}}}"]
                \ar[d,"{\dualL{\big((\gamma^r)\inv_{Y,X}\big)}}",swap]
                &
                \dualL{(\dualR{Y})} \otimes \dualL{(\dualR{X})} 
                \ar[d,"{\omega_Y \otimes \omega_X}"]
                \\
                \dualL{(\dualR{(Y\otimes X)})} 
                \ar[r,swap,"{\omega_{Y\otimes X}}"]
                &
                Y\otimes X
            \end{tikzcd}
        \end{equation}
        commutes.

        Dinaturality of $\dinatEnd[4]$ then implies
        \begin{align}
            \ipic{-0.5}{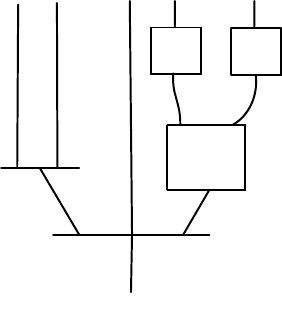}{1.2}
            \put (-101,058) {\small $\dualR{X}$}
            \put (-085,058) {\small $\dualR{Y}$}
            \put (-056,058) {\small $V$}
            \put (-046,058) {\small $\dualL{(\dualR{Y})}$}
            \put (-018,058) {\small $\dualL{(\dualR{X})}$}
            \put (-041,035) {$\omega$}
            \put (-013,035) {$\omega$}
            \put (-068,000) {$\id$}
            \put (-035,-03) {$\gamma\inv$}
            \put (-035,-35) {$\dinatEnd[4]$}
            \put (-065,-56) {\small $\hopfComonad[4][V]$}
            \quad = \quad
            \ipic{-0.5}{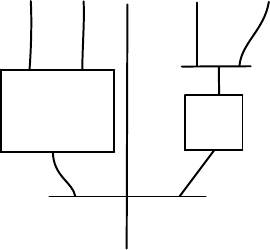}{1.2}.
            \put (-096,046) {\small $\dualR{X}$}
            \put (-077,046) {\small $\dualR{Y}$}
            \put (-056,046) {\small $V$}
            \put (-044,046) {\small $\dualL{(\dualR{Y})}$}
            \put (-014,046) {\small $\dualL{(\dualR{X})}$}
            \put (-004,018) {$\id$}
            \put (-092,001) {${(\gamma^r)}\inv$}
            \put (-027,-03) {$\omega$}
            \put (-035,-35) {$\dinatEnd[4]$}
            \put (-065,-56) {\small $\hopfComonad[4][V]$}
        \end{align}
        After plugging this into \eqref{eq:BigUglyRHS-2} and substituting the definitions
        of ${\gamma^r}, \gamma,$ and $\omega$, we see that this agrees with
        \eqref{eq:BigUglyLHS}.
    \end{proof}

    Combining \Cref{Lem:EquivalenceOfModulesComodules} and
    \Cref{prop:iso_as_A_modules} and using $\dualR{D}\cong\dualL{D}$ (this
    holds for all invertible objects), we get
    \begin{samepage}
    \begin{cor}
        \label{cor:OurCointIsShimizus}
        There is an isomorphism
        \begin{align}
            \cat_{\hopfMonad[2]}
            \big( \tensUnit, 
                ( \hopfMonad[2][\distInvObj], \multCoend[2][\distInvObj] )
            \big)
            \cong
            \cat^{\hopfComonad[4]}
            \big(
                ( \hopfComonad[4][\dualL{\distInvObj}], 
                \comultEnd[4][\dualL{\distInvObj}] )
                , \tensUnit
            \big)
            \ .
        \end{align}
    \end{cor}
    \end{samepage}

    \medskip
    After these preparations, we can show the existence and uniqueness (up to scalar) of
    monadic cointegrals.

    \begin{proof}[Proof of \Cref{prop:MonCoint_uniq_ex}]
        By the preceding corollary, the right monadic cointegral is equivalent to the
        categorical cointegral \eqref{eq:ShCointegral} of \cite{Sh-integrals}.
        Existence and uniqueness of categorical cointegrals were established in
        \cite[Thm.\ 4.8]{Sh-integrals}.
        The claim then follows
        from \Cref{remark:intertw-diagram-coint-transport}\,(2).
   \end{proof}

\begin{rmrk}\label{rem:H*mod-Hcomod}
    Recall that the definition of integrals and cointegrals in the Hopf case is symmetric
    under duality.
    More precisely, if $H$
    is a finite-dimensional Hopf algebra, then a left cointegral
    for $H$ is the same as a morphism $\coint^l: H \to \tensUnit$ in the category of left
    $H$-comodules, where we regard $H$ as the coregular comodule.
    Equivalently, we can consider the dual Hopf algebra $H^*$ (with the structure given by
    transposition of that of $H$, cf.\ \cref{fn:the-H*H*-convention}).        
    In this case it is a morphism $\coint^l: \tensUnit \to H^*$ in the category of left
    $H^*$-modules, where we regard $H^*$ as the regular module.

    Taking duals provides a (contravariant) equivalence,
    \begin{equation}
        \label{eq:cointegral-duality-Hopf-2}
        H\text{-comod} \cong H^*\text{-mod}
    \end{equation}
    and in particular we have an isomorphism 
    \begin{align}
        \label{eq:cointegral-duality-Hopf}
        (H\text{-comod})(H, \tensUnit) 
        \cong 
        (H^*\text{-mod})(\tensUnit, H^*)
    \end{align}
    of vector spaces. 
    We will need one more observation. Abbreviate $\cat = H$-mod and recall the
    computation \eqref{eq:from-sym-to-inter}.
    This showed that there is an $H$-module structure on $H^*=\hopfMonad[3][D]$ and
    similarly on $H=\hopfComonad[4][\dualL{D}]$ (which are not the (co)regular) such that:
    \begin{equation}
        \label{eq:cointegral-duality-Hopf-in-Hmod}
        \begin{array}{ccc}
            (H\text{-comod})(H, \tensUnit)
            &\cong&
            (H^*\text{-mod})(\tensUnit, H^*)
            \\
            \cap && \cap
            \\
            \cat(\hopfComonad[4][\dualL{D}], \tensUnit)
            &&
            \cat(\tensUnit, \hopfMonad[3][D])
        \end{array}
    \end{equation} 

    For quasi-Hopf algebras $H$ the corresponding line of reasoning to relate integrals
    and cointegrals fails at the outset, as $H^*$ is not again a quasi-Hopf algebra.
    Instead, there is the following categorical version of it. 
    We have a (contravariant) equivalence given by the composition of equivalences
    \begin{align}
        \label{eq:cointegral-duality-Cat-2}
        \cat^{\hopfComonad[4]}
        \overset{ \text{Lem.}~\ref{Lem:EquivalenceOfModulesComodules} }{\cong}
        \cat_{\hopfMonad[2]}
        \overset{ \text{dualising} }{\cong}
        \cat_{\hopfMonad[2]}
        \overset{ \text{Prop.}~\ref{prop:iso_as_HC_pivotal} }{\cong}
        \cat_{\hopfMonad[3]} \ .
    \end{align}
    \Cref{cor:OurCointIsShimizus} then implies
        \begin{align}
            \label{eq:cointegral-duality-Cat}
            \cat^{\hopfComonad[4]}
            \big(
                \,( \hopfComonad[4][\dualL{\distInvObj}],
                \comultEnd[4][\dualL{\distInvObj}] ),\,
                \tensUnit\,
            \big) 
            ~\cong~
            \cat_{\hopfMonad[3]}
            \big(
                \,\tensUnit,\,
                (\hopfMonad[3][\distInvObj], \multCoend[3][\distInvObj] )
            \,\big)
            \ .
        \end{align}
    To relate the right hand side of \eqref{eq:cointegral-duality-Cat} in the Hopf case to
    that of \eqref{eq:cointegral-duality-Hopf-in-Hmod}, one uses the explicit form of the
    monad multiplication in \eqref{eq:Hopf-example-A23-product}.
    For the left hand side, one correspondingly uses the coproduct of $\hopfComonad[4]$,
    we omit the details.

    The categorical version \eqref{eq:cointegral-duality-Cat} of the isomorphism
    \eqref{eq:cointegral-duality-Hopf-in-Hmod} 
    provides a more conceptual reason
    for the relation between the left monadic cointegral and the categorical cointegral
    of~\cite{Sh-integrals}.
\end{rmrk}

\subsection{Rewriting the monadic cointegral via \texorpdfstring{$\hopfComonad[4]$}{Z4}}
In \Cref{cor:OurCointIsShimizus} we saw one way to rewrite the definition of
monadic cointegrals in terms of Hom-spaces in $\cat^{\hopfComonad[4]}$.
We will later need the more direct relation we present here.

Under the equivalence from \Cref{Lem:EquivalenceOfModulesComodules}, specifically
\eqref{eq:InvEquivalence_Zcomod_Amod}, we can map the free $\hopfMonad[2]$-module on any
$U \in \cat$ to its corresponding $\hopfComonad[4]$-comodule
\begin{align}\label{eq:coactionLeftAdjoint}     
        (\hopfMonad[2](U), R_U) \text{ with }
        R_U =
    \left[ 
        \hopfMonad[2][U]
        \xrightarrow{\unitAdj_{\hopfMonad[2][U]}}
        \hopfComonad[4]\hopfMonad[2]^2(U)
        \xrightarrow{\hopfComonad[4][\multCoend[2][U]]}
        \hopfComonad[4]\hopfMonad[2][U]
    \right]\ .
\end{align}
Note that this assignment is in fact natural in $U$, i.e.\ we have a natural
transformation 
\begin{align}\label{eq:categoricalCoaction}
    R: \hopfMonad[2] \Rightarrow \hopfComonad[4]\hopfMonad[2], 
\end{align}
which we call the \emph{categorical coaction}.\footnote{
    The categorical coaction turns $\hopfMonad[2]$ into a $\hopfComonad[4]$-comodule in
    $\End(\cat)$.
}
Since we have equivalence \eqref{eq:Equivalence_Zcomod_Amod}, we can see the equality
\begin{align}
    \cat_{\hopfMonad[2]}\big(\tensUnit,\, (\hopfMonad[2][\distInvObj],
    \multCoend[2][\distInvObj]\, \big)
    ~=~
    \cat^{\hopfComonad[4]}\big( \tensUnit, (\hopfMonad[2][\distInvObj],
    R_{\distInvObj})\big)
\end{align}
of subspaces of $\cat\big( \tensUnit, \hopfMonad[2][\distInvObj]\!\big)$.
An element in the subspace on the left hand side is by definition a right monadic
cointegral.
Spelling out the condition to be in the subspace on the right hand side proves the
following lemma (recall from \eqref{eq:structureMapsEnd} the notation $\unitEnd[4]$ for
the unit of $\hopfComonad[4]$):

\begin{lemma}\label{lem:monCointViaCatCoaction}      
    A morphism $\coint:\tensUnit \to \hopfMonad[2][\distInvObj]$ in $\cat$ is a right
    monadic cointegral if and only if 
	\begin{align}\label{eq:MonadicCointegralsViaCoaction}
		R_{\distInvObj} \circ \coint = \hopfComonad[4][\coint] \circ \unitEnd[4] \ .
	\end{align}
\end{lemma}

\medskip

\section{Cointegrals for quasi-Hopf algebras}\label{SEC:Coints_qHopf}
In this section we first set out our conventions for quasi-Hopf algebras,
which agree with those used in \cite{FGR1,BGR1}.
Then, we specialise monadic cointegrals to the category of modules over a quasi-Hopf
algebra, and we recall the definition of cointegrals from \cite{HN-integrals, BC2-2011}.

Throughout this section, let $H$ be a finite-dimensional quasi-Hopf algebra over~$\field$.

\subsection{Conventions and definitions}\label{sec:Hopf-conf-def}
    The antipode of $H$ is denoted by $S$, and the coassociator and its inverse by 
    $\coassQ$ and $\invCoassQ$, respectively.
    For the coproduct, the counit, and the unit we use the standard notations
    $\Delta$, $\counit$, $\oneQ$.
    The evaluation and coevaluation elements are $\alphaQ$ and $\betaQ$, and without loss
    of generality we assume that $\counit(\alphaQ) = 1 = \counit(\betaQ)$.

    For elements in tensor powers of $H$, e.g.\ $a\in H\otimes H$, we write
    \begin{align}
        a = a_1 \otimes a_2
        , \quad
        a_{21} = \tau(a) = a_2 \otimes a_1,
    \end{align}
    where $\tau$ is the tensor flip of vector spaces.
    This notation is extended to higher tensor powers of $H$.

    The components of the coassociator and its inverse will be written with upper-case
    Latin letters and lower-case Latin letters like 
    \begin{align}
        \coassQ = \coassQ[1]_1 \otimes \coassQ[1]_2 \otimes \coassQ[1]_3
        \quad \text{and} \quad
        \invCoassQ = \invCoassQ[1]_1 \otimes \invCoassQ[1]_2 \otimes \invCoassQ[1]_3,
    \end{align}
    respectively. 
    For further copies of $\coassQ$ and $\invCoassQ$ in an expression different
    letters like $\coassQ[2], \invCoassQ[2]$ are used.
    We remark that our conventions for $\Phi$ differ from those
    of e.g.\ \cite{HN-integrals,BC2-2011}, in that we use the quasi-coassociativity
    condition
    \begin{align}\label{eq:quasiCoassociativity}
        (\Delta\otimes \id)\left( \Delta(h) \right) \cdot \coassQ
        =
        \coassQ\cdot (\id\otimes \Delta)\left( \Delta(h) \right)
    \end{align}
    for all $h\in H$. 
	Thus, the $\coassQ$ in \cite{HN-integrals,BC2-2011} is our $\invCoassQ$.
    The pentagon axiom is 
    \begin{align}
        \label{eq:pentagon-qHopf}
        (\Delta \otimes \id \otimes \id)(\coassQ) 
        \cdot (\id \otimes \id \otimes \Delta)(\coassQ) 
        = (\coassQ \otimes \oneQ) 
        \cdot (\id \otimes \Delta \otimes \id)(\coassQ)
        \cdot (\oneQ \otimes \coassQ).
    \end{align}

    A sumless Sweedler-notation for (iterated) coproducts is used, i.e.\
    \begin{align}
        \Delta(h) = h\sweedler{1} \otimes h\sweedler{2} 
        , \quad
        (\Delta\otimes \id)\left( \Delta(h) \right)
        = h\sweedler{1,1}\otimes h\sweedler{1,2} \otimes h\sweedler{2}
    \end{align}
    for $h\in H$.

    With the above notation we can write the antipode axioms as
    \begin{align}
        \label{eq:antipodeAxioms}
        S(h\sweedler{1}) \alphaQ h\sweedler{2}
        = \counit(h) \alphaQ
        , \quad
        h\sweedler{1} \betaQ S(h\sweedler{2})
        = \counit(h) \betaQ \ ,
        \quad h\in H,
    \end{align}
    and
    \begin{align}
        \label{eq:zigzag-qHopf}
        S(\coassQ[1]_1) \alphaQ \coassQ[1]_2 \betaQ S(\coassQ[1]_3) = \oneQ
        , \quad
        \invCoassQ[1]_1 \betaQ S(\invCoassQ[1]_2) \alphaQ \invCoassQ[1]_3 = \oneQ.
    \end{align}
    The latter are also referred to as \emph{zig-zag axioms}.

    \smallskip

    By considering either the opposite multiplication or the opposite comultiplication
    we get new quasi-Hopf algebras $H\op$ and $H\cop$, with the quasi-Hopf structure
    given by $S\op = S\cop = S\inv$, $\coassQ\op = \invCoassQ$, $\alphaQ\op =
    S\inv(\betaQ)$, $\betaQ\op = S\inv(\alphaQ)$, $\coassQ\cop = \Psi_{321}$,
    $\alphaQ\cop = S\inv(\alphaQ)$, and $\betaQ\cop = S\inv(\betaQ)$.

    We will make frequent use of the hook notation, for $h\in H$, $f\in H^*$,
    \begin{align}\label{eq:hookActionDefined-H}
        h \rightharpoonup f = \langle f \mid \placeholder h \rangle
        \quad , \quad
        f \leftharpoonup h = \langle f \mid h \placeholder \rangle
    \end{align}
    and 
    \begin{align}\label{eq:hookActionDefined-H-dual}
        f \rightharpoonup h = h\sweedler{1} f(h\sweedler{2})
        \quad , \quad
        h \leftharpoonup f = f(h\sweedler{1}) h\sweedler{2}.
    \end{align}

\subsection{Modules}
    Finite-dimensional left $H$-modules and their intertwiners form the category $\hmodM$,
    and we now recall its structure, following~\cite{BC2-2011}.
    This category is a finite tensor category; 
    the tensor product of two objects $V,W\in\hmodM$ is the vector space
    $V\otimes_{\field} W$ equipped with the diagonal $H$-action.
    The left dual $\dualL{V}$ and the right dual $\dualR{V}$ of a module $V$ are both
    modelled on the dual vector space $V^*$, and $h\in H$ acts on $v^*\in \dualL{V}$
    via
    \begin{align}
        \langle h.v^*, w \rangle = 
        \langle v^* \leftharpoonup S(h) \mid w \rangle =
        \langle v^*, S(h) w \rangle
    \end{align}
    and on $v^*\in \dualR{V}$ via
    \begin{align}
        \langle h.v^*, w \rangle = 
        \langle v^* \leftharpoonup S\inv(h) \mid w \rangle = 
        \langle v^*,  S\inv(h)w \rangle.
    \end{align}
    The left and right evaluations are
    \begin{align}
        \notag
        \evL_V:\dualL{V}\otimes V&\to \tensUnit,
        \quad 
        \evL_V(v^* \tensor w) = \langle v^*, \alphaQ w \rangle 
        \\
        \evR_V:V\otimes \dualR{V}&\to \tensUnit,
        \quad 
        \evR_V(w \tensor v^*) = \langle v^*, S\inv(\alphaQ) w \rangle ,
    \end{align}
    and the left and right coevaluations are
    \begin{align}
        \notag
        \coevL_V:\tensUnit&\to V\otimes \dualL{V},
        \quad 
        \coevL_V(1) = \sum_i \betaQ v_i \otimes v^i
        \\
        \coevR_V:\tensUnit&\to \dualR{V}\otimes V,
        \quad 
        \coevR_V(1) = \sum_i v^i \otimes S\inv(\betaQ) v_i,
    \end{align}
    which we expressed in terms of a basis $\{v_i\}$ of $V$, with dual basis $\{v^i\}$.

    \smallskip
    Recall the natural isomorphism $\gamma_{V,W}$ from \eqref{eq:gammaCanonicalIso}.
    There is a unique invertible element $\Dt\in H\otimes H$, called the \emph{Drinfeld
    twist}, such that
    \begin{align}
        \gamma_{V,W}(f \otimes g)(w\otimes v) = g(\Dt_1 w) f(\Dt_2 v)
    \end{align}
    for all $f \in V^*, v\in V$, $g\in W^*, w\in W$, see \cite{Dr} and e.g.\
    \cite[Lem.~6.7]{FGR1}.
    The statement that $\gamma_{V,W}$ be an intertwiner translates to the equation
    \begin{align}\label{eq:DrinfeldTwistProperty}
        \Dt\cdot \Delta(S(a)) \cdot \Dt\inv 
        =
        (S\otimes S) \left( \Delta\cop(a) \right)\ .
    \end{align}

    Analogously, we can express the natural isomorphism
    ${\gamma^r}_{V,W}:\dualR{V}\otimes \dualR{W} \to \dualR{(WV)}$, obtained by
    mirroring the diagram in \eqref{eq:gammaCanonicalIso} at the vertical axis, by using
    what we call the \emph{Drinfeld twist for right duals} $\Dt^r$:
    \begin{align}
        {\gamma^r}_{V,W}(f \otimes g)(w\otimes v) 
        = g(\Dt^r_1 w) f(\Dt^r_2 v)
    \end{align}
    for all $f \in V^*, v\in V$, $g\in W^*, w\in W$.

    Both $\Dt$ and $\Dt^r$ can be written in terms of the defining data of the quasi-Hopf
    algebra $H$, and we will give their explicit form 
    in \Cref{sec:DrinfeldTwist-closed}.

\medskip

    If $\hmodM$ is pivotal, the quasi-Hopf algebra $H$ is called \emph{pivotal}.
    The pivotal structure on $\hmodM$ corresponds uniquely to an element $\pivotQ \in H$
    called the \emph{pivot} satisfying 
    \begin{align}
        \label{eq:pivotConditions}
        \Delta(\pivotQ) = \Dt\inv \cdot (S\otimes S)(\Dt_{21}) \cdot (\pivotQ \otimes
        \pivotQ), \qquad \counit(\pivotQ) = 1,
    \end{align}
    and $S^2(h) = \pivotQ h \pivotQ\inv$, for all $h\in H$
    \cite[Prop.~3.2]{BCT-involutory}.
    As for Hopf algebras, the inverse of the pivot is $\pivotQ\inv = S(\pivotQ)$, see
    \cite[Prop.~3.12]{BT-sovereign}.

\medskip

\subsection{Monadic cointegrals for quasi-Hopf algebras}
\label{subseq:HopfMonad-qHopf}

    Left and right integrals for quasi-Hopf algebras are defined in the same way
    as for Hopf algebras, see \Cref{Example:DistInvHopf}, and it was shown
    in \cite{HN-integrals} that the space of left (resp.\ right) integrals of a
    quasi-Hopf algebra is one-dimensional.
    The \emph{modulus} $\modulus$ of a quasi-Hopf algebra is defined as in the
    Hopf case, cf.\ \eqref{eq:modulus_definition}.
    As for Hopf algebras, the distinguished invertible object $D$ of $\hmodM$ is the
    one-dimensional module with action given by $\modulus\inv = \modulus \circ S$, see
    \cite[Prop.~6.5.5]{EGNO},  and we have $\modulus\inv = \dualL{\modulus}$ as
    $H$-modules.
    We call $H$ \emph{unimodular} if $\modulus=\counit$, or
    equivalently, if every left integral is also right.

    \medskip
    We now want to describe the monadic cointegrals for $H$.
    To this end, let us first give our realizations of the central Hopf monads
    $\hopfMonad$ on the category $\hmodM$.
    As for ordinary Hopf algebras, the objects $\hopfMonad[i][V]$, $V\in \hmodM$, have
    $H^* \otimes V$ as underlying vector space,
    and the same dinatural transformations $\dinatCoend[i][V]$ as in
    \eqref{eq:dinat_trans-Hopf} and \eqref{eq:dinat_trans-Hopf-pivotal}.
    These uniquely determine the $H$-action on $\hopfMonad[i][V]$ (denoted by
    $\actionCoend$) to be
    \begin{align}
        \actionCoend[1][h][f \otimes v]
        &= 
        \langle 
            f \mid 
            S\inv( h\sweedler{1} ) \placeholder h\sweedler{2,2}
        \rangle
        \otimes h\sweedler{2,1}.v \ ,
        \notag \\
        \actionCoend[2][h][f \otimes v]
        &= 
        \langle 
            f \mid 
            S( h\sweedler{1} ) \placeholder h\sweedler{2,2}
        \rangle
        \otimes h\sweedler{2,1}.v \ ,
        \notag \\
        \actionCoend[3][h][f \otimes v]
        &= 
        \langle 
            f \mid 
            S\inv( h\sweedler{2} ) \placeholder h\sweedler{1,1}
        \rangle
        \otimes h\sweedler{1,2}.v \ ,
        \notag \\
        \actionCoend[4][h][f \otimes v]
        &= 
        \langle 
            f \mid 
            S( h\sweedler{2} ) \placeholder h\sweedler{1,1}
        \rangle
        \otimes h\sweedler{1,2}.v \ .
        \label{eq:hopfMonad-qHopf-actions}
    \end{align}
    Let us also record here the Hopf monad isomorphisms from
    \Cref{prop:iso_as_HC_pivotal} for (pivotal) quasi-Hopf algebras, using our
    realizations of the monads.

    \begin{prop}
        \label{prop:IsoBetweenHopfMonads-qHopf}
        The canonical Hopf monad isomorphisms 
        \begin{align}
            \hopfMonad[1][V]
            \xrightarrow{(\kappa_{1,2})_V}
            \hopfMonad[2][V]
            \xrightarrow{(\kappa_{2,3})_V}
            \hopfMonad[3][V]
            \xrightarrow{(\kappa_{3,4})_V}
            \hopfMonad[4][V]
        \end{align}
        from \Cref{prop:iso_as_HC_pivotal} (for the maps $(\kappa_{1,2})_V$,
        $(\kappa_{3,4})_V$, we require a pivot $\pivotQ$) are given by the linear maps
        \begin{align}
            (\kappa_{1,2})_V (f\otimes v) 
            &= (f \leftharpoonup \pivotQ\inv) \otimes v \ ,
            \notag \\
            (\kappa_{2,3})_V (f\otimes v) 
            &= \langle f \mid S(\placeholder \coassQ[1]_1) \coassQ[1]_3 \rangle 
            \otimes \coassQ[1]_2 . v \ ,
            \notag \\
            (\kappa_{3,4})_V (f\otimes v) 
            &= (f \leftharpoonup \pivotQ\inv) \otimes v \ ,
            \label{eq:HopfMonadIsos-qHopf}
        \end{align}
        for $f\in H^*$, $v\in V$.
    \end{prop}
    \begin{proof}
        The proof is a straightforward computation.
        For example, for $\kappa_{2,3}$ one needs to check the commutativity of
        \begin{equation}
            \begin{tikzcd}
                \dualL{X} \otimes (V \otimes X)
                \ar[r,"\sim"] 
                \ar[d,"{\dinatCoend[2][V][X]}"]
                & (\dualL{X} \otimes V) \otimes X \ar[r,"\sim"]
                &
                (\dualL{X} \otimes V) \otimes \dualR{(\dualL{X})}
                \ar[d,"{\dinatCoend[3][V][\dualL{X}]}"]
                \\
                \hopfMonad[2][V]
                \ar[rr,"(\kappa_{2,3})_V"]
                & & \hopfMonad[3][V]
            \end{tikzcd}\ .
        \end{equation}
        Note here that the isomorphism $X \cong \dualR{(\dualL{X})}$ is the same linear
        map as in $\Vect$.
    \end{proof}

    Define $\tau \in H^{\otimes 5}$ by
    \begin{align}
        \tau 
        &= \coassQ[1]_1 
        \otimes \coassQ[1]_2 \invCoassQ[2]_1
        \otimes \invCoassQ[1]_1 ({\coassQ[1]_3}\sweedler{1} {\invCoassQ[2]_2})\sweedler{1}
        \otimes \invCoassQ[1]_2 ({\coassQ[1]_3}\sweedler{1} {\invCoassQ[2]_2})\sweedler{2}
        \otimes \invCoassQ[1]_3 {\coassQ[1]_3}\sweedler{2} {\invCoassQ[2]_3}
        \end{align}
        The multiplication of $\hopfMonad[2]$ can be computed explicitly from
        \eqref{eq:DefMultCoend}.
        For $f,g\in H^*$ and $v\in V$, the image $\multCoend[2][V](f\otimes g\otimes v)
        \in H^*\otimes V$ under multiplication can be identified with the linear map
        \begin{align}
            H \ni 
            h \mapsto 
                (g \otimes f)
                \big(
                    (S\otimes S)(\tau_{21}) \Dt \Delta(h) \tau_{45}
                \big) \ 
                \tau_3.v 
            \, \in V 
            \ .
        \end{align}
        The counit $\counitCoend[2]:\hopfMonad[2][\tensUnit] \to \tensUnit$ is easily
        computed from \eqref{eq:CoendComonoidalStructure}, and we identify it with the
        element $\counitCoend[2] = \alphaQ \in H$.
        Let us recall the right monadic cointegral equation from
        \eqref{eq:monCointEquation}:
        \begin{align}
            \coint \circ \counitCoend[2]
            = \multCoend[2][\distInvObj] \circ \hopfMonad[2][\coint]\ .
        \end{align}
        This is an equality of (linear) endomorphisms of $H^*$, and evaluating it on $f\in
        H^*$, we immediately get
        \begin{align}
            f(\alphaQ) \coint 
            = 
            \modulus\inv(\tau_3)
            (\coint \otimes f)
            \big(
                (S\otimes S)(\tau_{21}) \Dt \Delta(e_i) \tau_{45}
            \big)
            e^i.
        \end{align}
        This is clearly equivalent to 
        \begin{align}\label{eq:qHopf-rmco-lin}
            \coint(h) \alphaQ 
            =
            \modulus\inv(\tau_3) 
            (\coint \otimes \id) 
            \big( 
                (S \otimes S)(\tau_{21}) \Dt \Delta(h) \tau_{45}
            \big)
        \end{align}
        for all $h\in H$.
    
    Altogether, an element $\coint \in H^*$ is a right monadic cointegral if and only if
    it satisfies \eqref{eq:qHopf-rmco-lin} and is an $H$-module intertwiner $\tensUnit \to
    \hopfMonad[2][\modulus\inv]$ (see \eqref{eq:app-linear-eqn-Hom-to-A2gamma} to see this
    written out as a linear equation).\footnote{
        The $H$-intertwiner condition is automatic for monadic cointegrals for Hopf
        algebras, and an analogous condition is automatic for cointegrals of quasi-Hopf
        algebras as defined in \Cref{def:cointegrals} below.
        The corresponding statement remains to be shown in the monadic setting for
        quasi-Hopf algebras, but it is not needed for the present paper.
    }

    Similarly, with $\Dt^r$ the Drinfeld twist for right duals and $\sigma =
    (\tau\cop)_{54321}$ given explicitly by
    \begin{align}
        \sigma
        &= {\invCoassQ[1]_1}\sweedler{1,1} \coassQ[2]_1 \coassQ[1]_1
        \otimes {\invCoassQ[1]_1}\sweedler{1,2} \coassQ[2]_2 {\coassQ[1]_2}\sweedler{1}
        \otimes {\invCoassQ[1]_1}\sweedler{2} \coassQ[2]_3 {\coassQ[1]_2}\sweedler{2}
        \otimes \invCoassQ[1]_2 \coassQ[1]_3 
        \otimes \invCoassQ[1]_3 \ ,
    \end{align}
    one obtains necessary conditions for the three remaining types of monadic cointegrals.
    Namely, if $\coint \in H^*$ is a 
    \begin{enumerate}
        \item 
            right $\distInvObj$-symmetrised monadic cointegral then it satisfies
            \begin{align}
                \coint(h) \pivotQ\inv \alphaQ
                =
                \modulus\inv(\tau_3) 
                (\coint \otimes \id) 
                \big( 
                    (S\inv \otimes S\inv)(\tau_{21}) \Dt^r \Delta(h)
                    \tau_{45}
                \big)
            \end{align}
        \setcounter{enumi}{2}
        \item 
            left monadic cointegral then it satisfies
            \begin{align}
                \coint(h) S\inv(\alphaQ)
                =
                \modulus\inv(\sigma_3) 
                (\id \otimes \coint) 
                \big( 
                    (S\inv \otimes S\inv)(\sigma_{54}) \Dt^r \Delta(h)
                    \sigma_{12}
                \big)
            \end{align}
        \item 
            left $\distInvObj$-symmetrised monadic cointegral then it
            satisfies
            \begin{align}
                \coint(h) \pivotQ S\inv(\alphaQ)
                =
                \modulus\inv(\sigma_3) 
                (\id \otimes \coint) 
                \big( 
                    (S \otimes S)(\sigma_{54}) \Dt \Delta(h) \sigma_{12}
                \big)
            \end{align}
    \end{enumerate}
    for all $h\in H$.

    \smallskip

\subsection {Special elements and relations}
\label{sec:DrinfeldTwist-closed}

    Here we want to give the closed form of the Drinfeld twist and its inverse.
    We closely follow \cite{BC2-2011}, but note that our conventions are slightly
    different, i.e.\ our $\coassQ$ is their $\coassQ\inv$.

    We will need the four elements $\qR, \pR, \qL, \pL$ in $H\otimes H$, given by 
    \begin{alignat}{2}
        \notag
        &\qR= \invCoassQ[1]_1\otimes S\inv(\alphaQ \invCoassQ[1]_3)\invCoassQ[1]_2 \ ,
        \qquad 
        && \pR = \coassQ[1]_1\otimes \coassQ[1]_2 \betaQ S(\coassQ[1]_3) \ , \\
        \label{eq:q,\pL}
        &\qL= S(\coassQ[1]_1) \alphaQ \coassQ[1]_2 \otimes \coassQ[1]_3 \ , \qquad
        && \pL = \invCoassQ[1]_2 S\inv(\invCoassQ[1]_1 \betaQ) \otimes \invCoassQ[1]_3 
        \ .
    \end{alignat}
    These satisfy the identities
    \begin{alignat}{2}
        \notag
        \Delta(\qR_1) \pR [\oneQ\otimes S(\qR_2)] &= \oneQ\otimes \oneQ \ , \qquad 
        & [\oneQ\otimes S\inv(\pR_2)] \qR \Delta(\pR_1) &=\oneQ\otimes \oneQ \ , \\
        \label{eq:identities qpL}
        \Delta(\qL_2) \pL [S\inv(\qL_1)\otimes \oneQ] &= \oneQ\otimes \oneQ \ , \qquad 
        & [S(\pL_1)\otimes \oneQ] \qL \Delta(\pL_2) &=\oneQ\otimes \oneQ \ ,
    \end{alignat}
    and, for all $a\in H$,
    \begin{alignat}{1}
        \label{eq:identities_qp_iterated_coproduct}
        [\oneQ \otimes S\inv(a\sweedler{2})] \qR \Delta (a\sweedler{1}) 
        &= [a \otimes \oneQ]\qR  \ , \notag \\  
        [S(a\sweedler{1})\otimes \oneQ] \qL \Delta (a\sweedler{2}) 
        &= [\oneQ \otimes a ]\qL \ , \notag \\
           \Delta(a\sweedler{1}) \pR [\oneQ \otimes S(a\sweedler{2})]
       &= \pR [a\otimes \oneQ] \ , \notag \\
        \Delta(a\sweedler{2}) \pL [S\inv(a\sweedler{1}) \otimes \oneQ]
        &= \pL [\oneQ \otimes a] \ .
    \end{alignat}
    These elements and relations are well-known in the representation theory of
    quasi-Hopf algebras. 
    For their interpretation in terms of natural transformations and categorical
    identities see \cite{HN-doubles} and e.g.\ \cite[Sec.~3]{BGR1}.
    
    Next, let us also define $\elQbold{\varepsilon},\elQbold{\delta} \in H\otimes H$ by
    \begin{align}\label{eq:gamma-qHopf-DT}
        \elQbold{\varepsilon}
        &= S(\invCoassQ[1]_2) \qL_1 {\invCoassQ[1]_3}\sweedler{1}
        \otimes S(\invCoassQ[1]_1) \alphaQ \qL_2 {\invCoassQ[1]_3}\sweedler{2} 
        = S(\qR_2 {\coassQ[1]_1}\sweedler{2}) \coassQ[1]_2 
        \otimes S(\qR_1 {\coassQ[1]_1}\sweedler{1}) \alphaQ \coassQ[1]_3
    \end{align}
    and
    \begin{align}\label{eq:delta-qHopf}
        \elQbold{\delta}
        &= {\invCoassQ[1]_1}\sweedler{1} \pR_1 \betaQ S(\invCoassQ[1]_3)
        \otimes {\invCoassQ[1]_1}\sweedler{2} \pR_2 S(\invCoassQ[1]_2)
        = \coassQ[1]_1 \betaQ S({\coassQ[1]_3}\sweedler{2} \pL_2)
        \otimes \coassQ[1]_2 S({\coassQ[1]_3}\sweedler{1} \pL_1)\ .
    \end{align}
    Then the Drinfeld twist is given by
    \begin{align}
        \label{eq:DrinfeldTwist-closedForm}
        \Dt 
        = (S\otimes S)(\Delta\cop ( \pR_1 )) 
        \elQbold{\varepsilon}
        \Delta(\pR_2)
    \end{align}
    and its inverse is
    \begin{align}
        \label{eq:DrinfeldTwistInverse-closedForm}
        \Dt\inv
        = \Delta(\qL_1) 
        \elQbold{\delta}
        (S\otimes S)(\Delta\cop ( \qL_2 ))\ .
    \end{align}
    The explicit form of the Drinfeld twist for right duals can similarly be given as
    \begin{align}
        \Dt^r 
        &= (S\inv \otimes S\inv)(\elQbold{\varepsilon}_{21} \Delta\cop(\pL_1))
        \Delta(\pL_2) \ ,
        \notag \\ 
        (\Dt^r)\inv
        &= \Delta( \qR_2 ) (S\inv \otimes S\inv) \big( \Delta\cop(\qR_1) \elDelta_{21}
        \big)
        \ .
    \end{align}

    \smallskip
    We end this subsection by stating some technical properties of the Drinfeld twist
    which we will need later on (see~\cite{BC2-2011}).    
    The Drinfeld twist satisfies the identity
    \begin{align}
        (\oneQ \otimes \Dt)\cdot (\id \otimes \Delta)(\Dt)\cdot \invCoassQ
        =
        (S\otimes S\otimes S)(\invCoassQ_{321}) \cdot (\Dt \otimes \oneQ)\cdot 
        (\Delta \otimes \id)(\Dt) \ ,
    \end{align}
    or, written in Sweedler-notation,
    \begin{align} 
        \notag
        \Dt_1 \invCoassQ[1]_1 \otimes \widetilde{\Dt}_1{\Dt_2}\sweedler{1} \invCoassQ[1]_2
        & \otimes \widetilde{\Dt}_2{\Dt_2}\sweedler{2} \invCoassQ[1]_3
        \\ \label{eq:DrinfeldTwist-twistedCoassoc}
        &=
        S(\invCoassQ[1]_3) \widetilde{\Dt}_1 {\Dt_1}\sweedler{1}
        \otimes S(\invCoassQ[1]_2) \widetilde{\Dt}_2 {\Dt_1}\sweedler{2}
        \otimes S(\invCoassQ[1]_1) \Dt_2 \ ,
    \end{align}
    where we use the symbol $\widetilde{\Dt}$ to denote another copy of the Drinfeld
    twist.

    A direct computation shows that $\Dt$ further satisfies
    \begin{align}\label{eq:Dt-alphBet-interplay}
        \Dt_1 \betaQ S(\Dt_2) = S(\alphaQ)
        \ , \quad
        S(\betaQ \Dt_1) \Dt_2 = \alphaQ \ .
    \end{align}
    Lastly, we have
    \begin{align}\label{eq:coproduct_beta}
        \Delta(\betaQ) \Dt\inv = \elDelta\ .
    \end{align}

\subsection{Cointegrals via coactions}
    \label{Sec:BC-CointegralsPartiallyExplained}
    In preparation for the proof of the main theorem let us now briefly review the
    original definition of cointegrals for quasi-Hopf algebras from \cite{HN-integrals}.

    \smallskip

    Following \cite{HN-integrals,BC2-2011}, we set
    \begin{alignat}{3}
        \elU &= \Dt\inv (S\otimes S)(\qR_{21}), 
        \quad
            && \elU\cop && = (S\inv\otimes S\inv)(\qL \Dt\inv),
        \notag \\
        \label{eq:Definition_UV}
        \elV &= (S\inv \otimes S\inv)\left( \Dt_{21} \pR_{21}\right),
        \quad\quad
            && \elV\cop && = (S\otimes S)(\pL) \Dt_{21}.
    \end{alignat}

    \begin{mydef}\label{def:cointegrals}
        A \emph{left cointegral} for $H$ is an element $\coint^l\in H^*$ satisfying
        \begin{align}\label{eq:LeftCointEq}
            (\id\otimes \coint^l)
            \left(
                \elV
                \Delta(h)
                \elU
            \right)
            = \modulus(\coassQ[1]_1) \coint^l(h S(\coassQ[1]_2)) \coassQ[1]_3
        \end{align}
        for all $h\in H$, and a \emph{right cointegral} for $H$ is a left cointegral for
        $H\cop$. 
        Explicitly, this means it is an element $\coint^r\in H^*$ satisfying
        \begin{align}\label{eq:RightCointEq}
            (\id \otimes \coint^r)\big( \elV\cop  \Delta\cop(h) \elU\cop \big)
            &=
            \modulus(\invCoassQ[1]_3) \coint^r(h S\inv(\invCoassQ[1]_2)) \invCoassQ[1]_1.
        \end{align}
    \end{mydef}

    \smallskip

    Cointegrals exist, and a non-zero cointegral is a non-degenerate linear form on~$H$,
    uniquely determined up to scalar~\cite[Thm.~4.3]{HN-integrals}.
    We denote the spaces of left and right cointegrals by 
    \begin{align}\label{eq:spacesOfBC-Cointegrals}
        \spaceLeftCoint 
        \quad \text{and} \quad 
        \spaceRightCoint,
    \end{align}
    respectively.
    Let now $\coint^l$ and $\coint^r$ be a left and a right cointegral of $H$,
    respectively.
	These satisfy
    \begin{align}\label{eq:symmetryPropertiesCoint}
        \coint^l(S\inv(a) b) = \coint^l( b S(a \leftharpoonup \modulus))
        \quad \text{and} \quad
        \coint^r(S(a) b) = \coint^r( b S\inv( \modulus \rightharpoonup a) ),
    \end{align}
    for $a,b\in H$, see \cite[Lem.~5.1]{HN-integrals}.
    With 
    \begin{align}\label{eq:elu-right-left-coint-defined}
        \elu = (\modulus \otimes S^2)(\elV)
        \quad \text{and} \quad
        \elu\cop = (\modulus \otimes S^{-2})(\elV\cop)
            \ ,
    \end{align}
    left and right cointegrals can be related as follows \cite[Prop.~4.3]{BC2-2011},
    \begin{align}\label{eq:relating_left_and_right_cointegrals}
        (\coint^r \leftharpoonup \elu) \circ S \in \spaceLeftCoint
        \ ,
        \quad \text{resp.} \quad
        (\coint^l \leftharpoonup \elu\cop) \circ S\inv \in \spaceRightCoint 
        \ .
    \end{align}
    The corresponding relation between left and right monadic cointegrals is given by
    $(\kappa_{2,3})_D$, see \Cref{remark:intertw-diagram-coint-transport}\,(2).
    It is worth comparing the definition of $\elu$ to that of $\kappa_{2,3}$ from
    \Cref{prop:IsoBetweenHopfMonads-qHopf}, which only used a single
    coassociator.

    \smallskip

    The left-hand side of the left cointegral equation \eqref{eq:LeftCointEq} has the
    following categorical interpretation, which we recall from \cite[Sec.~3]{BC2-2011}.
    Let $\hmodM[H][H]$ be the category of $H$-bimodules, equipped with the monoidal
    structure of the category of modules over the quasi-Hopf algebra $H\otimes H\op$.

    In $\hmodM[H][H]$, the regular bimodule $H$ is a coalgebra,\footnote{
        The coassociativity diagram for the coalgebra $H$ in the non-strict category
        $\hmodM[H][H]$ is precisely the quasi-coassociativity condition
        \eqref{eq:quasiCoassociativity} on the quasi-Hopf algebra $H$.
    }
    and we can thus consider the comonad $\cat[Y]^r:B \mapsto B \otimes H$ on
    $\hmodM[H][H]$.
    With $H$ equipped with the obvious $\cat[Y]^r$-comodule structure, via the coproduct,
    the right dual $\dualR{H}$ becomes a $\cat[Y]^r$-comodule via
    \begin{align}
        \rho^r&:\dualR{H} \to \cat[Y]^r(\dualR{H}),
        \notag \\
        \rho^r &=
        \bigg[
            \dualR{H} \xrightarrow{\sim} \tensUnit\dualR{H}
            \xrightarrow{\coevR_{\!\!H} \otimes \id}
            ( \dualR{H} H ) \dualR{H}[-0.25]
            \xrightarrow{(\id \otimes \Delta)\otimes \id}
            ( \dualR{H} (HH) ) \dualR{H}
            \xrightarrow{\sim}
        \notag \\
            &\quad\quad
            \xrightarrow{\sim}
        (\dualR{H} H) (H~\dualR{H}[-0.2])
            \xrightarrow{\id\otimes \evL_{\!\!H}} (\dualR{H} H)\tensUnit 
            \xrightarrow{\sim} \dualR{H} H
            = \cat[Y]^r(\dualR{H})
        \bigg] \ .
    \end{align}
    Here all coherence isomorphisms are those of $\hmodM[H][H]$, so that 
    explicitly,
    $\rho^r$ sends $f\in H^*$ to
    \begin{align}
        \rho^r(f) = f(\elV_2 (e_i)\sweedler{2} \elU_2) e^i 
        \otimes \elV_1 (e_i)\sweedler{1} \elU_1
        \ .
    \end{align}
    Then \eqref{eq:LeftCointEq} can be written as 
    \begin{align}
        \label{eq:LeftCointEq_withoutArgument}
        \rho^r(\coint^l)
        = \modulus(\coassQ[1]_1) \coint^l.\coassQ[1]_2 \otimes \coassQ[1]_3 \ ,
    \end{align}
    where by the dot we mean the right $H$-action 
    on the right dual $\dualR{H}$, i.e.\ $(f.a)(h) = f(h S(a))$ for $f\in \dualR{H}$,
    $a\in H$.\footnote{
        The full meaning of \eqref{eq:LeftCointEq_withoutArgument}
        is that $\coint^l$ is a
        `coinvariant' of the coaction $\rho^r$.
        This statement is made precise using the so-called \emph{fundamental theorem of
        quasi-Hopf bimodules}, cf.\ \cite[Sec.\ 3]{HN-integrals},  which states that there
        is a monoidal adjoint equivalence $(\hmodM[H][H])^{\cat[Y]^r} \cong \hmodM[H]$.
        The left $H$-module of \emph{coinvariants} of a $\cat[Y]^r$-comodule $B$ is then
        defined as the image of $B$ under the equivalence.
        By \cite[Cor.~3.9]{HN-integrals}, \eqref{eq:LeftCointEq_withoutArgument}
        is an equivalent characterization of the coinvariants of $\dualR{H}$.
        
    }

    Similarly, the left dual $\dualL{H}$ carries a natural left $H$-comodule
    structure, i.e.\ it is a comodule of the comonad $\cat[Y]^l: B\mapsto H\otimes B$;
    call the corresponding coaction $\rho^l$.
    Using the explicit expressions from above, one verifies $(\rho^r)\cop_{21} =
    \rho^l$, and we obtain that $\coint^r\in H^*$ is a right cointegral if and
    only if
    \begin{align}
        \label{eq:RightCointEq_withoutArgument}
        \rho^l(\coint^r)
        =
        \modulus(\invCoassQ[1]_3)  \invCoassQ[1]_1 \otimes \coint^r.\invCoassQ[1]_2\ ,
    \end{align}
    and the dot here denotes the action on the left dual.

\smallskip

    For later use, we relate the comonad $\hopfComonad[4]$ on $\hmodM[H]$ (recall its
    definition in \Cref{subsec:comonad-Z}) and $\cat[Y]^l$ on $\hmodM[H][H]$ as
    follows.
    Consider the functor 
    \begin{align}
        \label{eq:AdTwistFunctor}
        \Adtwist: \hmodM[H][H] &\to \hmodM[H]
    \end{align}
    sending a bimodule $B$ to the $H$-module $B$ with action 
    \begin{align}
        h\otimes b \mapsto h\sweedler{1}.b.S(h\sweedler{2} \leftharpoonup \modulus\inv)
        = \modulus\inv(h\sweedler{2,1}) h\sweedler{1}.b.S(h\sweedler{2,2}),
    \end{align}
    for $h\in H$, $b\in B$. 
    Then there is a natural isomorphism $\varphi$ making the diagram 
    \begin{equation}
        \label{cd:nat_trans:AY=ZA}
        \begin{tikzcd}[sep=large]
            \hmodM[H][H] 
            \ar[r, "{\cat[Y]^l}"]
            \ar[d, swap, "{\Adtwist}"]
            &
            \hmodM[H][H]
            \ar[d, "{\Adtwist}"]
            \ar[dl, Rightarrow, "{\varphi}"]
            \\
            \hmodM[H]
            \ar[r, swap, "{\hopfComonad[4]}"]
            &
            \hmodM[H]
        \end{tikzcd}
    \end{equation}
    commute.

    To see this, let us also specialise the comonad $\hopfComonad[4]$ from
    \cite{Sh-integrals} to the case $\cat=\hmodM$.
    We choose the realisation such that the objects $\hopfComonad[4][V]$ are given by the
    underlying vector space $H\otimes V$ with actions
    \begin{align} \label{eq:Z4-realisation}
        \actionEnd[4][h][a\otimes v]
        &= h\sweedler{1,1}aS(h\sweedler{2}) \otimes h\sweedler{1,2}.v \ .
    \end{align}
    For later use, let us record that the unit of $\hopfComonad[4]$ is given by the
    coevaluation element,
    \begin{align}\label{eq:unitEnd-4}
        \unitEnd[4] = \betaQ \ .
    \end{align}

    We then get the following explicit form of $\varphi$.

    \begin{prop}\label{prop:NatIso_AdY-ZAd}
        The family of maps
        \begin{align}\notag
            \varphi_B: \Adtwist \cat[Y]^l(B) &\to \hopfComonad[4]\Adtwist(B),\\ \notag
            h \otimes b &\mapsto 
            \modulus\inv({\invCoassQ[1]_3}\sweedler{1}
            {\coassQ[1]_2}\sweedler{1}\coassQ[2]_1)
            ~\invCoassQ[1]_1 {\coassQ[1]_1}\sweedler{1} h \Dt\inv_1 S(\coassQ[1]_3
            \coassQ[2]_3)  \\ 
            &\hspace*{2.5cm}\otimes
            \invCoassQ[1]_2 {\coassQ[1]_1}\sweedler{2} .b. \Dt\inv_2
            S({\invCoassQ[1]_3}\sweedler{2}
            {\coassQ[1]_2}\sweedler{2}\coassQ[2]_2) 
            ,
            \label{eq:phi-B-qHopf-explicit}
        \end{align}
        for $B\in \hmodM[H][H]$ defines a natural isomorphism as in
        \eqref{cd:nat_trans:AY=ZA}.
   \end{prop}

    \begin{proof}
        That the above map is natural in $B\in \hmodM[H][H]$ is immediate, and proving
        that it intertwines the corresponding $H$-actions is a straightforward
        calculation.
        For convenience we state that the action on 
        $\Adtwist \cat[Y]^l(B)$ is
        \begin{align}
            h\otimes(a\otimes b) \mapsto 
            h\sweedler{1,1} a S(h\sweedler{2} \leftharpoonup \modulus\inv)\sweedler{1}
            \otimes 
            h\sweedler{1,2}. b . S(h\sweedler{2} \leftharpoonup \modulus\inv)\sweedler{2}
        \end{align}
        and the action on $\hopfComonad[4]\Adtwist (B)$ is
        \begin{align}
            h\otimes(a\otimes b) \mapsto 
            h\sweedler{1,1} a S(h\sweedler{2})
            \otimes 
            h\sweedler{1,2,1}. b . S(h\sweedler{1,2,2} \leftharpoonup \modulus\inv)
        \end{align}
        for $a,h\in H$, $b\in B$.
        Here the dot denotes the action on the bimodule $B$.

        Finally, the inverse of $\varphi_B$ can be read off directly 
        from the explicit expression in~\eqref{eq:phi-B-qHopf-explicit}.
    \end{proof}

\medskip

\section{Main Theorem}\label{SEC:MainThm}
We are now ready to state our two main theorems, which are
\Cref{thm:MainThmSectionStatement,thm:MainThmPivotalCase} below.

\subsection{Left and right monadic cointegrals}
\begin{thm}\label{thm:MainThmSectionStatement}
    Let $H$ be a finite-dimensional quasi-Hopf algebra with modulus $\modulus$.
    \begin{enumerate}
        \item
            Define the linear map   
            \begin{align}\label{eq:MainThmMap}
                \monCointRight{(\placeholder)}:
                H^* &\to H^*, \quad
                \monCointRight{f} 
                = \langle f \mid S(\betaQ) ~ \placeholder ~ S\inv(\elX)
                \rangle
                \ ,
            \end{align}
            where $\elX = (\id \otimes \modulus) (\Dt\inv)$.
            Then $\monCointRight{\coint}$ is a right monadic cointegral if and only
            if $\coint \in H^*$ is a right cointegral.
        \item
            Define the linear map 
            \begin{align}
                \label{eq:MainThmMap-Left}
                \monCointLeft{(\placeholder)}:
                H^* &\to H^*, \quad
                \monCointLeft{f}
                = \langle f \mid S^{-2}(\betaQ) ~ \placeholder ~ S(\hat{\elX}) \rangle \ ,
            \end{align}
            where $\hat{\elX} = \elX\cop = (S\inv \otimes \modulus\inv) (\Dt\inv)$.
            Then $\monCointLeft{\coint}$ is a left monadic cointegral if and only if
            $\coint \in H^*$ is a left cointegral,
    \end{enumerate}
\end{thm}

Let us explain the main ideas in the proof.
First, we specialise the equivalent characterisation of (right) monadic cointegrals
from \Cref{lem:monCointViaCatCoaction} to the quasi-Hopf setting.
The resulting equation resembles the right cointegral equation
\eqref{eq:RightCointEq_withoutArgument} we encountered in our discussion of quasi-Hopf
algebras.
Indeed, we find a nice relationship between the categorical coaction $R_D$
from~\eqref{eq:coactionLeftAdjoint} and the left coaction from \cite{BC2-2011}, cf.\
\Cref{prop:CatCoaction_is_coaction} below.

Using the relation between these two coactions we then show that the map
\eqref{eq:MainThmMap} sends a right cointegral to a right monadic cointegral via a direct
calculation.
This establishes Part~1. Part~2 will then be inferred from Part~1 using the isomorphism
$\hopfMonad[2] \cong \hopfMonad[3]$.

The details follow below, with some technical steps deferred to \Cref{app:Proof1}.

\subsection*{Relation to quasi-Hopf cointegrals}
Let $\cat=\hmodM$ and recall our realization of the Hopf comonad $\hopfComonad[4]$ from
\eqref{eq:Z4-realisation} and \eqref{eq:unitEnd-4}.
In this setting, the distinguished invertible object of $\cat$ is $\dualL{\modulus}$,
and we can rewrite
Equation~\eqref{eq:MonadicCointegralsViaCoaction} on  right monadic cointegrals as the
linear equation
\begin{align}\label{eq:MonadicCointViaCoaction-qHopf}
    R_{\dualL{\modulus}} \circ \coint
    = \betaQ \otimes_{\field} \coint\ ,
\end{align}
where we identified morphisms from $\tensUnit$ to $H$ with elements in $H$.
We will show that this equation is equivalent to the right cointegral equation
\eqref{eq:RightCointEq_withoutArgument}, 
\begin{align}\label{eq:RightCointEq_withoutArgument-qHopf}
    \rho^l(\coint^r)
    =
    \modulus(\invCoassQ[1]_3)  \invCoassQ[1]_1 \otimes \coint^r.\invCoassQ[1]_2\ , 
\end{align}
where $\coint^r\in H^*$ satisfies $\coint = \monCointRight{(\coint^r)}$ with the
$(-)^{\mathrm{mon}}$ defined in~\eqref{eq:MainThmMap}.

We will first relate the categorical coaction $R_{\dualL{\modulus}}$ and the coaction
$\rho^l$.
To this end, recall the functor $\Adtwist:\hmodM[H][H]\to \hmodM$ from
\eqref{eq:AdTwistFunctor}.
One can check that in our realisation of the central Hopf monad we have the equality
$\Adtwist(\dualL{H}) = \hopfMonad[2](\dualL{\modulus})$ of $H$-modules.

\begin{prop}\label{prop:CatCoaction_is_coaction}
    With the natural isomorphism
   $\varphi: \cat[AY]^l \To \hopfComonad[4] \Adtwist$
    from \eqref{cd:nat_trans:AY=ZA} and the left $H$-coaction $\rho^l: \dualL{H}
    \to \cat[Y]^l(\dualL{H})$ from \Cref{Sec:BC-CointegralsPartiallyExplained} we have that
    \begin{align}
        R_{\dualL{\modulus}} = 
        \left[ 
            \Adtwist(\dualL{H}) 
            \xrightarrow{\Adtwist(\rho^l)}
            \Adtwist\cat[Y]^l(\dualL{H})
            \xrightarrow{\varphi_{\dualL{H}}}
            \hopfComonad[4]\Adtwist(\dualL{H})
        \right].
    \end{align}
\end{prop}
The proof of this proposition has been relegated to
\Cref{proof:CatCoaction_is_coaction}.

Note that $\Adtwist$ does not do anything on morphisms; 
in particular, the linear maps $\Adtwist(\rho^l)$ and $\rho^l$ are identical.
Then this proposition together with \eqref{eq:MonadicCointViaCoaction-qHopf}, says that
$\coint$ is a right monadic cointegral if and only if it satisfies 
\begin{align}
    \label{eq:MainThmProof-final-version-of-monCoint-Eq}
    (\varphi_{\dualL{H}} \circ \rho^l) (\coint) = 
    \betaQ \otimes_{\field} \coint \ .
\end{align}
In \Cref{app:Proof1} we prove that this is equivalent to the right cointegral
equation~\eqref{eq:RightCointEq_withoutArgument-qHopf} using the map
\eqref{eq:MainThmMap}.
This finishes the proof of the first part of \Cref{thm:MainThmSectionStatement}.

\bigskip

The second part is the same as the first part for the coopposite quasi-Hopf algebra, but
we follow a more direct approach using the isomorphism $\hopfMonad[2] \cong \hopfMonad[3]$
from \Cref{prop:iso_as_HC_pivotal} and
\Cref{prop:IsoBetweenHopfMonads-qHopf}, see \Cref{app:ProofMainThPart2}.

Namely, in the appendix, we show that the diagram 
\begin{equation}
    \begin{tikzcd}[row sep=large, column sep=large]
        \spaceRightCoint
        \ar[r,"\eqref{eq:MainThmMap}"]
        \ar[d,"(*)",swap]
        & \spaceRightMonCoint
        \ar[d,"(\kappa_{2,3})_\distInvObj ~ \circ ~ \placeholder"]
        \\
        \spaceLeftCoint
        \ar[r,"\eqref{eq:MainThmMap-Left}",swap]
        &
        \spaceLeftMonCoint
    \end{tikzcd}
\end{equation}
commutes, where $(*)$ is, up to a non-zero factor, the isomorphism between left and right
cointegrals \eqref{eq:relating_left_and_right_cointegrals}, and $\kappa_{2,3}:
\hopfMonad[2] \Rightarrow \hopfMonad[3]$ is the Hopf monad isomorphism from
\eqref{eq:HopfMonadIsos-qHopf}.

\medskip

\subsection{Left and right \texorpdfstring{$\distInvObj$}{D}-symmetrised monadic
cointegrals}\label{sec:D-sym-coint-qHopf}

    To prove an analogous result of \Cref{thm:MainThmSectionStatement} for
    $\distInvObj$-symmetrised monadic cointegrals, we first need to recall the appropriate version of cointegrals on the quasi-Hopf side.

\begin{mydef}
Let $(H,\pivotQ)$ be a finite-dimensional pivotal quasi-Hopf algebra. 
\begin{enumerate}
	\item
The \emph{right $\modulus$-symmetrised cointegrals} $\symRightCoint\in H^*$ are the solutions of
\begin{align}
    (\symRightCoint \otimes \id) (\qR \Delta(h) \pR) 
    = \modulus(\coassQ[1]_1) \symRightCoint(\coassQ[1]_2 h) \cdot \pivotQ\inv 
    S(\coassQ[1]_3) 
    ~~ \text{for all} ~~ h\in H \ .
\end{align}
\item
The \emph{left $\modulus$-symmetrised cointegrals} $\symLeftCoint$ are the solutions of
\begin{align}
    (\id \otimes \symLeftCoint) (\qL \Delta(h) \pL) 
    = \modulus(\invCoassQ[1]_3) \symLeftCoint(\invCoassQ[1]_2 h) \cdot \pivotQ
    S\inv(\invCoassQ[1]_1) 
    ~~ \text{for all} ~~ h\in H \ .
\end{align}
\end{enumerate}
\end{mydef}
An equivalent definition has been given in \cite[Sec.\,6.4]{Shi-Shi}, and the above
equations also appear in \cite[Lem.\,3.7]{BGR1}.
In that lemma it is shown that, given a right (resp.\ left) cointegral $\coint^r$ (resp.\
$\coint^l$) for $H$, the $\modulus$-symmetrised version can be expressed as
\begin{align}
    \label{eq:symCointViaCoint}
    \symRightCoint = \coint^r \leftharpoonup \elu \pivotQ
    \quad (\text{resp.} \quad
    \symLeftCoint = \coint^l \leftharpoonup \elu\cop \pivotQ\inv
    ) \ ,
\end{align}
where $\elu$ was defined in \eqref{eq:elu-right-left-coint-defined}.
(This is the actual definition of $\modulus$-symmetrised cointegrals given
in~\cite{Shi-Shi}.)
From \eqref{eq:symCointViaCoint} it is clear that non-zero $\modulus$-symmetrised
cointegrals exist and are unique up to scalars, and
from~\eqref{eq:symmetryPropertiesCoint}
one easily verifies
\begin{align}
    \label{eq:symmetryPropertiesSymmetrizedCoint}
    \symRightCoint(ab) = \symRightCoint((b\leftharpoonup \modulus) a)
    \quad \text{and} \quad
    \symLeftCoint(ab) = \symLeftCoint((\modulus \rightharpoonup b) a)
\end{align}
for all $a,b\in H$, see \cite[Eqn.\,(3.44)]{BGR1}.

We denote the one-dimensional spaces of right and left $\modulus$-symmetrised cointegral
by 
\begin{align}
    \spaceRightSymCoint
    \quad \text{and} \quad
    \spaceLeftSymCoint.
\end{align}

Now we can extend \Cref{thm:MainThmSectionStatement} to the pivotal case.

\begin{thm}\label{thm:MainThmPivotalCase}
    Let $(H,\pivotQ)$ be a pivotal quasi-Hopf algebra.
    \begin{itemize}
        \item
            Consider the linear map   
            \begin{align}\label{eq:MainThmMapPivotalRight}
                \monCointRightSym{(\placeholder)}:
                H^* &\to H^*, \quad
                \monCointRightSym{f} 
                = \langle f \mid S(\betaQ) ~ \placeholder ~ S\inv(\elTh) \rangle \ ,
            \end{align}
            where $\elTh = (\modulus\inv \otimes S\inv) (\pL)$.
            Then $\monCointRightSym{\coint}$ is a right $\distInvObj$-symmetrised
            monadic cointegral if and only if $\coint \in H^*$ is a right
            $\modulus$-symmetrised cointegral.
        \item
            Consider the linear map
            \begin{align}\label{eq:MainThmMapPivotalLeft}
               \monCointLeftSym{(\placeholder)}:
                H^* &\to H^*, \quad
                \monCointLeftSym{f}
                = \langle f \mid \betaQ ~ \placeholder ~ S(\hat{\elTh}) \rangle \ ,
            \end{align}
            where $\hat{\elTh} = \elTh\cop = (S \otimes \modulus\inv) (\pR)$.
            Then $\monCointLeftSym{\coint}$ is a left $\distInvObj$-symmetrised
            monadic cointegral if and only if $\coint \in H^*$ is a left
            $\modulus$-symmetrised cointegral.
    \end{itemize}
\end{thm}
The proof is via monad isomorphisms as in the second part of
\Cref{thm:MainThmSectionStatement} and can be found in \Cref{app:Proof1}.

\medskip

\section{Examples}
    \label{SEC:Examples}
    Here we give examples of quasi-Hopf algebras and their cointegrals.
    Our examples are mostly non-unimodular;    
    some unimodular examples can be found e.g.\ in \cite[Ex.~3.7]{BC2-2011} and
    \cite{BGR1}.
    All examples below are considered over the complex numbers $\field[C]$.
    
    \begin{example}
        This is example 2.2 and 3.4 in \cite{BC2-2011}.
        Consider the unital $\field[C]$-algebra generated by $g$ and $x$, obeying
        relations $g^2 = \oneQ$, $x^4 = 0$ and $gxg\inv = -x $. 
        Define two orthogonal idempotents $p_\pm = \tfrac{1}{2} ( \oneQ \pm g )$.
        The comultiplication and counit are given on generators by
        \begin{alignat}{2}
            \Delta(g) &= g\otimes g,\quad &&\counit(g) =1 \notag \\
            \Delta(x) &= x\otimes (p_+ \pm i p_-) + \oneQ\otimes p_+x + g\otimes p_-x,
            \quad &&\counit(x) = 0 \,
        \end{alignat}
        and with $\coassQ=\invCoassQ=\oneQ\otimes \oneQ\otimes \oneQ - 2 p_-\otimes
        p_-\otimes p_-$ we obtain two 8-dimensional quasi-bialgebras, denoted
        $H_\pm(8)$, both of which admit an antipode $S(g)=g$, $S(x)=-x(p_+ \pm i
        p_-)$, and with evaluation and coevaluation element $\alphaQ=g$ and
        $\betaQ=1$, respectively.

        A right integral is given by $c^r = x^3 p_+$, while $c^l = p_+ x^3 = x^3 p_-$ is a
        left integral.
        One computes that the modulus of $H_\pm(8)$ is $\modulus(x)=0$,
        $\modulus(g)=-1$.
        Thus, the quasi-Hopf algebra in this example is not unimodular.
        Note also that $\modulus = \modulus\inv$.

        Finally, we want to give the cointegrals of $H_\pm(8)$.
        Basis elements of $H_\pm(8)$ are of the form $B_{m,n}=g^m x^n$, $0\leq m \leq
        1$, $0\leq n \leq 3$, and we denote the element dual to $B_{m,n}$ by $B_{m,n}^*$.
        The Drinfeld twist is given by 
        \begin{align}
            \Dt^{\pm 1} = 2 p_+ \otimes p_+ - g \otimes g
        \end{align}
        and so we obtain the right monadic cointegral
        \begin{align}
            \coint^{r,\textup{mon}} 
            = B_{0,3}^* \pm i B_{1,3}^* \ .
        \end{align}
        Concretely, one solves \eqref{eq:qHopf-rmco-lin} and finds that its
        solution space is one-dimensional, so its elements automatically are
        morphisms in $\hmodM$.        
        With the Hopf monad isomorphism $\kappa_{2,3}$ from
        \Cref{prop:IsoBetweenHopfMonads-qHopf} it is then easy to show that
        \begin{align}
            \coint^{l,\textup{mon}} 
            = B_{1,3}^*
        \end{align}
        is a non-zero left monadic cointegral. Using the isomorphisms from 
        \Cref{thm:MainThmSectionStatement},
        we obtain the `classical' right and left cointegrals,
        \begin{align}
            \coint^r 
            = B_{0,3}^* \mp i B_{1,3}^* 
            \quad \text{and} \quad
            \coint^l
            = B_{0,3}^* \ .
        \end{align}
        The same expression for the left cointegral was also derived in
        \cite[Ex.~3.9]{BC2-2011}.

        Note that $H_\pm(8)$ is not pivotal.
        Indeed, one easily checks that already for the generator $x$, $S^2(x)h = h x$
        implies $h=0$, so that $S^2$ is not inner.
    \end{example}

    \begin{example}
        Fix $N\in \field[N],$ $\beta\in \field[C]$ satisfying $\beta^4 = (-1)^N$.
        This example is based on the symplectic fermion ribbon quasi-Hopf algebra
        $\sympFerm(N,\beta)$ from \cite[Sec.\,3]{FGR2}.
        $\sympFerm(N,\beta)$ is factorisable, so in particular unimodular and its
        cointegrals were already discussed in \cite[Sec.\,5]{BGR1}.
        We now restrict to the sub-quasi-Hopf algebra $H(N,\beta) \subset
        \sympFerm(N,\beta)$, which is defined as follows.
        As a unital $\field[C]$-algebra, it is generated by $\elQ{K},
        \elQ{f}_i$, $1\leq i \leq N$, with defining relations
        \begin{align}
            \{ \elQ{f}_i, \elQ{K} \} = 0 \ , \quad \{\elQ{f}_i, \elQ{f}_j \} = 0 \ ,
            \quad \elQ{K}^4 = \oneQ \ , 
        \end{align}
        where $\{a,b\}=ab+ba$ is the anticommutator.

        A PBW-type basis of $H(N,\beta)$ is\footnote{
            We use the convention $\prod_{i=1}^n a_i = a_1 \cdot \prod_{i=2}^n a_i$ for
            the non-commutative product.
        }
        \begin{align}
            \left\{
                B_{\vec{j},l} = 
                \bigg( \prod_{i=1}^N \elQ{f}_i^{j_i} \bigg)  \elQ{K}^l
                ~\bigg|~
                \vec{j} \in \{0,1\}^N,~~ 0\leq l \leq 3
            ~\right\}
            \ .
        \end{align}
        Elements in the corresponding dual basis are simply decorated by an asterisk.

        Using the orthogonal central idempotents $\eQ_0 = \tfrac{1}{2}(\oneQ +
        \elQ{K}^2)$ and $\eQ_1 = \oneQ - \eQ_0$, and setting $\omega = (\eQ_0 + i \eQ_1)
        \elQ{K}$, the comultiplication and the counit are
        \begin{alignat}{3}
            \Delta(\elQ{K}) &= \elQ{K} \otimes \elQ{K} - (1 + (-1)^N)\eQ_1
            \elQ{K}\otimes \eQ_1 \elQ{K} \ ,
            \quad \quad && \counit(\elQ{K}) &&= 1 \ ,
            \notag \\
            \Delta(\elQ{f_i}) &= \elQ{f}_i \otimes \oneQ + \omega \otimes \elQ{f}_i \ ,
            && \counit(\elQ{f}_i) &&= 0 \ .
        \end{alignat}
        The coassociator and its inverse are
        \begin{align}
            \coassQ^{\pm 1} = \oneQ \otimes\oneQ \otimes \oneQ +
            \eQ_1 \otimes \eQ_1 \otimes \{ \eQ_0 (\elQ{K}^N - \oneQ)
            + \eQ_1 (\elQbold{\beta}_\pm - \oneQ) \} \ ,
        \end{align}
        where $\elQbold{\beta}_\pm = \eQ_0 + \beta^2 (\pm i \elQ{K})^N \eQ_1$.
        The evaluation and coevaluation elements are $\alphaQ = \oneQ$, $\betaQ =
        \elQbold{\beta}_+$, and the antipode is
        \begin{align}
            S(\elQ{K}) = \elQ{K}^{(-1)^N} \ , \quad 
            S(\elQ{f}_i) = \elQ{f}_i (\eQ_0 + (-1)^N i \eQ_1) \elQ{K} \ . 
        \end{align}
        Then, with $X = \oneQ + \elQ{K} + \elQ{K}^2 + \elQ{K}^3$, we see that
        \begin{align}
            c^l = X \prod_{i=1}^N \elQ{f}_i, \quad 
            c^r = \left(\prod_{i=1}^N \elQ{f}_i \right) X
        \end{align}
        are a left and a right integral, respectively.
        From this, one easily computes that the modulus is the algebra homomorphism
        given on generators by 
        \begin{align}
            \modulus(\elQ{K})= (-1)^N, \quad \modulus(\elQ{f}_i) = 0.
        \end{align}
        In particular, $H(N,\beta)$ is unimodular if and only if $N$ is even.
        Note that just as in the previous example $\modulus = \modulus\inv$.

        Now we describe the cointegrals of $H(N,\beta)$.
        The Drinfeld twist and its inverse are given by
        \begin{align}
            \Dt^{\pm 1} 
            = \eQ_0 \otimes \oneQ
            + \eQ_1 \otimes \eQ_0 \elQ{K}^N 
            + \eQ_1 \betaQ_\mp \otimes \eQ_1\ ,
        \end{align}
        see also \cite[(3.35)]{FGR2}.
        We again find the right monadic cointegral via \eqref{eq:qHopf-rmco-lin}
        and then obtain the left monadic cointegral via the isomorphism of Hopf monads
        from \Cref{prop:IsoBetweenHopfMonads-qHopf}:
        \begin{align}
            \coint^{r,\textup{mon}}
            = B_{\vec{N},0}^*
            \quad \text{and} \quad
            \coint^{l,\textup{mon}}
            = 
            \delta_{N,\text{even}} B_{\vec{N},0}^*       
        + \delta_{N,\text{odd}} (B_{\vec{N},1}^* - i B_{\vec{N},3}^*)
            \ ,
        \end{align}
        where $\vec{N}$ is the multi-index consisting only of $1$s.
        In particular, the left and the right monadic cointegral do not agree unless $N$
        is even.
        
        With our main theorem, we obtain the right and the left quasi-Hopf cointegral
        \begin{align}
            \coint^r
            &=
            a^r_+ B_{\vec{N},0}^* + a^r_- B_{\vec{N},2}^*
            - \delta_{N,\text{odd}} ( B_{\vec{N},1}^* - B_{\vec{N},3}^* ), 
            \notag \\
            \coint^l 
            &=
            a^l_+ B_{\vec{N},0}^* + a^l_- B_{\vec{N},2}^*
            - \delta_{N,\text{odd}} ( B_{\vec{N},1}^* + B_{\vec{N},3}^* ),
        \end{align}
        where the coefficients are
        \begin{align}
            a^r_\pm = 
            \delta_{N, \text{even}} ( 1 \pm \beta^2 )
            + \delta_{N, \text{odd}} \beta^2 i 
            \quad \text{and} \quad
            a^l_\pm = \delta_{N, \text{even}} \pm \beta^2\ .
        \end{align}
    \end{example}

    \begin{example}
        Fix an odd integer $t$, let $p\geq 2$ be an integer, and set $q=e^{i\pi/p}$.
        Here we consider the quasi Hopf modification $\bar{U}^{(\coassQ)}_q
        \mathfrak{sl}(2)$ of the restricted quantum group $\bar{U}_q\mathfrak{sl}(2)$ from
        \cite[Sec.\,4]{CGR}.
        This quasi-Hopf algebra is factorisable, and
        as in the previous example we will
        consider a non-unimodular sub-quasi-Hopf algebra~$U^-$.\footnote{
            Different values of $t$ lead to twist-equivalent quasi-bialgebras, cf.\
            \cite[Thm.\,4.1]{CGR}.  
            Explicitly, with $ J(t',t) = \eQ_0 \otimes \oneQ + \eQ_1\otimes K^{(t'-t)/2} $
            we get e.g.\ $ J(t',t)\cdot \Delta_t(h) \cdot J(t',t)\inv = \Delta_{t'}(h) $
            for all $h\in U^-$. 
        }
        As a $\field[C]$-algebra, it is generated by $F$ and $K$, with defining relations
        \begin{align}
            F^p = 0, \quad K^{2p}=\oneQ ,\quad\text{and}\quad KFK\inv = q^{-2} F,
        \end{align}
        and a natural choice of basis of $U^-$ is therefore
        \begin{align}
            \{
                B_{m,n} = F^m K^n \mid 0\leq m\leq p-1, \quad 0\leq n\leq 2p-1 
            \}\ .
        \end{align}

        Using the two central idempotents $\eQ_0 = \tfrac{1}{2}(\oneQ + K^p)$ and
        $\eQ_1=\oneQ - \eQ_0$, the comultiplication and the counit are
        \begin{alignat}{2}
            \Delta_t(F) &= F\otimes \oneQ + (\eQ_0 + q^{-t} \eQ_1)K\inv\otimes F,
            \quad \qquad && \counit(F) = 0,
            \notag \\
            \Delta_t(K) &= K\otimes K,
            \quad && \counit(K) = 1.
        \end{alignat}
        The coassociator and the antipode are given by
        \begin{align}
            \coassQ_t &= \oneQ \otimes \oneQ \otimes \oneQ + \eQ_1 \otimes \eQ_1
            \otimes (K^{-t} - \oneQ) 
            \notag \\
            \quad S_t(F) &= -KF(\eQ_0 + q^{-t} \eQ_1), \quad S_t(K) = K\inv,
        \end{align}
        and finally, evaluation and coevaluation element are 
        \begin{align}
            \alphaQ = \oneQ, \quad \betaQ_t = \eQ_0 + K^{-t} \eQ_1\, 
        \end{align}
        respectively.
        
        \smallskip
        From these data, one finds that the Drinfeld twist and its inverse are
        \begin{align}
            \Dt^{\pm 1}
            =  \eQ_0 \otimes \oneQ
            + \eQ_1 \otimes \eQ_0 K^{\mp t}
            + \eQ_1 K^{\pm t} \otimes \eQ_1
            \ ,
        \end{align} 
        see~\eqref{eq:DrinfeldTwist-closedForm}.
        This quasi-Hopf algebra is pivotal, and the pivot we choose is 
        \begin{align}
            \pivotQ_t = \eQ_0 K - \eQ_1 K^{t+1} \ .
        \end{align}

        \smallskip

        Set $X=\sum_{i=0}^{2p-1} K^i$.
        Then one can see that $c^r = F^{p-1}X$ and $c^l = X F^{p-1}$ are a right and a
        left integral for $U^-$, respectively.
        From
        \begin{align}
            c^l F &= 0 \ , \quad c^l K = q^{-2} c^l \ , \notag \\
            F c^r &= 0 \ , \quad K c^r = q^2 c^r
        \end{align}
        we can see that $U^-$ is non-unimodular.
        The modulus is
        \begin{align}
            \modulus(F) = 0 \ , 
            \quad 
            \modulus(K) = q^{-2}\ .
        \end{align}
        The order of $\modulus$ is $p$, and in particular $\modulus \neq \modulus\inv$ if
        $p>2$.

        \smallskip

        Using \eqref{eq:qHopf-rmco-lin}
        one verifies that the space of right monadic cointegrals is
        \begin{align}
            \spaceRightMonCoint = \field[C]B_{p-1,0}^*\ .
        \end{align}
        The left monadic cointegral can then be found via the isomorphism from
        \Cref{prop:IsoBetweenHopfMonads-qHopf}.
        Normalizing the result, we obtain
        \begin{align}
            \coint^{l,\textup{mon}} 
            = (1 + q^{-t(p-1)}) B_{p-1,p-1}^*
            + (1 - q^{-t(p-1)}) B_{p-1,2p-1}^*\ .
        \end{align}
        Also from \Cref{prop:IsoBetweenHopfMonads-qHopf} we obtain the right
        $\distInvObj$-symmetrised monadic cointegral
        \begin{align}
            \coint^{r,\distInvObj-\textup{sym}}
            =
            B_{p-1,p-1}^*
            + B_{p-1,2p-1}^*
            + q^{2t} 
            \big(
                B_{p-1,p-t-1}^*
                - B_{p-1,2p-t-1}^*
            \big)
        \end{align}
        and the left $\distInvObj$-symmetrised monadic cointegral
        \begin{align}
            \coint^{l,\distInvObj-\textup{sym}}
            =
            B_{p-1,0}^*
            + B_{p-1,p}^*
            + q^{-t(p+1)} 
            \big(
                B_{p-1,t}^*
                - B_{p-1,p+t}^*
            \big)\ .
        \end{align}

        The `classical' \cite{HN-integrals}-cointegrals of $U^-$ can now be obtained using
        \Cref{thm:MainThmSectionStatement}.
        The right and the left cointegral are found to be
        \begin{align}
            \coint^{r}
            =
            B_{p-1,0}^*
            + B_{p-1,p}^*
            + B_{p-1,t}^*
            - B_{p-1,p+t}^*
        \end{align}
        and
        \begin{align}
            \coint^{l}
            =
            B_{p-1,p-1}^*
            + B_{p-1,2p-1}^*
            + q^{-t(p-1)} 
            \big(
                B_{p-1,p-t-1}^*
                - B_{p-1,2p-t-1}^*
            \big)\ , 
        \end{align}
        respectively.
        Finally, we give the $\modulus$-symmetrised cointegrals of $U^-$ using the
        characterization \eqref{eq:symCointViaCoint}.
        The right $\modulus$-symmetrised cointegral is given by
        \begin{align}
            \widehat{\coint}^r 
            = B_{p-1,p-1}^* \ ,
        \end{align}
        and the left $\modulus$-symmetrised cointegral is 
        \begin{align}
            \widehat{\coint}^l
            = 
            (1 + q^{-t(p-1)}) B_{p-1,0}^*
            + (1 - q^{-t(p-1)}) B_{p-1,p}^*\ .
        \end{align}

        \smallskip

        For $p=2$ and $t=1$, $U^-$ is, as a quasi-Hopf algebra, isomorphic to
        \begin{align}
            H(N=1,\beta=\exp(i\tfrac{\pi}{4}))
        \end{align}
        from the previous example, by mapping generators according to
        \begin{align}
            F \mapsto i \elQ{f},\quad K\mapsto \elQ{K},
        \end{align}
        cf.\ \cite[Rem.\,4.3(2)]{CGR}.
        Note that, under this isomorphism, the cointegrals of $U^-$ agree with those of
        $H(N=1, \beta=\exp(i \tfrac{\pi}{4}))$.
    \end{example}

\medskip

\section{Cointegrals for the coend in the braided case}\label{SEC:braided_case}

In a braided category $\cat$ there exist notions of integrals and cointegrals for Hopf
algebras internal to $\cat$.
If $\cat$ is in addition finite tensor, then the coend $\coend = \int^{X \in \cat}
\dualL{X} \otimes X$ is an example of such a Hopf algebra \cite{LyuMaj,
Lyubashenko-mod-transfs}.
In this section we relate left and right integrals for $\coend$ to right monadic cointegrals and
consider quasi-triangular quasi-Hopf algebras as an example.

\smallskip

In this section, let $\cat$ be a braided finite tensor category.

\subsection{Integrals and cointegrals for Hopf algebras in \texorpdfstring{$\cat$}{C}}

Let $A$ be a Hopf algebra in
    $\cat$ with invertible antipode, see e.g.~\cite{KerlerLyubashenko}
    or~\cite[Sec.\,2.2]{FGR1}.
    Then the notions of (left/right) integrals and (left/right) cointegrals are
    well-defined, see \cite[Prop.~4.2.4]{KerlerLyubashenko}.
    We repeat the definition of a left integral for $A$.
    It consists of an invertible object $\objectOfInts$, the \emph{object of
    integrals}, and a morphism $\intLyu_A: \objectOfInts \to A$ making the diagram 
    \begin{equation}
        \label{diag:LyuLeftIntegral}
        \begin{tikzcd}[column sep=large]
            \hspace{-43.0pt}A \otimes \objectOfInts
            \ar[r,"\id\otimes \intLyu_A"] 
            \ar[d,swap,"\counit \otimes \id", 
            ]
            & A\otimes A \ar[dd,"m"] \\
            \tensUnit \otimes \objectOfInts
            \ar[d,"\sim",swap] & \\[-10pt]
            \objectOfInts
            \ar[r,"\intLyu_A",swap]
            & A 
        \end{tikzcd}
    \end{equation}
    commute.
    Here, $m$ and $\counit$ are multiplication and counit of the Hopf algebra $A$,
    respectively.
    Right integrals are defined similarly (with the same object $\objectOfInts$).
    It is known that non-zero (left/right) integrals $\intLyu_A$ exists and are uniquely
    determined up to scalar \cite[Prop.~4.2.4]{KerlerLyubashenko}.
    Note that the above diagram is just the statement that a left integral for $A$ is a
    morphism $\intLyu_A : \objectOfInts \to A$ of left $A$-modules, where the $A$-actions
    are given by the left and right side of the diagram \eqref{diag:LyuLeftIntegral},
    respectively.    

    As remarked in \cite[Ex.~3.10]{BV-hopfmonads}, tensoring with a Hopf algebra with
    invertible antipode in a braided category yields a Hopf monad.
    The category of modules over the Hopf algebra is then the same as the category of
    modules over the corresponding Hopf monad.

\subsection{Central Hopf monad via the coend in braided case}
    \renewcommand{\objectOfInts}[1][\coend]{\operatorname{Int}#1}
    In the braided setting, the coend 
$\coend = \int^{X \in \cat} \dualL{X} \otimes X$
    with universal dinatural transformation $\dinatLyu$ becomes a Hopf algebra
    \cite{LyuMaj,Lyubashenko-mod-transfs}, see also \cite{Fuchs-Schweigert,FGR1} for a
    review.
    It is easy to see that $\hopfMonad[2]$ is in fact isomorphic to the Hopf monad
    obtained by tensoring with~$\coend$. 
    The isomorphism $\xi_V: \hopfMonad[2](V) \to \coend \otimes V$ we choose is  obtained
    via
    \begin{align}
        \label{eq:isoLyuCoendHopfMonad}
        \xi_V \circ \dinatCoend[2][V]_X = 
        \left[
            \hopfMonadComp{2}
        \xrightarrow{ \id \otimes \braiding_{X,V}\inv }
            \dualL{X} \left( X V \right) \xrightarrow{\sim} 
            \left( \dualL{X} X \right) V 
            \xrightarrow{\dinatLyu_X \otimes \id }
            \coend V
        \right].
    \end{align}
    The inverse of the braiding appears to make $\xi$ an isomorphism of bimonads, with the
    bimonad structure on $\coend \otimes \placeholder$ inherited from the bialgebra
    structure on $\coend$ as defined in \cite[Sec.~3.3]{FGR1}.
In the same way, $\placeholder \otimes \coend$ becomes a bimonad and we get a bimonad
isomorphism $\zeta: \hopfMonad[2] \To (\placeholder \otimes \coend)$ via
\begin{align}
	\label{eq:isoLyuCoendHopfMonad-right}
	\zeta_V \circ \dinatCoend[2][V]_X = 
	\big[
	\hopfMonadComp{2}
	&\xrightarrow{\sim}
	\left( \dualL{X} V \right) X
	\xrightarrow{ \braiding_{V,\dualL{X}}\inv \otimes \id }
	\left( V \dualL{X} \right) X
	\notag \\
	&\xrightarrow{\sim}
	V \left( \dualL{X} X \right) 
	\xrightarrow{\id \otimes \dinatLyu_X}
	V \coend
	\big] \ .
\end{align}
Again the inverse braiding is required to make $\zeta$ a bimonad morphism.

    In a finite tensor category, an invertible object has isomorphic left and right duals,
    and so in particular there is an up to scalars unique isomorphism $\dualL{\distInvObj}
    \xrightarrow{\sim} \dualR{\distInvObj}$.
    This fact is used in formulating the following proposition.

    \begin{prop}
        \label{prop:MonadicCointIsLyuInt}
        Let $\cat$ be a braided finite tensor category.
        \begin{enumerate}
            \item
                The distinguished invertible object is dual to the object of integrals for
                $\coend$,
                \begin{align}
                    \distInvObj \cong \dualL{(\objectOfInts)}.
                \end{align}
            \item 
                Let ${\intLyu_{\coend}:\dualR{\distInvObj} \to \coend}$ be non-zero.
                Then
                \begin{equation*}
                    \intLyu_{\coend} \text{  is a  } 
                    \begin{cases}
                        \text{left integral for $\coend$ in the sense of
                        \eqref{diag:LyuLeftIntegral}, resp.}\\
                        \text{right integral for $\coend$}
                    \end{cases}
                \end{equation*}
                if and only if
                \begin{align}
                    \label{eq:MonCointFromLyuInt}
                    \coint \defined
                    \begin{cases}
                        \big[
                            \tensUnit \xrightarrow{ \coevR_{\distInvObj} }
                            \dualR{\distInvObj} \otimes \distInvObj
                            \xrightarrow{ \intLyu_\coend \otimes \id}
                            \coend \otimes \distInvObj
                            \xrightarrow{\xi_{\distInvObj}\inv}
                            \hopfMonad[2][\distInvObj]
                        \big] \text{ , resp.}
                        \\
                        \big[
                            \tensUnit \xrightarrow{ \coevL_{\distInvObj} }
                            \distInvObj \otimes \dualL{\distInvObj}
                            \xrightarrow{\sim}
                            \distInvObj \otimes \dualR{\distInvObj}
                            \xrightarrow{ \id \otimes \intLyu_\coend}
                            \distInvObj \otimes \coend 
                            \xrightarrow{\zeta_{\distInvObj}\inv}
                            \hopfMonad[2][\distInvObj]
                        \big]
                    \end{cases}
            \end{align}
            is a non-zero right monadic cointegral of $\cat$.
        \end{enumerate}
    \end{prop}

The first statement was already observed in \cite[Thm.~6.8]{Sh-unimodFinTens}.

    \begin{proof}
        We will only treat the case of left integrals for $\coend$ explicitly, the case of
        right integrals can be shown analogously.    
    
        Let us abbreviate $X=\objectOfInts$, so that by existence of left cointegrals we
        can find a non-zero morphism $\intLyu_\coend : X \to \coend$ which satisfies
        \eqref{diag:LyuLeftIntegral}.
        We now define $\coint$ as in part (2), but with $X$ instead of $\dualR{D}$:
        \begin{align}
            \label{eq:MonCointFromLyuInt-aux1}
            \coint \vcentcolon=
            \big[
                \tensUnit \xrightarrow{ \coevL_{\!\!X} }
                X \otimes \dualL{X}
                \xrightarrow{ \intLyu_\coend \otimes \id}
                \coend \otimes \dualL{X}
                \xrightarrow{\xi_{\dualL{X}}\inv}
                \hopfMonad[2][\dualL{X}]
            \big] \ .
        \end{align}
        Note that $\coint$ is non-zero, too.

        The somewhat lengthy computation below will establish that $\coint$
        from~\eqref{eq:MonCointFromLyuInt-aux1} is an $\hopfMonad[2]$-inter\-twiner. 
        By \cite[Lem.~4.1]{Sh-integrals} and \Cref{cor:OurCointIsShimizus}
        the distinguished invertible object~$\distInvObj$ is the unique (up to unique
        isomorphism) invertible object such that the space of $\hopfMonad[2]$-intertwiners
        from $\tensUnit$ to $\hopfMonad[2][\distInvObj]$ is 
        non-zero.
        Thus we must have $\dualL{X} \cong D$, proving part (1).
        Together with part (1), the fact that $\coint$ is an $\hopfMonad[2]$-inter\-twiner
        implies that it is a right monadic cointegral, proving the direction $\Rightarrow$
        of part (2). 
        The direction $\Leftarrow$ of part (2) can be verified by an analogous
        computation, where a right monadic cointegral $\coint$ gets mapped to a left
        integral of $\coend$ via
        \begin{align}
            \label{eq:LamL-br-from-mon}
            \intLyu_{\coend} \defined
            \bigg[
                \dualR{\distInvObj}[-0.3]
                \xrightarrow{\sim}
                \tensUnit \otimes \dualR{\distInvObj}[-0.3] 
                \xrightarrow{ \coint \otimes \id }
                \hopfMonad[2][\distInvObj] \otimes \dualR{\distInvObj}[-0.3] 
                \xrightarrow{\xi_{\distInvObj} \otimes \id}
                &(\coend \distInvObj) \dualR{\distInvObj}[-0.3]
                \xrightarrow{\sim} 
                \coend (\distInvObj \dualR{\distInvObj}[-0.3]) 
                \notag \\
                &\xrightarrow{ \id \otimes \evR_{\distInvObj} }
                \coend \otimes \tensUnit \xrightarrow{\sim} 
                \coend
            \bigg] .
        \end{align}
        Note that this is indeed inverse to \eqref{eq:MonCointFromLyuInt}.

        Let us now turn to the verification that $\coint$ in
        \eqref{eq:MonCointFromLyuInt-aux1} is indeed an $\hopfMonad[2]$-inter\-twiner.
        Note that since $\xi$ is an isomorphism of bimonads, it satisfies
        \begin{align}\label{eq:xi-is-monad-iso-MULT}
            &\big[
            \coend (\coend V) 
            \xrightarrow{\sim}
            (\coend \coend) V
            \xrightarrow{ m \otimes \id_V }
            \coend V
            \xrightarrow{ \xi\inv_V }
            \hopfMonad[2][V]
            \big]
            \notag
            \\
            &= 
            \big[
            \coend (\coend V) 
            \xrightarrow{ \xi\inv_{\coend V} } 
            \hopfMonad[2][\coend V]
            \xrightarrow{ \hopfMonad[2][\xi\inv_V] }
            (\hopfMonad[2])^2(V)
            \xrightarrow{ \multCoend[2][V] }
            \hopfMonad[2][V]
            \big] \ ,
        \end{align}
        and
        \begin{align}
            \label{eq:xi-is-monad-iso-COUNIT}
            \big[ 
            \coend \tensUnit
            \xrightarrow{ \xi\inv_{\tensUnit} }
            \hopfMonad[2][\tensUnit]
            \xrightarrow{ \counitCoend[2] } 
            \tensUnit
            \big]
            = 
            \big[
                \coend \tensUnit \xrightarrow{ \varepsilon \otimes \id_{\tensUnit} }
                \tensUnit\tensUnit
                \xrightarrow{ \sim } \tensUnit 
        \big] \ ,
        \end{align}
        where $m$ and $\varepsilon$ are the multiplication and the counit of $\coend$.

        For the next calculation, let us explicitly denote components of the left unitor
        and the associator by 
        \begin{align}
            l_V: \tensUnit V \to V
            \quad \text{and} \quad
            \alpha_{U,V,W}: U(VW)\to (UV)W,
        \end{align}
        respectively, for $U,V,W\in \cat$.
        Then
        \begin{align}
            &\multCoend[2][\dualL{X}] \circ \hopfMonad[2][\coint]
            \notag \\
        &\overset{\eqref{eq:MonCointFromLyuInt-aux1}}=
            \multCoend[2][\dualL{X}] \circ 
            \hopfMonad[2]\big( \xi\inv_{\dualL{X}} \circ 
                (\intLyu_{\coend} \otimes \id_\dualL{X})
            \circ \coevL_X \big)
            \notag \\
            &\overset{\eqref{eq:xi-is-monad-iso-MULT}}=
            \xi\inv_{\dualL{X}} \circ 
            (m\otimes \id_\dualL{X}) 
            \circ \alpha_{\coend, \coend, \dualL{X}}
            \circ \xi_{\coend \otimes \dualL{X}}
            \circ \hopfMonad[2]\big( (\intLyu_{\coend} \otimes \id_\dualL{X})
            \circ \coevL_X \big)
            \notag \\
            &\overset{\xi\,\text{nat.}}=
            \xi\inv_{\dualL{X}} \circ 
            ( m\otimes \id_\dualL{X} )
            \circ \alpha_{\coend, \coend, \dualL{X}}
            \circ (\id_\coend \otimes (\intLyu_{\coend} \otimes \id_{\dualL{X}}))
            \circ (\id_\coend \otimes \coevL_X )
            \circ \xi_{\tensUnit}
            \notag \\
            &\overset{\alpha\,\text{nat.}}=
            \xi\inv_{\dualL{X}} \circ 
            \big( (m \circ (\id_\coend \otimes \intLyu_{\coend}) )\otimes \id_\dualL{X} \big)
            \circ \alpha_{\coend, X, \dualL{X}}
            \circ (\id_\coend \otimes \coevL_X )
            \circ \xi_{\tensUnit}
            \notag \\
            &\overset{\eqref{diag:LyuLeftIntegral}}=
            \xi\inv_{\dualL{X}} \circ 
            \big( 
            ( \intLyu_{\coend} \circ l_X \circ (\varepsilon \otimes \id_X) )
            \otimes \id_\dualL{X}
            \big)
            \circ \alpha_{\coend, X, \dualL{X}}
            \circ (\id_\coend \otimes \coevL_X )
            \circ \xi_{\tensUnit}
            \notag \\
            &\overset{ \alpha\,\text{nat.} }=
            \xi\inv_{\dualL{X}} \circ 
            \big( 
            ( \intLyu_{\coend} \circ l_X )
            \otimes \id_\dualL{X}
            \big)
            \circ \alpha_{\tensUnit, X, \dualL{X}}
            \circ (\id_{\tensUnit} \otimes \coevL_X )
            \circ (\varepsilon \otimes \id_{\tensUnit})
            \circ \xi_{\tensUnit}
            \notag \\
            &\overset{ \text{coher.} }=
            \xi\inv_{\dualL{X}} \circ 
            (\intLyu_{\coend} \otimes \id_\dualL{X})
            \circ l_{X \dualL{X}}
            \circ (\id_{\tensUnit} \otimes \coevL_X )
            \circ (\varepsilon \otimes \id_{\tensUnit})
            \circ \xi_{\tensUnit}
            \notag \\
            &\overset{ l\,\text{nat.} }=
            \xi\inv_{\dualL{X}} \circ 
            (\intLyu_{\coend} \otimes \id_\dualL{X})
            \circ \coevL_X 
            \circ \ l_{\tensUnit}
            \circ (\varepsilon \otimes \id_{\tensUnit})
            \circ \xi_{\tensUnit}
            \notag \\
            &\overset{\eqref{eq:xi-is-monad-iso-COUNIT}}=\,
            \xi\inv_{\dualL{X}} \circ 
            (\intLyu_{\coend} \otimes \id_\dualL{X})
            \circ \coevL_X 
            \circ \counitCoend[2]
            \notag \\
        &\overset{\eqref{eq:MonCointFromLyuInt-aux1}}= 
            \coint \circ \counitCoend[2],
        \end{align}
        showing that $\coint$ is an $\hopfMonad[2]$-inter\-twiner.
        \end{proof}

By composing with $\kappa_{3,2}$ one obtains analogous statements to those in the above
proposition for left monadic cointegrals. 

\begin{rmrk}\label{rem:unimod-left=right}
    Let $\cat$ be a unimodular braided finite tensor category.
    Then $\distInvObj = \tensUnit$, and by \Cref{prop:MonadicCointIsLyuInt},
    the object of integrals of $\coend$ is the tensor unit.
    It follows from the coherence of braided monoidal categories that
    \begin{align}
        \xi_\tensUnit = 
        \big[ 
            \hopfMonad[2][\tensUnit] \xrightarrow{ \zeta_\tensUnit }
            \tensUnit \coend \xrightarrow{ \sim }
            \coend \tensUnit
        \big]
    \end{align}
    for the bimonad isomorphisms $\xi$ and $\zeta$ from \eqref{eq:isoLyuCoendHopfMonad}
    and \eqref{eq:isoLyuCoendHopfMonad-right}.
    Now one can show that left and right integrals for $\coend$ agree.
    Indeed, the composition
    \begin{align}
        \{ \text{right $\coend$-integrals} \}
        \xrightarrow{ \eqref{eq:MonCointFromLyuInt} }
        \cat_{\hopfMonad[2]}( \tensUnit, \hopfMonad[2][\tensUnit] )
        \xrightarrow{ \eqref{eq:LamL-br-from-mon} }
        \{ \text{left $\coend$-integrals} \}
    \end{align}
    of isomorphisms is proportional to the identity on the one-dimensional subspace of
    right $\coend$-integrals of $\cat(\tensUnit, \coend)$.
    Therefore, a right integral for $\coend$ is also left and vice versa.
    This result has also been shown by different means in
    \cite[Thm.\,6.9]{Sh-unimodFinTens}.    
\end{rmrk}

\smallskip

\subsection{Quasi-triangular quasi-Hopf algebras}\label{sec:qtriang-qHopf}

Let $H$ be a finite-dimensional quasi-tri\-an\-gu\-lar qua\-si-Hopf algebra with universal
$R$-matrix $\rMatrix$, see e.g.~\cite[Sec.\,6]{FGR1} for details in the same notation as
used here.
We denote the multiplicative inverse of the $R$-matrix by $\rMatrixInv$.
The category $\cat = \hmodM$ is a braided finite tensor category.

The coend $\coend$ can be realised by $H^*$ with the coadjoint action, see e.g.\
\cite[Sec.~7]{FGR1}, and with our realization of the Hopf monad $\hopfMonad[2]$ as in
\Cref{subseq:HopfMonad-qHopf} 
we get the following formula for the Hopf monad
isomorphism $\hopfMonad[2] \cong \coend \otimes \placeholder$ from
\eqref{eq:isoLyuCoendHopfMonad}.

\begin{lemma}
    \label{lem:isoLyuCoendHopfMonad-qHopf}
    The isomorphism $\xi_V : \hopfMonad[2][V] \to \coend \otimes V$ from
    \eqref{eq:isoLyuCoendHopfMonad} is given by
    \begin{align}
        \label{eq:isoLyuCoendHopfMonad-qHopf}
        \xi_V(f\otimes v) =
        \langle f \mid S(\coassQ[1]_1)\placeholder \coassQ[1]_2 \rMatrixInv_1 \rangle
        \otimes \coassQ[1]_3 \rMatrixInv_2.v
    \end{align}
    for $V\in \cat$, $f\in H^*$, $v\in V$.
\end{lemma}
The proof is a straightforward computation.

\medskip

Next, we give the explicit formulas relating right monadic cointegrals and left integrals
for the coend.
As usual, we identify linear maps $\field \to V$ with elements in $V$.

\begin{lemma}
    \label{lem:intCointIsMonadicCoint-qHopf}
    Let $\coint \in H^*$ be a right monadic cointegral.
    Then
    \begin{align}\label{eq:braided-v-monadic}
        \intLyu_\coend \defined
        \modulus\inv(\qR_2 \coassQ[1]_3 \rMatrixInv_2)
        \langle
            \coint \mid
            S({\qR_1}\sweedler{1} \coassQ[1]_1) \placeholder 
            {\qR_1}\sweedler{2} \coassQ[1]_2 \rMatrixInv_1
        \rangle
    \end{align}
    is a left integral for the coend $\coend$.
\end{lemma}

The proof amounts to evaluating \eqref{eq:LamL-br-from-mon} in $\hmodM$
using \Cref{lem:isoLyuCoendHopfMonad-qHopf}.
We arrive at the following corollary.

\begin{cor}
    \label{cor:LyubashenkoIntVsBCCoint}
    Let $\coint^r\in H^*$ be a right cointegral for the quasi-Hopf algebra $H$.
    Then
    \begin{align}\label{eq:lyubashenkoIntViaHopfCoint}
        \intLyu_\coend \defined
        \modulus\inv(\qR_2 \coassQ[1]_3 \rMatrixInv_2 S\inv(\Dt\inv_2))
        \langle
            \coint^r \mid
            S({\qR_1}\sweedler{1} \coassQ[1]_1 \betaQ) \placeholder
            {\qR_1}\sweedler{2} \coassQ[1]_2 \rMatrixInv_1 S\inv(\Dt\inv_1)
        \rangle
    \end{align}
    is a left integral for the coend $\coend$.
\end{cor}

\begin{proof}
    Combining \Cref{thm:MainThmSectionStatement} with the previous lemma
    immediately yields the formula.
\end{proof}

Using \Cref{prop:MonadicCointIsLyuInt}\,(2), one can also write  formulas similar
to~\eqref{eq:braided-v-monadic}, resp.~\eqref{eq:lyubashenkoIntViaHopfCoint}, for the
relation between right integrals for $\coend$ and right monadic cointegrals, resp.\ right
cointegrals for $H$.
	We will skip the details.

\begin{rmrk}
    \label{rmrk:unimodqHopf-monLyuCoint}
    Let $H$ be unimodular with right cointegral $\coint^r$.
    Observe that then $\distInvObj = \tensUnit$ and $\hopfMonad[2][\tensUnit] = \coend$
    as $H$-modules.
    The relationship \eqref{eq:braided-v-monadic} between integrals for $\coend$ and right
    monadic cointegrals is now particularly simple:
    \begin{align}\label{eq:unimod-simple-int-moncoint-rel}
        \intLyu_\coend  
        = 
        \coint \ .
    \end{align}
    By \Cref{rem:unimod-left=right}, left and right integrals for $\coend$ coincide,
    and so \eqref{eq:unimod-simple-int-moncoint-rel} says that the right monadic
    cointegral and the left/right integral for $\coend$ are given by the same linear form
    on $H$.
    The relation to right cointegrals for $H$ also simplifies: $\intLyu_\coend = \langle
    \coint^r \mid S(\betaQ) \placeholder \rangle$.
\end{rmrk}

\medskip

\section{Application: \texorpdfstring{$\slTwoZ$}{SL2Z}-action}
\label{SEC:Sl2Z}

\subsection{\texorpdfstring{$\slTwoZ$}{SL2Z}-action for modular tensor categories}
\label{sec:SL2Z-mtc}

In a braided finite tensor category $\cat$, the Hopf algebra $\coend = \int^X
\dualL{X}\otimes X$ admits a Hopf pairing
\begin{align}\label{eq:hopfPairing}
    \hopfPair: \coend \otimes \coend \to \tensUnit\ ,
\end{align}
see \cite{Lyubashenko-mod-transfs} for details or \cite[Sec.~3.3]{FGR1} for a review.
By a \emph{modular tensor category} 
we mean a ribbon finite tensor category which is \emph{factorisable}, that is, in which
the Hopf pairing \eqref{eq:hopfPairing} induces an isomorphism $\coend \cong
\dualL{\coend}$ of Hopf algebras.
Equivalent definitions of factorisability can be found in \cite{Sh-nonDegen}.

Let for the rest of this section $\cat$ be a modular tensor category with ribbon
twist~$\ribTwist$.
Since $\cat$ is factorisable it is in particular
unimodular~\cite[Lem.~5.2.8]{KerlerLyubashenko}, and the Hopf algebra $\coend$ has a
two-sided integral $\intLyu_\coend: \tensUnit \to \coend$ by
\Cref{rem:unimod-left=right}, see also \cite[Cor.~5.2.11]{KerlerLyubashenko}.

\smallskip

Define the morphism $\mathcal{Q} : \coend \otimes \coend \to \coend \otimes \coend$ by
\begin{align}
	\ipic{-0.5}{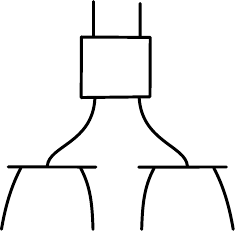}{1.2}
	\put (-053,045) {$\coend$}
	\put (-037,045) {$\coend$}
	\put (-047,012) {$\mathcal{Q}$}
	\put (-089,-13) {$\dinatLyu_X$}
	\put (-003,-13) {$\dinatLyu_Y$}
	\put (-087,-50) {$\dualL{X}$}
	\put (-055,-50) {$X$}
	\put (-037,-50) {$\dualL{Y}$}
	\put (-005,-50) {$Y$}
	\quad = \quad
	\ipic{-0.5}{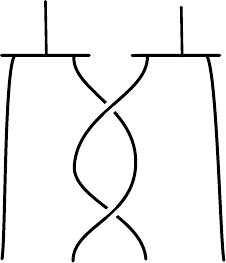}{1.2}
	\put (-66,048) {$\coend$}
	\put (-19,048) {$\coend$}
	\put (-80,031) {$\dinatLyu_X$}
	\put (-00,028) {$\dinatLyu_Y$}
	\put (-83,-56) {$\dualL{X}$}
	\put (-59,-56) {$X$}
	\put (-32,-56) {$\dualL{Y}$}
	\put (-05,-56) {$Y$}
	\ .
\end{align}
This is related to the Hopf pairing $\omega$ via $ \hopfPair = (\counit \otimes \counit)
\circ \mathcal{Q}, $ denoting by $\counit$ the counit of $\coend$.
Next, define $\modS, \modT \in \End_{\cat}(\coend)$ by
\begin{align}
	\modS 
	= (\counit \otimes \id) \circ \mathcal{Q} \circ (\id \otimes \intLyu_\coend)
	\ , \quad 
	\modT \circ \dinatLyu_X 
	= \dinatLyu_X \circ (\id \otimes \ribTwist_X) \ .
\end{align}
These endomorphisms satisfy 
\begin{align}
    (\modS \modT)^3 = \lambda \modS^2 = \lambda S_\coend\inv
\end{align}
where $\lambda$ is a non-zero constant and $S_\coend$ is the antipode of $\coend$
\cite{Lyubashenko-mod-transfs}.

Recall that $\slTwoZ$ is the group generated by
$\mathbf{S}=(\begin{smallmatrix}0&-1\\1&0\end{smallmatrix})$ 
and 
$\mathbf{T}=(\begin{smallmatrix}1&1\\0&1\end{smallmatrix})$ 
together with the relations
\begin{align}
    (\mathbf{ST})^3 = \mathbf{S}^2
    ,\quad \mathbf{S}^4 = \id.
\end{align}

It was shown in \cite{Lyubashenko-mod-transfs} that the $\field$-vector space
$\cat(\tensUnit,\coend)$ carries a projective $\slTwoZ$-action, given by
\begin{align}
    \mathbf{S}.f = \modS \circ f
    \ ,\quad
    \mathbf{T}.f = \modT \circ f
\end{align}
for $f:\tensUnit \to \coend$, see \cite[Sec.~5]{FGR1} for a review.

Equivalently, one also obtains an action on $\cat(\coend,\tensUnit)$, given by
\begin{align}
    \mathbf{S}.f = f \circ \modS 
    \ ,\quad
    \mathbf{T}.f = f \circ \modT
\end{align}
for $f:\coend \to \tensUnit$.

\medskip

\subsection{\texorpdfstring{$\slTwoZ$}{SL2Z}-action for factorisable quasi-Hopf algebras}\label{sec:SL2Z-qHopf}

Let now $H$ be a finite-dimensional factorisable\footnotemark{} ribbon quasi-Hopf algebra
with ribbon element~$\elRibbon$ and $R$-matrix~$\rMatrix$.
As mentioned before, $H$ being factorisable implies that it is unimodular. 
Thus, by \Cref{rmrk:unimodqHopf-monLyuCoint}, the left integral for $\coend$ and the
right monadic cointegral of $\hmodM$ are given by the same linear form on $H$.
\footnotetext{
    A quasi-triangular quasi-Hopf algebra $H$ is factorisable \cite{BT-Factorizable} if
    and only if a certain linear map $\mathfrak{M}:H^* \to H$ involving the monodromy is
    bijective.
    In \cite[Sec.~7.3]{FGR1} it was shown that this is equivalent to $\hmodM$ being
    factorisable.
    In particular, $\hmodM$ is a modular tensor category for $H$ factorisable and ribbon.
}

The linear injection
\begin{align}
    \alphaQ Z = \{\alphaQ z \mid z \in Z(H)\} 
    \to 
    \cat(\coend, \tensUnit)
    , \quad
    \alphaQ z \mapsto 
    \pivotalStruct_{\Vect}(\alphaQ z) = \langle \placeholder \mid \alphaQ z \rangle
\end{align}
can be shown to be an isomorphism\footnote{
    Note also that $\cat(\tensUnit, \coend) \cong \betaQ Z$, which is defined similarly.
},
and so we get an action $\modS_{\alphaQ Z}$ of the $\mathbf{S}$-generator
on $\alphaQ Z$ by setting
\begin{align}
    \modS_{\alphaQ Z}(h) 
    = \pivotalStruct_{\Vect}\inv 
    ( \langle \placeholder \mid h \rangle \circ \modS )
    = \pivotalStruct_{\Vect}\inv 
    ( \langle \modS(\placeholder) \mid h \rangle ),
\end{align}
for $h\in \alphaQ Z$.
Since $\modS\in \End_{\cat}(\coend)$ and $\coend \cong_{\field} H^*$, we can define
$\hat{\modS} \in \End_{\field}(H)$ via
\begin{align}
    \langle f \mid \hat{\modS}(h) \rangle 
    = 
    \langle \modS(f) \mid h \rangle
\end{align}
for all $h\in H$, $f\in H^*$, and it is then immediate that
\begin{align}
    \label{eq:SalphaZ-equals-Shat}
    \modS_{\alphaQ Z} = \hat{\modS}\big|_{\alphaQ Z}.
\end{align}

We will express the Hopf pairing $\hopfPair$ from \eqref{eq:hopfPairing} via an element
$\elOmega\in H\otimes H$ such that
\begin{align}
    \hopfPair (f \otimes g) = g(\elOmega_1) f(\elOmega_2)
\end{align}
for all $f,g\in H^*$.
        An expression of $\elOmega$ in term of quasi-Hopf data was derived
        in~\cite[Thm.~7.3]{FGR1}. 
        We will not need it explicitly.

\begin{prop}
    \label{prop:SL2Z-action-via-monadic-qHopf}
    Let $\coint \in H^*$ be the right monadic cointegral for $\hmodM$.
    The $\mathbf{S}$- and $\mathbf{T}$-transformations on $\alphaQ Z$ are given by the
    linear maps
    \begin{align}
        \modS_{\alphaQ Z}(\alphaQ z) 
        &= \langle 
            \coint \mid 
            \elOmega_1 z 
        \rangle
        \ \elOmega_2
        \notag \\
        \modT_{\alphaQ Z}(\alphaQ z) 
        &= \elRibbon\inv \alphaQ z
    \end{align}
    for $z \in Z$.
\end{prop}

\begin{proof}
    The action of $\modT$ is immediate from \cite[Sec.\,8]{FGR1}.
For the action of $\modS$ we use
    \eqref{eq:SalphaZ-equals-Shat} and compute
    \begin{align}
        \hat{\modS}(\alphaQ z)
        &\oversetEq[\text{(i)}]
        \langle 
            \coint \mid 
            S(\coassQ[1]_1) \elOmega_1 \coassQ[1]_2 
            \dashuline{ S({\coassQ[1]_3}\sweedler{1} }
            \pL_1)
            \dashuline{\alphaQ} z 
            \dashuline{ {\coassQ[1]_3}\sweedler{2} }
            \pL_2
        \rangle
        \ \elOmega_2
        \notag \\
        &\oversetEq[\text{(ii)}]
        ~ ~
        \langle 
            \coint \mid 
            \elOmega_1 S(\pL_1) \alphaQ \pL_2 z
        \rangle
        \ \elOmega_2
        \notag \\
        &\oversetEq[\text{(iii)}]
        \quad 
        \langle 
            \coint \mid 
            \elOmega_1 z 
        \rangle
        \ \elOmega_2 \ ,
    \end{align}
    where in the first step (i) we used the form of $\hat{\modS}$ as given
    in~\cite[(8.15)]{FGR1}, (ii) uses \eqref{eq:antipodeAxioms} for the underlined part
    and that $\coassQ$ is normalised, and (iii) follows from the definition of $\pL$
    in~\eqref{eq:q,\pL} and the zig-zag axiom~\eqref{eq:zigzag-qHopf}.
\end{proof}

We can express the action of $\modS$ on $\alphaQ Z$ using the right
cointegral $\coint^r$ from \eqref{eq:RightCointEq} and \Cref{thm:MainThmSectionStatement}
as
\begin{align}
    \modS_{\alphaQ Z}(\alphaQ z) 
    &= \langle 
        \coint^r \mid 
        S(\betaQ) \elOmega_1 z 
    \rangle
    \ \elOmega_2\ .
\end{align}

\begin{rmrk}
    One can also show that
    \begin{align}
        \modS_{\alphaQ Z}(\alphaQ z) 
        = \elOmega_1 \langle \coint \mid \elOmega_2 z \rangle \ ,
    \end{align}
    where $\coint$ is the right monadic cointegral.
    To see this, one checks $\omega \circ (\ribTwist_{\coend} \otimes \id) = \omega \circ
    \braiding_{\coend,\coend}\inv$ using $\ribTwist_{\coend} = (S_{\coend})^2$,
    see~\cite[Lem.\,5.2.4]{KerlerLyubashenko}.
    This then readily implies
    \begin{align}
        \hopfPair \circ (f\otimes \id) 
        = \hopfPair \circ (\id \otimes f)
    \end{align}
    for any $f \in \cat(\tensUnit, \coend)$.
    The claim follows since $\langle \coint \mid \placeholder z \rangle \in
    \cat(\tensUnit, \coend)$ for $z$ central.
\end{rmrk}

\medskip

\appendix

\section{Proofs for \texorpdfstring{\Cref{SEC:MainThm}}{Section}}
\label{sec:proofs-sec-4}

\subsection{Proof of \texorpdfstring{\Cref{prop:CatCoaction_is_coaction}}{Proposition}}
\label{app:ProofProp}

Before giving the proof we need to show some intermediate
results.

Using the explicit form of the unit and the counit of the adjunction, we can give the
following simple characterization of the components of $R$.
\begin{lemma}
    \label{lem:CoactionR_univ_prop}
    The coactions $R_U$ defined in \eqref{eq:coactionLeftAdjoint} satisfy
    \begin{equation}
        \ipic{-0.5}{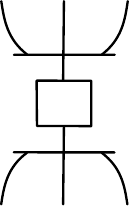}{1.5}
        \put (-59,049) {$Y$}
        \put (-39,049) {$\hopfMonad[2]U$}
        \put (-03,049) {$\dualL{Y}$}
        \put (-02,019) {$\dinatEnd[4]$}
        \put (-37,-04) {$R_U$}
        \put (-03,-24) {$\dinatCoend[2]$}
        \put (-63,-57) {$\dualL{X}$}
        \put (-33,-57) {$U$}
        \put (-07,-57) {$X$}
        \quad \quad \quad \quad
        =
        \quad \quad
        \ipic{-0.5}{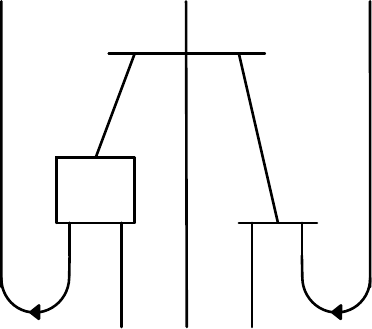}{1.1}
        \put (-121,0056) {$Y$}
        \put (-070,0056) {$\hopfMonad[2]U$}
        \put (-003,0056) {$\dualL{Y}$}
        \put (-031,0033) {$\dinatCoend[2]$}
        \put (-098,-012) {$\gamma_{Y,X}$}
        \put (-026,-016) {$\id$}
        \put (-087,-063) {$\dualL{X}$}
        \put (-064,-063) {$U$}
        \put (-045,-063) {$X$}
        \ .
    \end{equation}
\end{lemma}
\begin{proof}
    By definition,
    \begin{align}
        (\hopfMonad[2][U], R_U) = (\hopfMonad[2][U], 
        \hopfComonad[4][\multCoend[2][U]]
        \circ \unitAdj_{\hopfMonad[2]U})
    \end{align}
    as $\hopfComonad[4]$-comodules.
    Then 
    \begin{equation}
        \ipic{-0.5}{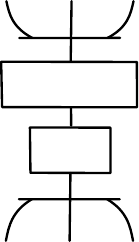}{1.5}
        \put (-60,058) {$Y$}
        \put (-39,058) {$\hopfMonad[2]U$}
        \put (-04,058) {$\dualL{Y}$}
        \put (-06,032) {$\dinatEnd[4]$}
        \put (-57,011) {$\hopfComonad[4][\multCoend[2][U]]$}
        \put (-41,-16) {$\unitAdj_{\hopfMonad[2]U}$}
        \put (-06,-36) {$\dinatCoend[2]$}
        \put (-63,-64) {$\dualL{X}$}
        \put (-34,-64) {$U$}
        \put (-08,-64) {$X$}
        \quad 
        =
        \quad 
        \ipic{-0.5}{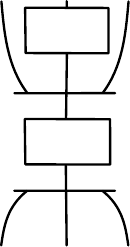}{1.5}
        \put (-60,058) {$Y$}
        \put (-39,058) {$\hopfMonad[2]U$}
        \put (-04,058) {$\dualL{Y}$}
        \put (-42,038) {$\multCoend[2][U]$}
        \put (-04,011) {$\dinatEnd[4]$}
        \put (-39,-11) {$\unitAdj_{\hopfMonad[2]U}$}
        \put (-04,-31) {$\dinatCoend[2]$}
        \put (-61,-65) {$\dualL{X}$}
        \put (-33,-65) {$U$}
        \put (-07,-65) {$X$}
        \quad 
        =
        \quad 
        \ipic{-0.5}{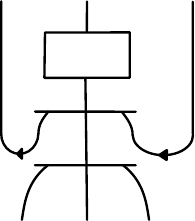}{1.7}
        \put (-98,059) {$Y$}
        \put (-63,059) {$\hopfMonad[2]U$}
        \put (-04,059) {$\dualL{Y}$}
        \put (-68,025) {$\multCoend[2][U]$}
        \put (-26,-01) {$\dinatCoend[2]$}
        \put (-25,-29) {$\dinatCoend[2]$}
        \put (-91,-66) {$\dualL{X}$}
        \put (-57,-66) {$U$}
        \put (-28,-66) {$X$}
    \end{equation}
    together with the definition \eqref{eq:DefMultCoend} of the multiplication of
    $\hopfMonad[2]$ proves the claim.
\end{proof}

\begin{lemma}\label{lem:catCoaction-qHopf}
    Let $V\in \hmodM$, $v\in V$, $h^* \in H^*$, and choose the realization of the central
    Hopf monad $\hopfMonad[2]$ as given in \eqref{eq:hopfMonad-qHopf-actions} and the
    comonad $\hopfComonad[4]$ in \eqref{eq:Z4-realisation}.
    Then
    \begin{samepage}
        \begin{align}\label{eq:categoricalCoaction-qHopf}
            R_V(h^* \otimes v)
            &=
            \langle h^* \mid
            S( {\invCoassQ[1]_2}\sweedler{2} \pL_2 \coassQ[1]_1 )
            \Dt_1 
                \left[e_i {\invCoassQ[1]_3}\sweedler{2} {\coassQ[1]_2}\sweedler{2}
                    \coassQ[2]_2
                \right]\sweedler{1}
            \pR_1
            \rangle 
            \notag \\ 
            &\times 
            \invCoassQ[1]_1 S( {\invCoassQ[1]_2}\sweedler{1} \pL_1 )
            \Dt_2 
                \left[
                    e_i {\invCoassQ[1]_3}\sweedler{2} {\coassQ[1]_2}\sweedler{2}
                    \coassQ[2]_2
                \right]\sweedler{2}
                \pR_2 S( \coassQ[1]_3 \coassQ[2]_3 )
            \notag \\ 
            & \otimes e^i \otimes 
            {\invCoassQ[1]_3}\sweedler{1} {\coassQ[1]_2}\sweedler{1} \coassQ[2]_1 . v ,
        \end{align}
    \end{samepage}
    where $\{e_i\}$ is a basis of $H$ with corresponding dual basis $\{e^i\}$, and
    summation over $i$ is implied.
\end{lemma}

\begin{proof}
    Set 
    \begin{align}
        \Xi &= (\id \otimes \id \otimes \Delta \otimes \id)(\invCoassQ \otimes \oneQ)
        \cdot (\Delta\otimes \Delta \otimes \id)(\coassQ)
        \cdot (\oneQ\otimes \oneQ\otimes \coassQ)
        \notag \\
            &= 
        \invCoassQ[1]_1 {\coassQ[1]_1}\sweedler{1} 
        \otimes 
        \invCoassQ[1]_2 {\coassQ[1]_1}\sweedler{2} 
        \otimes
        {\invCoassQ[1]_3}\sweedler{1} {\coassQ[1]_2}\sweedler{1} \coassQ[2]_1
        \otimes
        {\invCoassQ[1]_3}\sweedler{2} {\coassQ[1]_2}\sweedler{2} \coassQ[2]_2
        \otimes
        \coassQ[1]_3 \coassQ[2]_3
        , \notag \\
        \Theta &= (S\otimes S)(\pL_{21}) \Dt
        , \notag \\
        \Omega &= \Xi_1 \otimes S(\Xi_2) \otimes \Xi_4 \otimes S(\Xi_5) \otimes \Xi_3,
        \label{eq:RV-lemma-aux1}
    \end{align}
    then from \Cref{lem:CoactionR_univ_prop} one computes
    \begin{align}
        \dinatEnd[4][\hopfMonad[2]U][Y]
        \circ R_U
        \circ \dinatCoend[2][U][X]
        &= 
        \quad
        \ipic{-0.5}{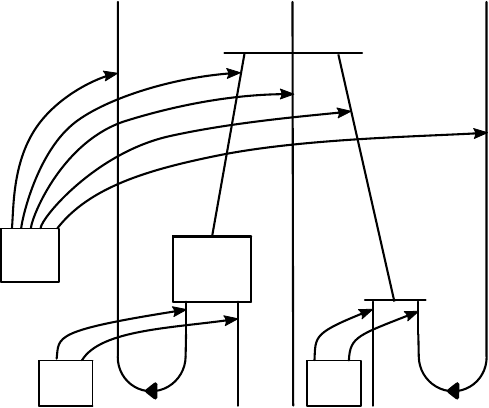}{1.5}
        \put (-164,0092) {$Y$}
        \put (-096,0092) {$\hopfMonad[2]U$}
        \put (-003,0092) {$\dualL{Y}$}
        \put (-052,0063) {$\dinatCoend[2]$}
        \put (-132,-030) {\large $\gamma_{Y,X}$}
        \put (-204,-027) {\large $\Xi$}
        \put (-028,-040) {$\id$}
        \put (-189,-082) {$\pL$}
        \put (-074,-082) {$\pR$}
        \put (-115,-100) {$\dualL{X}$}
        \put (-090,-100) {$U$}
        \put (-056,-100) {$X$}
        \put (0020,0100) {$\boxed{\Vect_{\field}}$}
        \notag \\
        &=
        \ipic{-0.5}{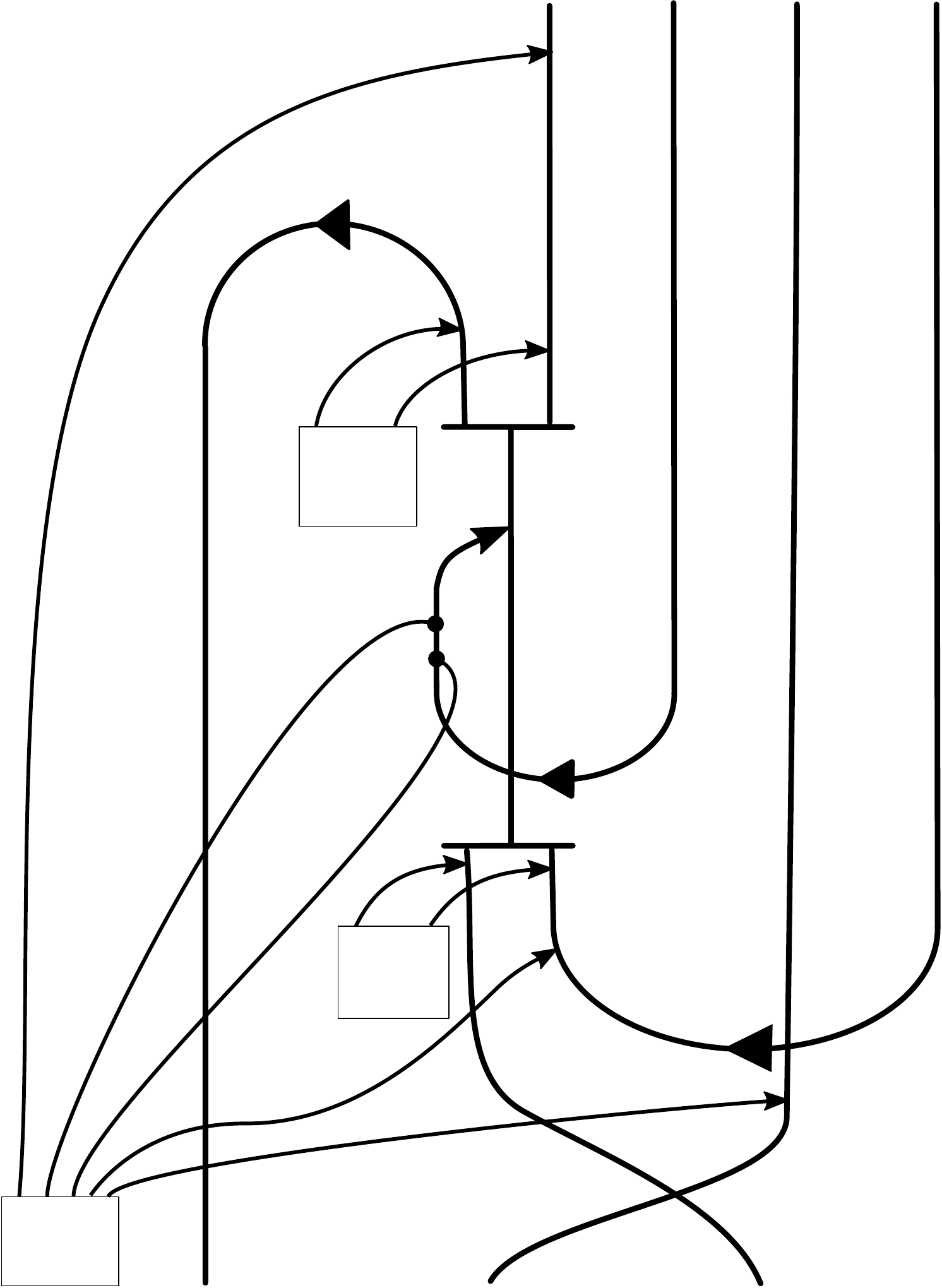}{0.4}
        \put (-075,0120) {$Y$}
        \put (-055,0120) {$H^*$}
        \put (-030,0120) {$U$}
        \put (-004,0120) {$\dualL{Y}$}
        \put (-066,0040) {$\id$}
        \put (-112,0027) {$\Theta$}
        \put (-066,-040) {$\id$}
        \put (-108,-064) {$\pR$}
        \put (-167,-113) {$\Omega$}
        \put (-142,-129) {$\dualL{X}$}
        \put (-087,-129) {$U$}
        \put (-039,-129) {$X$}
        \put (0020,0120) {$\boxed{\Vect_{\field}}$}
        \ .
    \end{align}
    Recall that the box with $\Vect_{\field}$ means that these pictures are to be
    understood as linear maps.
    Specializing $X$ and $Y$ to the regular left module $H$, we note that
    $\dinatEnd[4][V][H]$ has a left inverse
    \begin{align}
        (H\otimes V) \otimes \dualL{H} \to H\otimes V,
        \quad
        h\otimes v \otimes f \mapsto f(\oneQ) h \otimes v\ ,
    \end{align}
    while $\dinatCoend[4][V][H]$ has a right inverse
    \begin{align}
        H^*\otimes V \to \dualL{H} \otimes (V \otimes H)
        \quad
        f\otimes v \mapsto f \otimes v \otimes \oneQ \ .
    \end{align}

    Applying the inverses we obtain the explicit form of $R_U$,
    \begin{align}
        \ipic{-0.5}{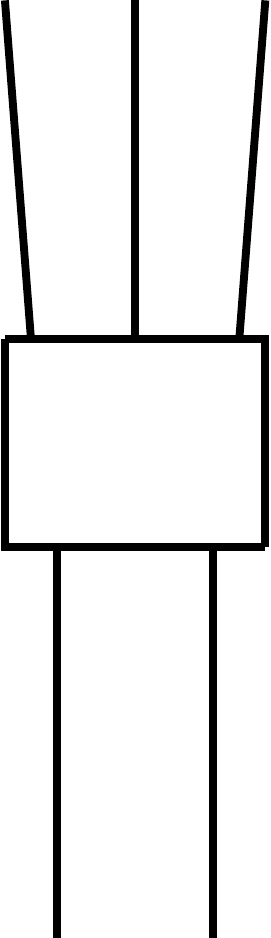}{0.5}
        \put (-44,072) {$H$}
        \put (-24,072) {$H^*$}
        \put (-03,072) {$U$}
        \put (-27,000) {$R_U$}
        \put (-37,-79) {$H^*$}
        \put (-13,-79) {$U$}
        \quad 
        &=
        \quad 
        \ipic{-0.5}{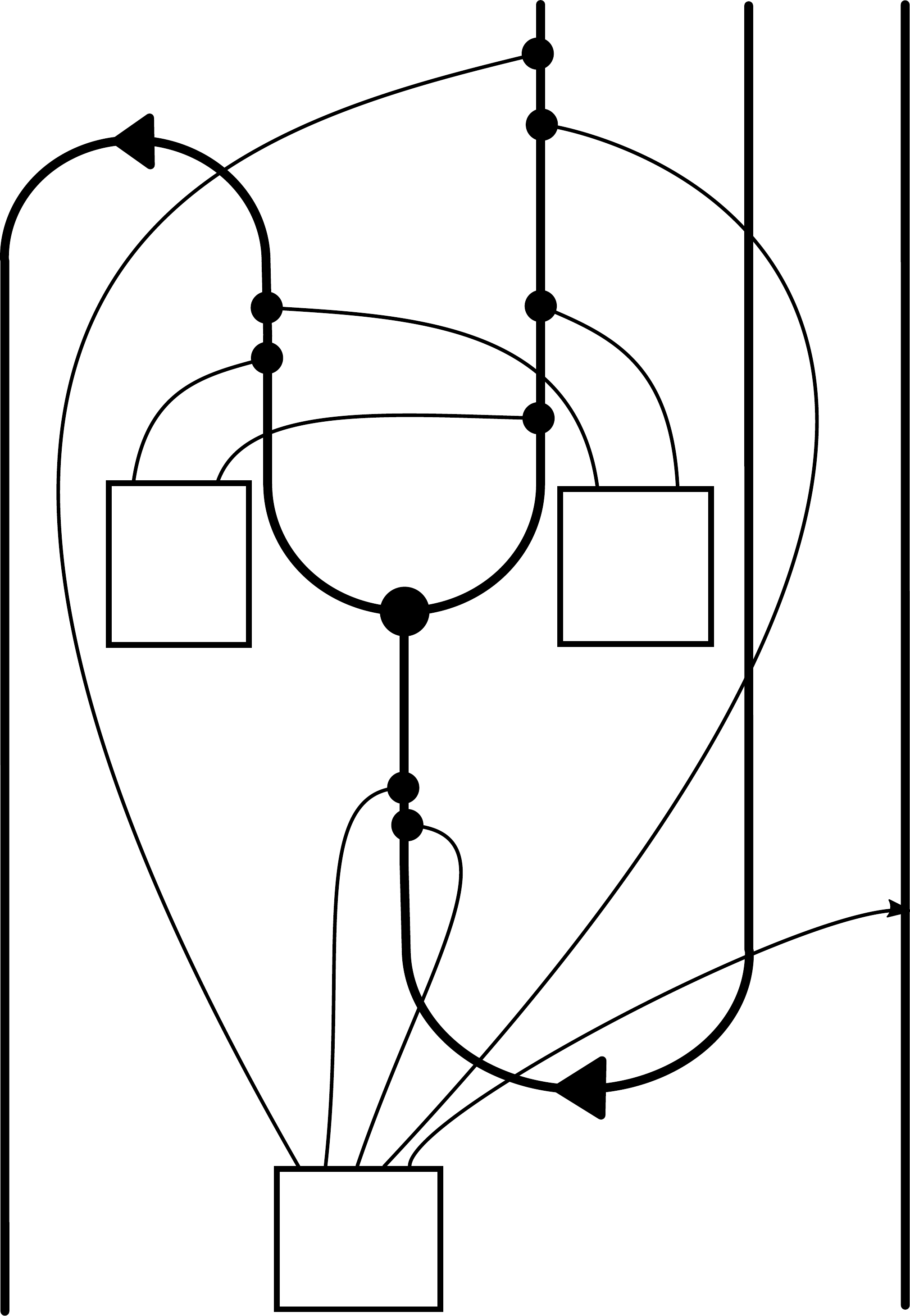}{0.18}
        \put (-047,079) {$H$}
        \put (-024,079) {$H^*$}
        \put (-004,079) {$U$}
        \put (-090,006) {\large $\Theta$}
        \put (-038,007) {$\pR$}
        \put (-069,-72) {\large $\Omega$}
        \put (-111,-87) {$H^*$}
        \put (-005,-87) {$U$}
        \put (0020,080) {$\boxed{\Vect_{\field}}$}
    \end{align}
From this we can read off that $R_U$ is the linear map given by
(to better see where we apply the changes from one line to the next we sometimes underline
the relevant part)
    \begin{align}
        R_U(h^*\otimes u)
        &= 
        \langle 
            h^* \otimes \id \mid 
            (\oneQ \otimes \Omega_1) \Theta \Delta( \Omega_2 e_i \Omega_3 ) \pR 
            (\oneQ \otimes \Omega_4) 
        \rangle
        \otimes e^i \otimes \Omega_5.v
        \notag \\
        &= 
        \langle 
            h^* \otimes \id \mid 
            (\oneQ \otimes \Xi_1) \Theta \Delta( S(\Xi_2) e_i \Xi_4 ) \pR 
            (\oneQ \otimes S(\Xi_5))
        \rangle
        \otimes e^i \otimes \Xi_3.v
        \notag \\
        &= 
        \langle 
            h^* \mid 
            S(\pL_{2}) \Dt_1 \left[ S(\Xi_2) e_i \Xi_4 \right]\sweedler{1}
            \pR_1 
        \rangle
        \notag \\
        &~\times
        \Xi_1 S(\pL_1) \Dt_2 \left[ S(\Xi_2) e_i \Xi_4 \right]\sweedler{2} \pR_2 S(\Xi_5)
        \otimes e^i \otimes \Xi_3.v
        \notag \\
        &\oversetEq[\eqref{eq:RV-lemma-aux1}]~ 
        \langle 
            h^* \mid 
            S(\pL_{2}) \dashuline{\Dt_1}
            \big[ 
                \dashuline{ S(\invCoassQ[1]_2 {\coassQ[1]_1}\sweedler{2}) }
                e_i
                {\invCoassQ[1]_3}\sweedler{2} {\coassQ[1]_2}\sweedler{2} \coassQ[2]_2 
            \big]\sweedler{1}
            \pR_1 
        \rangle
        \notag \\
        &~\times
        \invCoassQ[1]_1 {\coassQ[1]_1}\sweedler{1} S(\pL_1) \dashuline{\Dt_2}
        \big[ 
            \dashuline{ S(\invCoassQ[1]_2 {\coassQ[1]_1}\sweedler{2}) }
            e_i 
            {\invCoassQ[1]_3}\sweedler{2} {\coassQ[1]_2}\sweedler{2} \coassQ[2]_2 
        \big]\sweedler{2}
        \pR_2 S(\coassQ[1]_3 \coassQ[2]_3)
        \notag \\
        &~\otimes e^i \otimes {\invCoassQ[1]_3}\sweedler{1} {\coassQ[1]_2}\sweedler{1}
        \coassQ[2]_1.v
        \notag \\
        &\oversetEq[\eqref{eq:DrinfeldTwistProperty}]~
        \langle 
            h^* \mid 
            S({\invCoassQ[1]_2}\sweedler{2} 
                \dashuline{ {\coassQ[1]_1}\sweedler{2,2}\pL_{2} } 
            ) \Dt_1 
            \left[ 
                e_i
                {\invCoassQ[1]_3}\sweedler{2} {\coassQ[1]_2}\sweedler{2} \coassQ[2]_2 
            \right]\sweedler{1}
            \pR_1 
        \rangle
        \notag \\
        &~\times
        \invCoassQ[1]_1 {\coassQ[1]_1}\sweedler{1} 
        S({\invCoassQ[1]_2}\sweedler{1} 
            \dashuline{ {\coassQ[1]_1}\sweedler{2,1}\pL_1 }
        ) \Dt_2 
        \left[ 
            e_i 
            {\invCoassQ[1]_3}\sweedler{2} {\coassQ[1]_2}\sweedler{2} \coassQ[2]_2 
        \right]\sweedler{2}
        \pR_2 S(\coassQ[1]_3 \coassQ[2]_3)
        \notag \\
        &~\otimes e^i \otimes {\invCoassQ[1]_3}\sweedler{1} {\coassQ[1]_2}\sweedler{1}
        \coassQ[2]_1.v
        \notag \\
        &\oversetEq[\eqref{eq:identities_qp_iterated_coproduct}]~
        \langle 
            h^* \mid 
            S({\invCoassQ[1]_2}\sweedler{2} \pL_{2} \coassQ[1]_1) \Dt_1 
            \left[ 
                e_i
                {\invCoassQ[1]_3}\sweedler{2} {\coassQ[1]_2}\sweedler{2} \coassQ[2]_2 
            \right]\sweedler{1}
            \pR_1 
        \rangle
        \notag \\
        &~\times
        \invCoassQ[1]_1 S({\invCoassQ[1]_2}\sweedler{1} \pL_1) \Dt_2 
        \left[ 
            e_i 
            {\invCoassQ[1]_3}\sweedler{2} {\coassQ[1]_2}\sweedler{2} \coassQ[2]_2 
        \right]\sweedler{2}
        \pR_2 S(\coassQ[1]_3 \coassQ[2]_3)
        \notag \\
        &~\otimes e^i \otimes {\invCoassQ[1]_3}\sweedler{1} {\coassQ[1]_2}\sweedler{1}
        \coassQ[2]_1.v
    \end{align}
    for $h^*\in H^*$, $u\in U$.
\end{proof}

\label{proof:CatCoaction_is_coaction}
\begin{proof}[Proof of \texorpdfstring{\Cref{prop:CatCoaction_is_coaction}}{Proposition}]
We will need the identity
\begin{align}\label{eq:....}
    \qL_2 [\Dt\inv_2]\sweedler{2} \otimes S(\Dt\inv_1) \qL_1
    [\Dt\inv_2]\sweedler{1}
    = (S\otimes S)(\pR) \Dt_{21} \ , 
\end{align}
which can be seen as follows:
\begin{align}
    \qL_2 &[\Dt\inv_2]\sweedler{2} \otimes S(\Dt\inv_1) \qL_1
    [\Dt\inv_2]\sweedler{1} 
    \notag \\
    \oversetEq[\eqref{eq:q,p^L}]~&\quad
    \coassQ[1]_3 [\Dt\inv_2]\sweedler{2} 
    \otimes S(\coassQ[1]_1 \Dt\inv_1) \alphaQ \coassQ[1]_2 [\Dt\inv_2]\sweedler{1} 
    \notag \\
    \oversetEq[\eqref{eq:DrinfeldTwist-twistedCoassoc}]~&\quad
    \Dt\inv_2 S(\coassQ[1]_1) \Dt_2 
    \otimes 
    \dashuline{ S([\Dt\inv_1]\sweedler{1} } \widetilde{\Dt}\inv_1 S(\coassQ[1]_3))
    \dashuline{ \alphaQ [\Dt\inv_1]\sweedler{2} }
    \widetilde{\Dt}\inv_2 S(\coassQ[1]_2) \Dt_1
    \notag \\
    \oversetEq[\eqref{eq:antipodeAxioms}]~&\quad
    \dotuline{ \counit(\Dt\inv_1) \Dt\inv_2 } S(\coassQ[1]_1) \Dt_2 
    \otimes \dashuline{ S(\widetilde{\Dt}\inv_1 } S(\coassQ[1]_3))
    \dashuline{ \alphaQ \widetilde{\Dt}\inv_2 } S(\coassQ[1]_2) \Dt_1
    \notag \\
    \oversetEq[(*)]~&\quad
    S(\coassQ[1]_1) \Dt_2 
    \otimes S^2(\coassQ[1]_3) S(\betaQ) S(\coassQ[1]_2) \Dt_1
    \notag \\
    \oversetEq[\eqref{eq:q,p^L}]~&\quad
    (S\otimes S)(\pR) \cdot \Dt_{21} \ .
\end{align}
In the step labelled $(*)$ one uses \eqref{eq:Dt-alphBet-interplay} (dashed underline) and
that the counit applied to any leg of the inverse Drinfeld twist yields $\oneQ$ (dotted
underline).
The identity \eqref{eq:....} immediately implies
\begin{align}\label{eq:useThisInProof}
	\pR = 
	S\inv( \qL_2 {\Dt\inv_2}\sweedler{2} \widetilde{\Dt}\inv_2 )
	\otimes S\inv( \qL_1 {\Dt\inv_2}\sweedler{1} \widetilde{\Dt}\inv_1 )
	\Dt\inv_1.
\end{align}

For the proof of the proposition, let now $h^* \in H^*$. 
Then 
\begin{align}
    \big( \varphi_{\dualL{H}} \circ &\Adtwist(\rho) \big) (h^*)
    \notag \\
    =~&
    \modulus\inv( {\invCoassQ[1]_3}\sweedler{1} {\coassQ[1]_2}\sweedler{1}
    \coassQ[2]_1 )
    \langle
        h^* \mid
        S(\pL_2) \Dt_1 
        \left[ 
            e_i
        \right]\sweedler{1}
        S\inv( \qL_2 \Dt\inv_2 ) 
    \rangle 
    \notag \\
    &\times
    \invCoassQ[1]_1 {\coassQ[1]_1}\sweedler{1} S(\pL_1) \Dt_2 
    \left[ 
        e_i
    \right]\sweedler{2}
    S\inv( \qL_1 \Dt\inv_1 ) \widetilde{\Dt}\inv_1 S(\coassQ[1]_3
    \coassQ[2]_3)
    \notag \\
    &\otimes
    \dashuline{\invCoassQ[1]_2 {\coassQ[1]_1}\sweedler{2}}
    . e^i .
    \dashuline{\widetilde{\Dt}\inv_2 S( {\invCoassQ[1]_3}\sweedler{2}
    {\coassQ[1]_2}\sweedler{2} \coassQ[2]_2 )}
    \notag \\[0.2em]
    \oversetEq[\eqref{eq:hookActionDefined-H}]~&
    \modulus\inv( {\invCoassQ[1]_3}\sweedler{1} {\coassQ[1]_2}\sweedler{1}
    \coassQ[2]_1 )
    \langle
        h^* \mid
        S(\pL_2) \Dt_1 
        \left[ 
            e_i
        \right]\sweedler{1}
        S\inv( \qL_2 \dashuline{\Dt\inv_2} ) 
    \rangle 
    \notag \\
    &\times
    \invCoassQ[1]_1 {\coassQ[1]_1}\sweedler{1} S(\pL_1) \Dt_2 
    \left[ 
        e_i
    \right]\sweedler{2}
    S\inv( \qL_1 \dashuline{\Dt\inv_1} ) \widetilde{\Dt}\inv_1 S(\coassQ[1]_3
    \coassQ[2]_3)
    \notag \\
    &\otimes
    \dashuline{
        {\invCoassQ[1]_3}\sweedler{2}
        {\coassQ[1]_2}\sweedler{2} \coassQ[2]_2 S\inv(\widetilde{\Dt}\inv_2)
    }
    \rightharpoonup e^i \leftharpoonup
    S(\invCoassQ[1]_2 {\coassQ[1]_1}\sweedler{2})
    \notag \\[0.2em]
    \oversetEq[\eqref{eq:DrinfeldTwistProperty}]~&\quad
    \modulus\inv( {\invCoassQ[1]_3}\sweedler{1} {\coassQ[1]_2}\sweedler{1}
    \coassQ[2]_1 )
    \langle
        h^* \mid
        S(\pL_2) \Dt_1 
        \left[ 
            e_i
            {\invCoassQ[1]_3}\sweedler{2} {\coassQ[1]_2}\sweedler{2} \coassQ[2]_2
        \right]\sweedler{1}
        \dashuline{
            S\inv( \qL_2 {\Dt\inv_2}\sweedler{2} \widetilde{\Dt}\inv_2 ) 
        }
    \rangle 
    \notag \\
    &\times
    \invCoassQ[1]_1 {\coassQ[1]_1}\sweedler{1} S(\pL_1) \Dt_2 
    \left[ 
        e_i
        {\invCoassQ[1]_3}\sweedler{2} {\coassQ[1]_2}\sweedler{2} \coassQ[2]_2
    \right]\sweedler{2}
    \dashuline{
        S\inv( \qL_1 {\Dt\inv_2}\sweedler{1} \widetilde{\Dt}\inv_1 )
        \widetilde{\Dt}\inv_1 
    }
    S(\coassQ[1]_3 \coassQ[2]_3)
    \notag \\
    &\otimes
    e^i \leftharpoonup
    S(\invCoassQ[1]_2 {\coassQ[1]_1}\sweedler{2})
    \notag \\[0.2em]
    \oversetEq[\eqref{eq:useThisInProof}]~&\quad
    \modulus\inv( {\invCoassQ[1]_3}\sweedler{1} {\coassQ[1]_2}\sweedler{1}
    \coassQ[2]_1 )
    \langle
        h^* \mid
        S(\pL_2) \dashuline{\Dt_1}
        \left[ 
            e_i
            {\invCoassQ[1]_3}\sweedler{2} {\coassQ[1]_2}\sweedler{2} \coassQ[2]_2
        \right]\sweedler{1}
        \pR_1
    \rangle 
    \notag \\
    &\times
    \invCoassQ[1]_1 {\coassQ[1]_1}\sweedler{1} S(\pL_1) \dashuline{\Dt_2}
    \left[ 
        e_i
        {\invCoassQ[1]_3}\sweedler{2} {\coassQ[1]_2}\sweedler{2} \coassQ[2]_2
    \right]\sweedler{2}
    \pR_2 S(\coassQ[1]_3 \coassQ[2]_3)
    \notag \\
    &\otimes
    e^i \leftharpoonup
    \dashuline{
        S(\invCoassQ[1]_2 {\coassQ[1]_1}\sweedler{2})
    }
    \notag \\[0.2em]
    \oversetEq[\eqref{eq:DrinfeldTwistProperty}]~&\quad
    \modulus\inv( {\invCoassQ[1]_3}\sweedler{1} {\coassQ[1]_2}\sweedler{1}
    \coassQ[2]_1 )
    \langle
        h^* \mid
        S({\invCoassQ[1]_2}\sweedler{2} 
        \dashuline{
            {\coassQ[1]_1}\sweedler{2,2} \pL_2
        }
        ) \Dt_1 
        \left[ 
            e_i
            {\invCoassQ[1]_3}\sweedler{2} {\coassQ[1]_2}\sweedler{2} \coassQ[2]_2
        \right]\sweedler{1}
        \pR_1
    \rangle 
    \notag \\
    &\times
    \invCoassQ[1]_1 \dashuline{{\coassQ[1]_1}\sweedler{1}}
    S({\invCoassQ[1]_2}\sweedler{1} 
    \dashuline{
        {\coassQ[1]_1}\sweedler{2,1}\pL_1
    }
    ) \Dt_2 
    \left[ 
        e_i
        {\invCoassQ[1]_3}\sweedler{2} {\coassQ[1]_2}\sweedler{2} \coassQ[2]_2
    \right]\sweedler{2}
    \pR_2 S(\coassQ[1]_3 \coassQ[2]_3) 
    \otimes e^i 
    \notag \\[0.2em]
    \oversetEq[\eqref{eq:identities_qp_iterated_coproduct}]~&\quad
    \modulus\inv( {\invCoassQ[1]_3}\sweedler{1} {\coassQ[1]_2}\sweedler{1}
    \coassQ[2]_1 )
    \langle
        h^* \mid
        S({\invCoassQ[1]_2}\sweedler{2} \pL_2 \coassQ[1]_1) \Dt_1 
        \left[ 
            e_i
            {\invCoassQ[1]_3}\sweedler{2} {\coassQ[1]_2}\sweedler{2} \coassQ[2]_2
        \right]\sweedler{1}
        \pR_1
    \rangle 
    \notag \\ 
    &\times
    \invCoassQ[1]_1 S({\invCoassQ[1]_2}\sweedler{1} \pL_1) \Dt_2 
    \left[ 
        e_i
        {\invCoassQ[1]_3}\sweedler{2} {\coassQ[1]_2}\sweedler{2} \coassQ[2]_2
    \right]\sweedler{2}
    \pR_2 S(\coassQ[1]_3 \coassQ[2]_3) 
    \otimes e^i .
\end{align}
From \Cref{lem:catCoaction-qHopf} we obtain
\begin{align}
    R_{\dualL{\modulus}}(h^*)
    &=
    \modulus\inv( {\invCoassQ[1]_3}\sweedler{1} {\coassQ[1]_2}\sweedler{1}
    \coassQ[2]_1 ) ~
    \langle h^* \mid
    S( {\invCoassQ[1]_2}\sweedler{2} \pL_2 \coassQ[1]_1 )
    \Dt_1 
        \left[
            e_i {\invCoassQ[1]_3}\sweedler{2} {\coassQ[1]_2}\sweedler{2}
            \coassQ[2]_2
        \right]\sweedler{1}
    \pR_1
    \rangle 
    \notag \\ 
    &\times 
    \invCoassQ[1]_1 S( {\invCoassQ[1]_2}\sweedler{1} \pL_1 )
    \Dt_2 
        \left[
            e_i {\invCoassQ[1]_3}\sweedler{2} {\coassQ[1]_2}\sweedler{2}
            \coassQ[2]_2
        \right]\sweedler{2}
        \pR_2 S( \coassQ[1]_3 \coassQ[2]_3 )
    \otimes e^i\ ,
\end{align}
so that
\begin{align}
    R_{\dualL{\modulus}} = \varphi_{\dualL{H}} \circ \Adtwist(\rho)
\end{align}
indeed holds, finishing the proof.
\end{proof}

\medskip

\subsection{Proof of \texorpdfstring{\Cref{thm:MainThmSectionStatement}}{4.1}
(1)}\label{app:Proof1}

The first step in the proof of the \Cref{thm:MainThmSectionStatement} is to map cointegrals to a Hom-space containing monadic cointegrals.
To do this, we define the space
\begin{align}\label{eq:Def_gammaSSym}
    \gammaSSym & = 
    \{
        f\in H^* \mid 
        f \leftharpoonup S(a) = S\inv(\modulus \rightharpoonup a) \rightharpoonup f 
        \quad \forall a\in H
    \} \notag \\
    &=
    \{
        f\in \dualL{H}\in \hmodM[H][H] \mid 
        a.f = f.(\modulus \rightharpoonup a)
        \quad \forall a\in H
    \}, 
\end{align}
where the dot denotes the action on the left dual of the regular bimodule in 
    $\hmodM[H][H]$,
the category of $H\otimes H\op$-modules as introduced in \Cref{Sec:BC-CointegralsPartiallyExplained}.

Note that right cointegrals are automatically in $\gammaSSym$ by
\eqref{eq:symmetryPropertiesCoint}:
\begin{align}
    \spaceRightCoint \subset \gammaSSymR\ .
\end{align}
By \eqref{eq:hopfMonad-qHopf-actions}, we have
\begin{align}
&\cat(\tensUnit,\hopfMonad[2][\dualL{\modulus}])  \notag\\
&=
    \{
        f \in H^* \mid 
        \counit(h) f(a) = f(S(h\sweedler{1}) a (h\sweedler{2} \leftharpoonup \modulus\inv))
        \quad \forall h,a\in H
    \}
    \notag \\
    &=\{
        f \in \dualL{H} \in \hmodM[H][H] \mid 
        \counit(h) f = h\sweedler{1}.f.S(h\sweedler{2} \leftharpoonup \modulus\inv)
        \quad \forall h\in H
    \},
\label{eq:app-linear-eqn-Hom-to-A2gamma}    
\end{align}
where in the second line we again let the dot denote the action on the left dual of the
regular bimodule.

We then have the following proposition.

\begin{prop}\label{prop:X2=C2}
    Let $\elX = (\id \otimes \modulus)(\Dt\inv)$.
    Then the map
    \begin{align}
        \isoXC_2:\gammaSSym &\to \cat \big( \tensUnit,
        \hopfMonad[2](\dualL{\modulus}) \big),
        \quad f\mapsto \betaQ.f.\elX,
    \end{align}
    is a linear isomorphism.
\end{prop}
\begin{proof}
    \label{proof:X2=C2}
    Abbreviate $\cat_2 \defined \cat \big( \tensUnit, \hopfMonad[2](\dualL{\modulus})
    \big)$, and let us check that $\isoXC_2(\gammaSSym) \subset \cat_2$.
    To this end, observe that the defining equation for $f\in H^*$ to be in
    $\gammaSSym$ may be rewritten as
    \begin{align}\label{eq:AlternateX2Condition}
        S(a).f = f.\elX S(a \leftharpoonup \modulus\inv) \elX\inv
    \end{align}
    by using the definition of the Drinfeld twist.
    Then we compute
    \begin{alignat}{2}
        h\sweedler{1}.\isoXC_2(f).S(h\sweedler{2} \leftharpoonup \modulus\inv)
        &~\oversetEq &&\quad
        h\sweedler{1}\betaQ.f.\elX S(h\sweedler{2} \leftharpoonup \modulus\inv) 
        \notag \\ &~\oversetEq[\eqref{eq:AlternateX2Condition}] &&\quad
        h\sweedler{1}\betaQ S(h\sweedler{2}).f.\elX  
        \notag \\ &~\oversetEq[\eqref{eq:antipodeAxioms}] &&\quad
        \counit(h) \betaQ.f.\elX 
        \notag \\ &~\oversetEq &&\quad
        \counit(h) \isoXC_2(f).
    \end{alignat}

    \noindent Next, we claim that the assignment
    \begin{align}
        \label{eq:inverseOfMainThm}
        \mathbb{B}_2: f \mapsto 
        \qL_1 . f . S(\qL_2 \leftharpoonup \modulus\inv) \elX\inv
    \end{align}
    is the two-sided inverse of $\isoXC_2$. 
    First of all, $\mathbb{B}_2(\cat_2) \subset \gammaSSym$.
    Indeed,
    \begin{samepage}
    \begin{alignat}{2}
        \mathbb{B}_2(f).\elX S(a \leftharpoonup \modulus\inv) \elX\inv
        &~\oversetEq &&\quad
        \qL_1 . f . S(\qL_2 \leftharpoonup \modulus\inv) \elX\inv \elX S(a
        \leftharpoonup \modulus\inv) \elX\inv
        \notag \\ &~\oversetEq  &&\quad
        \qL_1 . f . S((a\qL_2) \leftharpoonup \modulus\inv) \elX\inv
        \notag \\ &~\oversetEq[\eqref{eq:identities_qp_iterated_coproduct}]  &&\quad
        S(a\sweedler{1}) \qL_1 a\sweedler{2,1} . f. S((\qL_2 a\sweedler{2,2})
        \leftharpoonup \modulus\inv) \elX\inv
        \notag \\ &~\oversetEq[(\star)]  &&\quad
        S(a) \qL_1 . f. S((\qL_2) \leftharpoonup \modulus\inv) \elX\inv
        \notag \\ &~\oversetEq  &&\quad
        S(a). \mathbb{B}_2(f).
    \end{alignat}
    \end{samepage}
    Here $(\star)$ uses that $f\in \cat_2$.

    It is not hard to see that $\mathbb{B}_2$ is a left inverse of $\isoXC_2$:
    \begin{align}
        \mathbb{B}_2\isoXC_2(f) 
        &= \mathbb{B}_2(\betaQ.f.\elX) 
        = \qL_1 \betaQ . f . \elX S( \qL_2 \leftharpoonup \modulus\inv) \elX \inv
        = \qL_1 \betaQ S(\qL_2) . f
        = f.
    \end{align}
    To see that $\mathbb{A}_2\mathbb{B}_2=\id$ we need the fact that $\betaQ=(S\otimes
    \counit)(\pL)$, and the $\pL,\qL$-relation in \eqref{eq:identities qpL}.
    We compute
    \begin{alignat}{2}
        \isoXC_2\mathbb{B}_2(f) 
        &~\oversetEq &&\quad
        \isoXC_2(
            \qL_1 . f . S( \qL_2 \leftharpoonup \modulus\inv) \elX \inv 
            )
        \notag \\ &~\oversetEq &&\quad
        \betaQ \qL_1 . f . S( \qL_2 \leftharpoonup \modulus\inv) 
        \notag \\ &~\oversetEq &&\quad
        S(\pL_1) \qL_1 . \counit(\pL_2) f . S( \qL_2 \leftharpoonup \modulus\inv) 
        \notag \\ &~\oversetEq[(\star)] &&\quad
        S(\pL_1) \qL_1 {\pL_2}\sweedler{1}. f . S( (\qL_2 {\pL_2}\sweedler{2})
        \leftharpoonup \modulus\inv) 
        \notag \\ &~\oversetEq[\eqref{eq:identities qpL}] &&\quad
        1.f.S(1\leftharpoonup \modulus \inv)
        \notag \\ &~\oversetEq &&\quad
        f\ ,
    \end{alignat}
    using that $f\in \cat_2$ in $(\star)$.
\end{proof}

We will need the following technical lemma.

\begin{lemma}\label{lem:weird_identity}
    Let $f\in \gammaSSym$.
    Then
\begin{align}
    \modulus(\invCoassQ[1]_3)~
    \varphi_{\dualL{H}} 
    \big(
        \betaQ\sweedler{1} \invCoassQ[1]_1 \elX\sweedler{1} 
        \otimes
        \betaQ\sweedler{2} . f . \invCoassQ[1]_2 \elX\sweedler{2} 
    \big)
    =
    \betaQ \otimes_{\field} \betaQ. f . \elX 
\end{align}
\end{lemma}

\begin{proof}
    \begin{align}
        \modulus(\invCoassQ[1]_3)~
        &\varphi_{\dualL{H}} 
        \big(
            \betaQ\sweedler{1} \invCoassQ[1]_1 \elX\sweedler{1} 
            \otimes
            \betaQ\sweedler{2} . f . \invCoassQ[1]_2 \elX\sweedler{2} 
        \big)
        \notag \\ 
        \oversetEq[\eqref{eq:coproduct_beta}]\quad &
        \modulus(\invCoassQ[1]_3)~
        \varphi_{\dualL{H}} 
        \big(
            \elDelta_1 \Dt_1 \invCoassQ[1]_1 \elX\sweedler{1} 
            \otimes
            \elDelta_2 \Dt_2  . f . \invCoassQ[1]_2 \elX\sweedler{2} 
        \big)
        \notag \\[1em]
        =\quad&
        \modulus\inv({\invCoassQ[2]_3}\sweedler{1}
        {\coassQ[1]_2}\sweedler{1}\coassQ[2]_1) \modulus(\invCoassQ[1]_3)~
            \invCoassQ[2]_1 {\coassQ[1]_1}\sweedler{1}
            \elDelta_1 \Dt_1 \invCoassQ[1]_1 
            \dashuline{\elX\sweedler{1}}
            \Dt\inv_1 S(\coassQ[1]_3 \coassQ[2]_3) 
        \notag \\ 
        &\otimes
            \invCoassQ[2]_2 {\coassQ[1]_1}\sweedler{2}
            \elDelta_2 \Dt_2  . f . \invCoassQ[1]_2 
            \dashuline{\elX\sweedler{2}}
            \Dt\inv_2 S({\invCoassQ[2]_3}\sweedler{2}
            {\coassQ[1]_2}\sweedler{2}\coassQ[2]_2)
        \notag \\[1em]
        =\quad&
        \modulus\inv({\invCoassQ[2]_3}\sweedler{1}
        {\coassQ[1]_2}\sweedler{1}\coassQ[2]_1) 
        \modulus(\invCoassQ[1]_3 \MakeUppercase{\Dt}\inv_2 )~
            \invCoassQ[2]_1 {\coassQ[1]_1}\sweedler{1}
            \elDelta_1 
            \dashuline{\Dt_1}
            \invCoassQ[1]_1 {\MakeUppercase{\Dt}\inv_1}\sweedler{1} 
            \Dt\inv_1 S(\coassQ[1]_3 \coassQ[2]_3) 
        \notag \\ 
        &\otimes
            \invCoassQ[2]_2 {\coassQ[1]_1}\sweedler{2}
            \elDelta_2 
            \dashuline{\Dt_2}
            . f . \invCoassQ[1]_2
            {\MakeUppercase{\Dt}\inv_1}\sweedler{2} 
            \Dt\inv_2 S({\invCoassQ[2]_3}\sweedler{2}
            {\coassQ[1]_2}\sweedler{2}\coassQ[2]_2)
        \notag \\[1em]
        \oversetEq[(\star)]\quad&
        \modulus\inv({\invCoassQ[2]_3}\sweedler{1}
        {\coassQ[1]_2}\sweedler{1}\coassQ[2]_1) 
        \modulus(
        \dashuline{
            {\Dt_2}\sweedler{2} \invCoassQ[1]_3 \MakeUppercase{\Dt}\inv_2 
        }
        )~
            \invCoassQ[2]_1 {\coassQ[1]_1}\sweedler{1}
            \elDelta_1 
            \dashuline{
                \Dt_1 \invCoassQ[1]_1 {\MakeUppercase{\Dt}\inv_1}\sweedler{1} 
                \Dt\inv_1 
            }
            S(\coassQ[1]_3 \coassQ[2]_3) 
        \notag \\ 
        &\otimes
            \invCoassQ[2]_2 {\coassQ[1]_1}\sweedler{2}
            \elDelta_2 . f . 
            \dashuline{
                {\Dt_2}\sweedler{1} \invCoassQ[1]_2
                {\MakeUppercase{\Dt}\inv_1}\sweedler{2} 
                \Dt\inv_2 
            }
            S({\invCoassQ[2]_3}\sweedler{2}
            {\coassQ[1]_2}\sweedler{2}\coassQ[2]_2)
        \notag \\[1em]
        \oversetEq[\eqref{eq:DrinfeldTwist-twistedCoassoc}]\quad&
        \modulus\inv({\invCoassQ[2]_3}\sweedler{1}
        {\coassQ[1]_2}\sweedler{1}
        \dotuline{\coassQ[2]_1} )
        \modulus( 
            \dashuline{
                \Dt\inv_2 
            }
            \dotuline{
                S(\invCoassQ[1]_1) 
            }
            )~
            \invCoassQ[2]_1 {\coassQ[1]_1}\sweedler{1} \elDelta_1 
            \dotuline{S(\invCoassQ[1]_3)}
            S(\coassQ[1]_3 \dotuline{\coassQ[2]_3}) 
        \notag \\ 
        &\otimes
            \invCoassQ[2]_2 {\coassQ[1]_1}\sweedler{2}
            \elDelta_2 . f . 
            \dashuline{\Dt\inv_1} \dotuline{S(\invCoassQ[1]_2)}
            S({\invCoassQ[2]_3}\sweedler{2}
            {\coassQ[1]_2}\sweedler{2}
            \dotuline{\coassQ[2]_2} )
        \notag \\[1em]
        =\quad &
        \modulus\inv({\invCoassQ[2]_3}\sweedler{1}
        {\coassQ[1]_2}\sweedler{1}) 
            \invCoassQ[2]_1 {\coassQ[1]_1}\sweedler{1} \elDelta_1 
            S(\coassQ[1]_3) 
        \otimes
            \invCoassQ[2]_2 {\coassQ[1]_1}\sweedler{2}
            \elDelta_2 . f . \elX
            S({\invCoassQ[2]_3}\sweedler{2}
            {\coassQ[1]_2}\sweedler{2})
        \notag \\[1em]
        \oversetEq[\eqref{eq:delta-qHopf}] \quad &
        \modulus\inv({\invCoassQ[2]_3}\sweedler{1}
        {\coassQ[1]_2}\sweedler{1}) 
            \invCoassQ[2]_1 {\coassQ[1]_1}\sweedler{1} 
            {\invCoassQ[1]_1}\sweedler{1} \coassQ[2]_1 \betaQ S(\invCoassQ[1]_3)
            S(\coassQ[1]_3) 
        \notag \\
        &\otimes
            \invCoassQ[2]_2 {\coassQ[1]_1}\sweedler{2}
            {\invCoassQ[1]_1}\sweedler{2} \coassQ[2]_2 \betaQ 
            \dashuline{
                S(\invCoassQ[1]_2 \coassQ[2]_3)
            }
            . f . \dashuline{~\elX~} 
            S({\invCoassQ[2]_3}\sweedler{2} {\coassQ[1]_2}\sweedler{2})
        \notag \\[1em]
        =\quad &
        \modulus\inv(
        \dashuline{
            {\invCoassQ[2]_3}\sweedler{1}
            {\coassQ[1]_2}\sweedler{1}
            (\invCoassQ[1]_2 \coassQ[2]_3)\sweedler{1} )
        }
        \dashuline{
            \invCoassQ[2]_1 {\coassQ[1]_1}\sweedler{1} 
            {\invCoassQ[1]_1}\sweedler{1} \coassQ[2]_1 
        }
        \betaQ 
        \dashuline{
            S(\invCoassQ[1]_3) S(\coassQ[1]_3) 
        }
        \notag \\
        &\otimes
        \dashuline{
            \invCoassQ[2]_2 {\coassQ[1]_1}\sweedler{2}
            {\invCoassQ[1]_1}\sweedler{2} \coassQ[2]_2 
        }
        \betaQ . f .  \elX
        \dashuline{
            S((\invCoassQ[1]_2 \coassQ[2]_3)\sweedler{2})
            S({\invCoassQ[2]_3}\sweedler{2} {\coassQ[1]_2}\sweedler{2})
        }
        \notag \\[1em]
        &= \betaQ \otimes \betaQ . f . \elX
        \ ,
    \end{align}
    where in $(\star)$ we used that $f\in \gammaSSymR$.
\end{proof}
\medskip

Now we have all the necessary ingredients and can prove our main theorem.

\begin{proof}[{Proof of \texorpdfstring{\Cref{thm:MainThmSectionStatement}}{Theorem} (1)}]
    By \Cref{prop:X2=C2}, for each $\cointCat\in \cat(\tensUnit,
    \hopfMonad[2][\distInvObj])$ there is a unique $\coint \in \gammaSSym$ such that
    $\cointCat = \isoXC_2(\coint) = \betaQ.\coint.\elX$.
    Assume first that $\coint$ is a right cointegral.
    Then
    \begin{align}
        (\varphi_{\dualL{H}} \circ \rho^l)(\cointCat)
        &\oversetEq[(*)]
    \modulus(\invCoassQ[1]_3) \varphi_{\dualL{H}}( \Delta(\betaQ) .
        (\invCoassQ[1]_1 \otimes \coint . \invCoassQ[1]_2) . \Delta(\elX))
        ~
        \oversetEq[(**)]
        ~
        \betaQ \otimes \cointCat
    \end{align}
    shows that $\cointCat$ is a right monadic cointegral, using the equivalent
    characterisation~\eqref{eq:MainThmProof-final-version-of-monCoint-Eq}.
    Here ($*$) uses that $\coint$ is a right cointegral, and ($**$) uses
    \Cref{lem:weird_identity}.

    Conversely, assume that $\cointCat$ is a right monadic cointegral.
    Note that for any $f\in \gammaSSym$ we have
    \begin{align}
        \label{eq:needed_as_first_step}
        f
        =
        \qL_1 \betaQ S(\qL_2) . f 
        =
        \qL_1 \betaQ . f . (\modulus \rightharpoonup S(\qL_2)) ,
    \end{align}
    where the first step is the zig-zag axiom \eqref{eq:zigzag-qHopf}, and the second step
    uses that $f\in \gammaSSymR$.
    Then
    \begin{align}
        \rho^l(\coint) 
    ~&\oversetEq[\eqref{eq:needed_as_first_step}]~
        \rho^l( \qL_1 \betaQ . \coint . \elX \elX\inv 
        (\modulus \rightharpoonup S(\qL_2)) )
        \notag \\
        &\oversetEq[(1)] 
        \Delta(\qL_1) \rho^l( \betaQ . \coint . \elX ) 
        \Delta(\elX\inv (\modulus \rightharpoonup S(\qL_2)) ))
        \notag \\
        &\oversetEq[(2)] 
        \Delta(\qL_1) 
        \varphi_{\dualL{H}}\inv(\betaQ \otimes \betaQ . \coint . \elX )
        \Delta(\elX\inv (\modulus \rightharpoonup S(\qL_2)) ))
        \notag \\
        &\oversetEq[(3)]
        \modulus(\invCoassQ[1]_3)~
        \Delta(\qL_1) . 
        \big(
            \betaQ\sweedler{1} \invCoassQ[1]_1 \elX\sweedler{1} 
            \otimes
            \betaQ\sweedler{2} . \coint . \invCoassQ[1]_2 \elX\sweedler{2} 
        \big) .
        \Delta(\elX\inv (\modulus \rightharpoonup S(\qL_2)) )
        \notag \\
        &= 
        \modulus(\invCoassQ[1]_3)~
        \big(
            (\qL_1\betaQ)\sweedler{1} \invCoassQ[1]_1 
            \otimes
            (\qL_1\betaQ)\sweedler{2} . \coint . \invCoassQ[1]_2 
        \big) .
        \Delta(\modulus \rightharpoonup S(\qL_2))
        \notag \\
        &\oversetEq[(4)] 
        \modulus((\qL_1\betaQ)\sweedler{2,2} \invCoassQ[1]_3)~
        \big(
            (\qL_1\betaQ)\sweedler{1} \invCoassQ[1]_1 
            \otimes
            \coint . (\qL_1\betaQ)\sweedler{2,1} \invCoassQ[1]_2 
        \big) .
        \Delta(\modulus \rightharpoonup S(\qL_2))
        \notag \\
        &\oversetEq[(5)] 
        \modulus( \invCoassQ[1]_3 )~
        \big(
            \invCoassQ[1]_1 
            \otimes
            \coint . \invCoassQ[1]_2 
        \big) .
        \Delta(\modulus \rightharpoonup (\qL_1\betaQ) )
        \Delta(\modulus \rightharpoonup S(\qL_2))
        \notag \\
        &=
        \modulus( \invCoassQ[1]_3 )~
        \invCoassQ[1]_1 
        \otimes
        \coint . \invCoassQ[1]_2 
    \end{align}
    shows that $\coint \in \gammaSSym$ is a right cointegral in the sense of
    \cite{BC2-2011, HN-integrals}.
    The step labelled (1) uses that $\rho^l$ is a bimodule morphism, (2) is the fact that
    $\cointCat = \betaQ . \coint . \elX$ is a monadic cointegral, (3) follows from
    \Cref{lem:weird_identity}, (4) uses that $\coint\in \gammaSSym$, and (5) is an
    application of quasi-coassociativity.
\end{proof}

\subsection{Proof of \Cref{thm:MainThmSectionStatement} (2)}
\label{app:ProofMainThPart2}

Similarly to right cointegrals, left cointegrals for $H$ are automatically contained
in the space
\begin{align}
    \gammaSSymL 
    &= 
    \{
        f\in H^* \mid
        f \leftharpoonup S\inv(a) = S(a \leftharpoonup \modulus) \rightharpoonup f
    \}
    \notag \\
    &= 
    \{
        f\in \dualR{H}\in \hmodM[H][H] \mid
        a.f = f.(a \leftharpoonup \modulus)
    \}, 
\end{align}
see \eqref{eq:symmetryPropertiesCoint}, and analogously to \Cref{prop:X2=C2}
one can show that 
\begin{align}
    \isoXC_3 = \isoXC_2\cop :
    \gammaSSymL \to \cat(\tensUnit, \hopfMonad[3][\dualL{\modulus}]),
    \quad \isoXC_3(f) = S\inv(\beta).f.\hat{\elX}
\end{align}
is an isomorphism.
Note that here the dot denotes the action on the right dual of the regular
$H\otimes H\op$-module.
The Hopf monads $\hopfMonad[2]$ and $\hopfMonad[3]$ are canonically isomorphic
via $\kappa_{2,3}: \hopfMonad[2] \Rightarrow \hopfMonad[3]$, 
see \eqref{eq:A2isoA3} and \Cref{prop:IsoBetweenHopfMonads-qHopf}.
This allows us to transport right monadic cointegrals to left monadic cointegrals.
Thus, upon showing that
\begin{equation}\label{eq:left-right-integral-transport-maps}
    \begin{tikzcd}[row sep=large, column sep=large]
        \spaceRightCoint
        \ar[r, hook]
        \ar[d,"(*)",swap]
        &\gammaSSymR
        \ar[r,"{\isoXC_2}"]
        &
        \cat(\tensUnit, \hopfMonad[2][\dualL{\modulus}])
        \ar[d,"{(\kappa_{2,3})_{\dualL{\modulus}} ~ \circ ~ \placeholder}"]
        &
        \ar[l,hook']
        \spaceRightMonCoint
        \ar[d,"{(\kappa_{2,3})_\dualL{\modulus} ~ \circ ~ \placeholder}"]
        \\
        \spaceLeftCoint
        \ar[r,hook]
        &\gammaSSymL
        \ar[r,"{\isoXC_3}",swap]
        &
        \cat(\tensUnit, \hopfMonad[3][\dualL{\modulus}])
        &
        \ar[l,hook']
        \spaceLeftMonCoint
    \end{tikzcd}
\end{equation}
commutes, we know that $\isoXC_3$ maps left cointegrals to left monadic cointegrals.
Here $(*)$ maps the right cointegral $\coint^r$ to
    \footnote{
    The zig-zag axiom \eqref{eq:zigzag-qHopf} implies that both $\modulus(\alphaQ)$ and
    $\modulus(\betaQ)$ are invertible in $\field$.
    Therefore, the prefactor in \eqref{eq:mapFromRightToLeftCoint-MainThm} is
    well-defined.
    }
\begin{align}
    \coint^l = 
    \modulus(\alphaQ S(\betaQ))\inv \cdot 
    (\coint^r \circ S \leftharpoonup (\elu\cop)\inv).
    \label{eq:mapFromRightToLeftCoint-MainThm}
\end{align}
By \cite[Prop.~4.3]{BC2-2011} this is a left cointegral, cf.\
\eqref{eq:relating_left_and_right_cointegrals}.

The right hand square in~\eqref{eq:left-right-integral-transport-maps} commutes by
construction.
Using the explicit formula \eqref{eq:HopfMonadIsos-qHopf} for $\kappa_{2,3}$,
one finds that the upper path of the left hand square is 
\begin{align}\label{eq:mainthm_Left_version_upper_path}
    \coint^r 
    &\mapsto 
    \modulus\inv(\coassQ[1]_2) 
    \langle 
        \coint^r \mid 
        S(\betaQ) S(\placeholder \coassQ[1]_1) \coassQ[1]_3 S\inv(\elX) 
    \rangle
    \notag  \\
    &=
    \modulus\inv(\coassQ[1]_2) 
    \langle 
        \coint^r \circ S \mid 
        S^{-2}(\elX) S\inv(\coassQ[1]_3) \placeholder \coassQ[1]_1 \betaQ
    \rangle
    \notag  \\
    &=
    \modulus(\alphaQ S(\betaQ)) \modulus\inv(\coassQ[1]_2) 
    \langle 
        \coint^l \mid 
        \elu\cop S^{-2}(\elX) S\inv(\coassQ[1]_3) \placeholder \coassQ[1]_1 \betaQ
    \rangle
    \notag  \\
    &\oversetEq[(1)]
    \modulus(\alphaQ S(\betaQ)) \modulus\inv(\coassQ[1]_2 \pL_1)
    \langle 
        \coint^l \mid 
        S\inv(\coassQ[1]_3 \pL_2) \placeholder \coassQ[1]_1 \betaQ
    \rangle
    \notag  \\
    &\oversetEq[(2)]
    \modulus(\alphaQ S(\betaQ)) \modulus\inv(\coassQ[1]_2 \pL_1)
    \modulus((\coassQ[1]_3 \pL_2)\sweedler{1})
    \langle 
        \coint^l \mid 
        \placeholder \coassQ[1]_1 \betaQ S((\coassQ[1]_3 \pL_2)\sweedler{2})
    \rangle,
\end{align}
the step marked (1) follows directly from the definition of $\elu$ and $\elX$,
    see \eqref{eq:elu-right-left-coint-defined} resp.\
    \Cref{thm:MainThmSectionStatement}, 
and step (2) uses $\coint^l\in \gammaSSymL$.

The lower path of the left square of~\eqref{eq:left-right-integral-transport-maps}
evaluates to
\begin{align}\label{eq:mainthm_Left_version_lower_path}
    \coint^r 
    &\mapsto 
    \langle 
        \coint^l \mid
        S^{-2}(\betaQ) \placeholder S(\hat{\elX})
    \rangle
    \notag \\
    &\oversetEq[\eqref{eq:symmetryPropertiesCoint}]~
    \langle 
        \coint^l \mid
        \placeholder S((S\inv(\betaQ) \leftharpoonup \modulus) \hat{\elX})
    \rangle
    \notag \\
    &\oversetEq[(*)]
    \modulus\inv( \betaQ\sweedler{2} \Dt\inv_2 )
    \langle 
        \coint^l \mid
        \placeholder  \betaQ\sweedler{1} \Dt\inv_1 
    \rangle
    \notag \\
    &\oversetEq[\eqref{eq:coproduct_beta}]~
    \modulus\inv(\coassQ[1]_2) \modulus((\coassQ[1]_3)\sweedler{1} \pL_1)
    \langle 
        \coint^l \mid
        \placeholder  \coassQ[1]_1 \betaQ S((\coassQ[1]_3)\sweedler{2} \pL_2)
    \rangle \ ,
\end{align}
where ($*$) uses the definition of $\hat{\elX}$ as in \Cref{thm:MainThmSectionStatement}.

We have
\begin{align}
    \modulus \big(\alphaQ S(\betaQ) S(\pL_1) {\pL_2}\sweedler{1} \big)
    {\pL_2}\sweedler{2}
    &\oversetEq[\eqref{eq:q,p^L}]~
    \modulus\big(\alphaQ S(\betaQ) \invCoassQ[1]_1 \betaQ 
    S(\invCoassQ[1]_2) {\invCoassQ[1]_3}\sweedler{1}\big)
    {\invCoassQ[1]_3}\sweedler{2}
    \notag \\
    &\oversetEq[\eqref{eq:pentagon-qHopf},(*)]\quad
    \modulus\big(
        \alphaQ S(\betaQ)
        \invCoassQ[1]_1 \betaQ S(\invCoassQ[2]_1 \invCoassQ[1]_2)
        \invCoassQ[2]_2 \invCoassQ[1]_3
    \big)
    \invCoassQ[2]_3
    \notag \\
    &\oversetEq[\eqref{eq:zigzag-qHopf}]~
    \modulus\big(
        S(\invCoassQ[2]_1 \betaQ)
        \invCoassQ[2]_2 
    \big)
    \invCoassQ[2]_3
    \notag \\
    &\oversetEq[(**)]\quad
    \modulus(\pL_1) \pL_2
\end{align}
where $(*)$ uses $(\modulus \otimes \modulus\inv)\circ \Delta = \counit$, and $(**)$
uses $\modulus \circ S = \modulus \circ S\inv$ and \eqref{eq:q,p^L}.
Therefore the two expressions \eqref{eq:mainthm_Left_version_upper_path} and
\eqref{eq:mainthm_Left_version_lower_path} are equal.

\subsection{Proof of the pivotal case}
The proof of \Cref{thm:MainThmPivotalCase} is similar to the proof of the second
part above.

First, define the spaces
\begin{align}
    \gammaSSymRS &= 
    \{ 
        f\in H^* \mid f(ab) = f((b \leftharpoonup \modulus) a) 
    \} \ , 
    \notag \\
    \gammaSSymLS &= 
    \{ 
        f\in H^* \mid f(ab) = f((\modulus \rightharpoonup b) a) 
    \} \ .
    \label{eq:Def_gammaSSym_pivotal}
\end{align}
By \eqref{eq:symmetryPropertiesSymmetrizedCoint} we have $\symRightCoint\in \gammaSSymRS$
and $\symLeftCoint\in \gammaSSymLS$, and similarly to \Cref{prop:X2=C2} one may
show that 
\begin{align}
    \isoXC_1: \gSSSymbol_1 &\to \cat( \tensUnit, \hopfMonad[1][\dualL{\modulus}] )
    \ ,\quad 
    f \mapsto \langle f \mid S\inv(\betaQ) \placeholder S(\elTh) \rangle \ ,
    \notag \\ \label{eq:isoXC-pivotal}
    \isoXC_4: \gSSSymbol_4 &\to \cat( \tensUnit, \hopfMonad[4][\dualL{\modulus}] )
    \ ,\quad 
    f \mapsto \langle f \mid \betaQ \placeholder S\inv(\hat{\elTh}) \rangle
	\ ,
\end{align}
with $\elTh = (\modulus\inv \otimes S\inv) (\pL)$ and $\hat{\elTh} = \elTh\cop$, are
linear isomorphisms.

Then a simple computation shows that the diagram 
\begin{equation}
    \begin{tikzcd}[row sep=large, column sep=large]
        \spaceRightCoint
        \ar[r, hook]
        \ar[d,"(*)",swap]
        &\gammaSSymR
        \ar[r,"{\isoXC_2}"]
        &
        \cat(\tensUnit, \hopfMonad[2][\dualL{\modulus}])
        \ar[d,"{\sim}"]
        &
        \ar[l,hook']
        \spaceRightMonCoint
        \ar[d,"{\sim}"]
        \\
        \spaceRightSymCoint
        \ar[r,hook]
        &\gammaSSymRS
        \ar[r,"{\isoXC_1}",swap]
        &
        \cat(\tensUnit, \hopfMonad[1][\dualL{\modulus}])
        &
        \ar[l,hook']
        \spaceRightSymMonCoint
    \end{tikzcd}
\end{equation}
commutes, with ($*$) sending a right cointegral $\coint^r$ to the right symmetrised
cointegral $\coint^r \leftharpoonup \elu \pivotQ$, and $\sim$ is induced by the
isomorphism of Hopf monads from \Cref{prop:iso_as_HC_pivotal}.

Indeed, a right cointegral $\coint^r$ gets mapped to the right
$\distInvObj$-symmetrised monadic cointegral
\begin{align}
    \coint^{r,\distInvObj-\textup{sym}}
    = \langle \coint^r \mid S(\betaQ) \pivotQ \placeholder S\inv(\elX) \rangle
\end{align}
by the upper path, and to
\begin{align}
    (\coint^{r,\distInvObj-\textup{sym}})'
    = \langle \coint^r \mid \elu \pivotQ S\inv(\betaQ) \placeholder S(\elTh) \rangle
\end{align}
by the lower path.

The upper path, evaluated on $S\inv(h)$, $h\in H$, yields
\begin{align}
    \coint^{r,\distInvObj-\textup{sym}} (S\inv(h))
    &=~ \langle \coint^r \mid S(\betaQ) \pivotQ S\inv(\elX h) \rangle
    \notag \\
    &=~ \langle \coint^r \mid S(\elX h \betaQ) \pivotQ \rangle
    \notag \\
    &\oversetEq[\eqref{eq:relating_left_and_right_cointegrals}] ~~ 
    \langle \coint^l \mid \elu\cop \pivotQ\inv \elX h \betaQ \rangle
    \ .
\end{align}
Evaluating the lower path on $S\inv(h)$, $h\in H$, we get
\begin{samepage}
\begin{align}
    (\coint^{r,\distInvObj-\textup{sym}})'(S\inv(h))
    &=~ \langle \coint^r \mid \elu \pivotQ S\inv(h \betaQ) S(\elTh) \rangle
    \notag \\
    &\oversetEq[\eqref{eq:relating_left_and_right_cointegrals}] ~~ 
    \langle \coint^l \circ S\inv \mid S(h \betaQ) \pivotQ S(\elTh) \rangle
    \notag \\
    &\oversetEq ~~ 
    \langle \coint^l \mid \elTh \pivotQ\inv h \betaQ \rangle
    \ .
\end{align}
\end{samepage}
The claim then follows from
\begin{align}
    \elu\cop \pivotQ\inv \elX \pivotQ 
    ~ &\oversetEq[\eqref{eq:elu-right-left-coint-defined}] ~~
    \modulus(\elV\cop_1 \Dt\inv_2) S^{-2}(\elV\cop_2) \pivotQ\inv \Dt\inv_1 \pivotQ
    \notag \\ 
    &= \modulus(\elV\cop_1 \Dt\inv_2) \pivotQ\inv \elV\cop_2 \Dt\inv_1 \pivotQ
    \notag \\ 
    &\oversetEq[\eqref{eq:Definition_UV}] ~~
    \modulus \big(S(\pL_1) \tilde{\Dt}_2 \Dt\inv_2 \big) 
    \pivotQ\inv S(\pL_2) \tilde{\Dt}_1 \Dt\inv_1 \pivotQ
    \notag \\ 
    &= \modulus\inv( \pL_1 ) S\inv(\pL_2)
    \notag \\
    &= \elTh
\end{align}

A similar diagram involving left cointegrals and their symmetrised version then finishes
the proof of the theorem.
\hfill $\qed$

\bigskip

\subsection*{Acknowledgements}

We thank Alain Brugui\`eres, Jonas Haferkamp, Vincent Koppen and Christoph Schweigert for
helpful discussions.
JB is supported by the Research Training Group RTG\,1670 of the Deutsche
Forschungsgemeinschaft.
AMG is supported by CNRS, and thanks Humboldt Foundation and ANR grant JCJC
ANR-18-CE40-0001 for a partial financial support. 
IR is partially supported by the RTG\,1670 and the Cluster of Excellence EXC 2121.


\definecolor{antiquebrass}{rgb}{0.7, 0.38, 0.46}

\newcommand{\arxiv}[2]
    {[arXiv: \href{http://arXiv.org/abs/#1}{ \color{antiquebrass}#1 [#2]}]}
\newcommand{\doi}[2]{\href{http://dx.doi.org/#1}{#2}}

\end{document}